\documentclass[10pt]{article} 
\usepackage{amsfonts,amsmath,latexsym}
\usepackage[T1]{fontenc}
\usepackage{hyperref}
\usepackage[dvips]{graphicx}
\usepackage{epsfig}
\usepackage{enumerate}
\usepackage{epsf}   
\usepackage{amsthm}
\usepackage{color}
\usepackage{amssymb}

\usepackage{caption}
\usepackage{subfigure}

\usepackage{fancyhdr} 

\usepackage{palatino}

\newtheorem{thm}{Theorem}[section]
\newtheorem{prop}[thm]{Proposition}
\newtheorem{lem}[thm]{Lemma}
\newtheorem{corro}[thm]{Corollary}
\newtheorem{defi}[thm]{Definition}
\newtheorem{rem}[thm]{Remark}
\newtheorem{example}[thm]{Example}
\newtheorem{ass}{Assumption}

\def\R{\mathbb R}
\def\N{\mathbb N}

\def\C{\mathbb C}
\def\E{\mathbb E}
\def\P{\mathbb P}

\def\shb{{\cal B}}
\def\shc{{\cal C}}

\def\shf{{\cal F}}

\def\shl{{\cal L}}
\def\shp{{\cal P}}
\def\shs{{\cal S}}
\def\shu{{\cal U}}

\author{
{\sc Anthony LECAVIL}
\thanks{ENSTA-ParisTech. Unit\'e de Math\'ematiques Appliqu\'ees (UMA).
  E-mail:{ \tt anthony.lecavil@ensta-paristech.fr}} 
 {\sc,}\ {\sc Nadia OUDJANE}
\thanks{EDF R\&D,   and FiME (Laboratoire de Finance des March\'es de l'Energie
(Dauphine, CREST,  EDF R\&D) www.fime-lab.org).
E-mail:{\tt  
nadia.oudjane@edf.fr}}
\ {\sc and}\ {\sc Francesco RUSSO} 
\thanks{ENSTA-ParisTech. Unit\'e de Math\'ematiques Appliqu\'ees (UMA). 
  E-mail:{\tt  francesco.russo@ensta-paristech.fr}}.
}

\date{April 11th 2015}
\title{Probabilistic representation of a class of non conservative nonlinear Partial Differential Equations}

\newcommand{\MBFigure}[6]{
$\left. \right.$ \\
\refstepcounter{figure}
\addcontentsline{lof}{figure}{\numberline{\thefigure}{\ignorespaces #5}}
\begin{center}
\begin{minipage}{#1cm}
\centerline{\includegraphics[width=#2cm,angle=#3]{#4}}
\begin{center}
\upshape{F\textsc{ig} \normal
\end{center}
size{\thefigure}. $-$} #5
\end{center}
\label{#6}
\end{minipage}
\end{center}
$\left. \right.$ \\}

\begin{document}
\maketitle 
\begin{abstract} We introduce a new class of nonlinear Stochastic Differential Equations
in the sense of McKean, related to non conservative nonlinear Partial Differential equations (PDEs).
We discuss existence and uniqueness pathwise and in law under various assumptions. 
 We propose an original interacting particle system 
 for which we discuss the propagation of chaos. 
To this system, we associate a random function which is proved to converge
to a solution of a regularized version of PDE.
\end{abstract}

\medskip\noindent  {\bf 2010  AMS-classification}:
 65C05; 65C35; 60H10;60H30; 60J60; 58J35 

\section{Introduction}
\label{SIntroduction}


Probabilistic representations of nonlinear Partial Differential Equations (PDEs) are interesting in several aspects. From a theoretical point of view, such representations allow for probabilistic tools to study the analytic properties of the equation (existence and/or uniqueness of a solution, regularity,\ldots).
They also have their own interest typically when they provide a microscopic interpretation of physical phenomena macroscopically drawn by a nonlinear PDE. 
Similarly, stochastic control problems are a way of interpreting non-linear PDEs through Hamilton-Jacobi-Bellman equation that have their own theoretical and practical interests~(see~\cite{FlemingSoner}).  
Besides, from a numerical point of view, such representations allow for new approximation schemes potentially 
less sensitive to the dimension of the state space thanks to their probabilistic nature  involving Monte Carlo 
based methods. 
 
The present paper focuses on a specific forward approach relying on nonlinear 
SDEs in the sense of McKean~\cite{McKean}. The coefficients of that SDE instead of depending only on 
the position of the solution  $Y$, also depend on the law of the process,
 in a non-anticipating way.   
One historical contribution
on the subject was performed by \cite{sznitman}, which concentrated on non-linearities  on the drift coefficients.\\
%
Let us consider $d,p \in \N^{\star}$. Let 
$\Phi:[0,T] \times \R^d \times \R \rightarrow \R^{d \times p}$,
 $g:[0,T] \times \R^d \times \R \rightarrow \R^d$,
$\Lambda:[0,T] \times \R^d \times \R \rightarrow \R$,
 be Borel bounded  functions and $\zeta_0$ be a probability  on $\R^d$. 
When it is absolutely continuous we denote by  
$v_0$ its density so that $\zeta_0(dx) = v_0(x) dx$. 
We are motivated in 
 non-linear PDEs (in the sense of the distributions) of the  form
\begin{equation}
\label{epde}
\left \{
\begin{array}{l}
\partial_t v = \frac{1}{2} \displaystyle{\sum_{i,j=1}^d} \partial_{ij}^2 \left( (\Phi \Phi^t)_{i,j}(t,x,v) v \right) - div \left( g(t,x,v) v \right) +\Lambda(t,x,v) v\ , \quad \textrm{for any}\  t\in [0,T]\ ,\\
v(0,dx) = \zeta_0(dx), 
\end{array}
\right .
\end{equation}
where $v:]0,T] \times \R^d \rightarrow \R$ is the unknown function and the second equation
means that $v(t,x) dx$ converges weakly to $\zeta_0(dx)$ when $t \rightarrow 0$.
When $\Lambda = 0$, PDEs of the type \eqref{epde} are generalizations of the Fokker-Planck equation
and they are  often denominated 
in the literature as McKean type equations.
Their solutions are  probability measures dynamics which often describe 
  the macroscopic distribution law  of  a {\it microscopic particle}
which behaves in a diffusive way. For that reason, those time evolution PDEs
are  {\it conservative} 
in the sense that their solutions  $v(t, \cdot)$ verify the property $\int_{\R^d} v(t,x)dx$ to be constant in $t$,
generally equal to $1$, which is the mass of a probability measure.
More precisely, often  the solution $v$ of \eqref{epde} is  
associated with a couple $(Y,v)$, 
where $Y$ is a stochastic process and $v$ a real valued function defined on $[0,T] \times \R^d$ such that 
\begin{equation}
\label{ENLSDE}
\left \{
\begin{array}{l}
Y_t = Y_0 + \int_0^t \Phi(s,Y_s,v(s,Y_s)) dW_s + \int_0^t g(s,Y_s,v(s,Y_s)) ds \ ,\quad \textrm{with}\ Y_0\sim \zeta_0\\
v(t,\cdot) \text{ is the density of the law of } Y_t \ ,
\end{array}
\right .
\end{equation}
and $(W_t)_{t\geq 0}$ is a $p$-dimensional Brownian motion on a filtered probability
 space $(\Omega, \mathcal{F},\mathcal{F}_t,\mathbb{P})$.
 A major technical difficulty arising when studying the existence and uniqueness for solutions of~\eqref{ENLSDE} 
is due to the point dependence of the SDE coefficients w.r.t. the probability density $v$. 
In the literature \eqref{ENLSDE} was  generally faced by analytic methods.
A lot of work was performed in the case of smooth Lipschitz coefficients with regular initial condition,
 see for instance Proposition 1.3. of \cite{JourMeleard}. The authors  also assumed 
to be in the non-degenerate case, with $\Phi \Phi^t$ being an invertible matrix
  and some parabolicity condition. 
An interesting 
 earlier work concerns  the case $\Phi(t,x,u) = u ^k \; (k \geq 1), g = 0$ see \cite{Ben_Vallois}.
In dimension $d =1$ with $g = 0$ and $\Phi$ being bounded measurable, probabilistic representations 
of \eqref{epde} via solutions of \eqref{ENLSDE} were obtained in \cite{BRR1, BRR2}.
\cite{BCR3} extends partially those results to the multidimensional case.
Finally \cite{BCR2} treated the case of fast diffusion.
All those techniques were based on the resolution of the corresponding non-linear Fokker-Planck equation,
so through an analytic tool. 
 \\
In the present article, we are however especially interested in \eqref{epde}, in the case where 
$\Lambda$ does not vanish. In that context, the natural generalization of \eqref{ENLSDE} is
given by
\begin{equation}
\label{ENLSDELambda}
\left \{
\begin{array}{l}
Y_t = Y_0 + \int_0^t \Phi(s,Y_s,v(s,Y_s)) dW_s + \int_0^t g(s,Y_s,v(s,Y_s)) ds\ ,\quad\textrm{with}\quad Y_0\sim  \zeta_0 \ , \\
v(t,\cdot):=\frac{d\nu_t}{dx} \quad \textrm{such that for any bounded continuous test function}\  \varphi \in \mathcal{C}_b(\R^d, \R)\\ 
\nu_t(\varphi):=\E\left[ \varphi(Y_t) \, \exp \left \{\int_0^t\Lambda \big (s,Y_s,v(s,Y_s)\big )ds\right \} \right]\ ,\quad\textrm{for any}\ t\in [0,T]\ .
\end{array}
\right.
\end{equation}
 The aim of the paper  is precisely to extend the McKean probabilistic representation to a
 large class of nonconservative PDEs. The first step in that direction was done by \cite{BRR3} where the 
Fokker-Planck equation is a stochastic PDE with multiplicative noise. Even though that equation is pathwise not conservative, 
the  expectation of the mass was constant and equal to $1$. Here again, these developments relied on analytic tools. \\
%
%
To avoid the technical difficulty 
due to the pointwise dependence of the SDE coefficients w.r.t. the function $v$, 
this paper focuses on the following regularized version of~\eqref{ENLSDELambda}:
\begin{equation}
\label{eq:NSDE}
\left\{
\begin{array}{l}
Y_t=Y_0+\int_0^t \Phi(s,Y_s,u(s,Y_s)) dW_s+\int_0^tg(s,Y_s, u(s,Y_s))ds\ ,\quad\textrm{with}\quad Y_0\sim  \zeta_0 \ ,\\
u(t,y):=\E[K(y-Y_t) \, \exp \left \{\int_0^t\Lambda \big (s,Y_s,u(s,Y_s)\big )ds\right \}]\ ,\quad\textrm{for any}\ t\in [0,T]\ ,
\end{array}
\right .
\end{equation}
where  $K: \R^d \rightarrow \R$ is a smooth mollifier in $\R^d$. 
When $K = \delta_0$ \eqref{eq:NSDE} reduces, at least formally to \eqref{ENLSDELambda}.
An easy application of It\^o's formula (see e.g. Proposition \ref{PIDE}) shows that if there is a solution 
$(Y,u)$ of~\eqref{eq:NSDE}, $u$ is related to the solution (in the distributional sense) of the following partial integro-differential equation  (PIDE) 
\begin{equation}
\label{epide}
\left \{
\begin{array}{l}
\partial_t \bar v = \frac{1}{2} \displaystyle{\sum_{i,j=1}^d} \partial_{ij}^2 \left( (\Phi \Phi^t)_{i,j}(t,x,K\ast \bar v) \bar v \right) - div \left( g(t,x,K\ast \bar v) \bar v \right) +\Lambda(t,x,K\ast \bar v) \bar v \\
\bar v(0,x) = v_0 \ , 
\end{array}
\right .
\end{equation}
by the relation $u=K\ast \bar v:=\int_{\R^d}K(\cdot-y)\bar v(y)dy$. 
Setting $K^\varepsilon(x) = \frac{1}{\varepsilon^d} K\left( \frac{\cdot}{\varepsilon} \right)$
the generalized sequence $K^\varepsilon$ is weakly convergent to 
the Dirac measure at zero.
Now, consider the couple $(Y^\varepsilon, u^\varepsilon)$  solving  \eqref{eq:NSDE} 
replacing $K$ with $K^\varepsilon$. Ideally, $u^\varepsilon$ should converge to a solution of 
the {\it limit partial differential equation}~\eqref{epde}. 
In the case $\Lambda\equiv 0$, with smooth $\Phi, g$ and initial condition with other technical conditions,
that convergence was established in Lemma 2.6 of \cite{JourMeleard}. 
In our extended setting, again, no mathematical argument is for the moment available  
but this limiting behavior is  explored empirically 
by numerical simulations 
 in Section \ref{S8}.
Always in the case $\Lambda = 0$  with $g=0$, but with $\Phi$ only measurable,
the qualitative behavior of the solution for large time was numerically simulated
in \cite{BCR1, BCR3} respectively for the one-dimensional and multi-dimensional case.

Besides the theoretical aspects related to the well-posedness of \eqref{epde} (and \eqref{eq:NSDE}),
our main motivation is to simulate numerically efficiently their solutions.
With this numerical objective, several types of probabilistic representations have been developed in the literature, each one having 
specific features regarding the implied approximation schemes.\\
One method which has been largely investigated for
approximating solutions of time evolutionary PDEs is the method
of forward-backward SDEs.  FBSDEs were initially developed in~\cite{pardoux}, see
also \cite{pardouxgeilo} for a survey and \cite{rascanu} for a recent monograph on the subject.  
The idea is to express the PDE solution $v(t,\cdot)$ at time $t$  as the expectation of
a functional of the 
so called forward diffusion process $X$.
  Numerically, many judicious schemes have been proposed~\cite{Zhang,BouchardTouzi,GobetWarin}. 
But they all rely on computing recursively conditional expectation functions which is known to be a difficult
 task in high  dimension. Besides, the FBSDE approach is \textit{blind} in the sense that the forward process is
 not ensured to explore the most relevant space regions to approximate efficiently the backward process of interest.
 On the theoretical side, the FBSDE representation of fully nonlinear PDEs still requires complex developments and is 
the subject of active research (see for instance~\cite{cheridito}). \\
Branching diffusion processes are another way of providing a probabilistic representation of semi-linear PDEs
 involving a specific form of non-linearity on the zero order term. We refer to~\cite{Dynkin} for the case of 
superprocesses. This type of approach has been recently extended in~\cite{labordere,LabordereTouziTan} to a
 more general class of non-linearities on the zero order term, with the so-called \textit{marked branching process}. 
One of the main advantage of this approach compared to BSDEs is that it reduces in a forward algorithm without 
any regression computation. \\
One numerical intuition motivating our interest in (possibly non-conservative) 
 PDEs representation of McKean type is the possibility to take 
advantage of the  forward feature of this representation to bypass the dimension problem by localizing the particles
 precisely in the \textit{regions of interest}, although this point will not be developed in the present paper.
 Another benefit of this approach is that it is potentially able to represent fully nonlinear PDEs. 

 In this paper, for the considered class of nonconservative PDE,  besides various theoretical results
 of existence and uniqueness, we establish the so called {\it propagation of chaos} of 
 an associated interacting particle system and we develop
  a numerical scheme based on it. 
 The convergence of this algorithm is proved by propagation of chaos  and through the
  control of the time discretization error.  Finally, some numerical simulations illustrate 
 the practical interest of this new algorithm. 

The main contributions of this paper are twofold.
\begin{enumerate}
	\item We provide a refined analysis of existence and/or uniqueness of a solution to~\eqref{eq:NSDE} under a variety of regularity assumptions on the coefficients $\Phi$, $g$ and $\Lambda$. 
	This analysis faces two main difficulties.
	In the first equation composing the 
system~\eqref{eq:NSDE} the coefficients  depend on the density $u$, itself depending 
on $Y$. 
This is the standard
 situation already appearing in the context of classical McKean type equations  
 when  $u(t,\cdot)$ is characterized by the law of $Y_t, t \ge 0$.
 This situation can be recovered formally
here when the function $\Lambda \equiv 0$ and the mollifier $K=\delta_0$.  
In the second equation characterizing 
$u$ in~\eqref{eq:NSDE}, for a given process 
$Y\in \mathcal{C}^d := \mathcal{C}([0,T],\R^d)$, 
 $u$ also appears on the right-hand-side (r.h.s)
 via the \textit{weighting function} $\Lambda$. This additional difficulty is specific to our extended
 framework since
 in the standard McKean type equation, $\Lambda\equiv 0$ implies that $u(t,\cdot)$ is explicitely defined 
by the law  density of $Y_t$.  

In Section \ref{S3}, one shows existence and uniqueness of strong
 solutions of \eqref{eq:NSDE} when
$\Phi, g, \Lambda$ are Lipschitz.  This result is stated in Theorem \ref{prop:NSDE}.
The second equation of \eqref{eq:NSDE} can be rewritten as
\begin{equation} \label{NSDE3}
u(t,y) = 
\int_{\mathcal C^d} 
 K(y- \omega_t)
\exp \left \{\int_0^t\Lambda \big (s,\omega_s,u(s,\omega_s) \big ) ds\right \} dm(\omega) \ ,
\end{equation}
where $m = m_Y$ is the law of $Y$ on the canonical space
$\mathcal{C}^d$. 
In particular, given   a law $m$ on $\mathcal{C}^d$,
  using an original fixed point argument
 on stochastic processes $Z$ of the type $Z_t = u(t,X_t)$ where $X$ is the canonical process,
in Lemma \ref{lem:u}, we first study the existence of $u = u^m $ being solution
 of \eqref{NSDE3}.
A careful analysis in Lemma \ref{lem:uu'}
 is carried on the functional $(t,x,m) \mapsto u^m(t,x)$:
this associates  to each Borel probability measure $m $ on $\shc^d$,
  the solution of \eqref{eq:u}, which is  the second
line of \eqref{eq:NSDE}. 
  In particular that lemma describes carefully the dependence on all variables.
Then we consider the first equation of \eqref{eq:NSDE} using more 
standard arguments following Sznitman~\cite{sznitman}.  
 In Section \ref{S4}, we show strong existence 
of  \eqref{eq:NSDE} when $\Phi, g$ are Lipschitz and $\Lambda$ is only continuous, see Theorem \ref{thm_semi_weak}. Indeed,  uniqueness, however, does not hold if $\Lambda$ is only continuous, see Example \ref{E41}. 
In Section \ref{S5}, Theorem \ref{thm_weak} states   existence in law in all cases  when $\Phi, g, \Lambda$ are only continuous.
\item We introduce an interacting particle system associated to~\eqref{eq:NSDE} and prove that the propagation of chaos holds, under the assumptions of Section   \ref{S3}. This is  the object of  Section \ref{SChaos}, 
 see  Theorem \ref{TPC} and subsequent remarks. That theorem  also states the convergence of the solution  $u^m$ of \eqref{NSDE3},
when $m = S^N(\xi)$ is the  empirical measure of the particles to $u^{m_0}$, where $m_0$ is the law of the solution of \eqref{eq:NSDE},  in the $L^p, p = 2, +\infty$ mean error, in term of the number $N$ of particles.
We estimate in particular,   rates of   convergence
making use of  a refined analysis of the Lipschitz properties of $m \mapsto u^m$
 w.r.t.  various metrics on probability measures. 
 This crucial theorem is an obligatory step in a complete proof of the convergence of the stochastic particle algorithm:
it distinguishes clearly the control of the perturbation error induced by 
the approximation and the control of the propagation of this error through the particle dynamical system.
By our techniques, the proof of chaos propagation 
does not rely on the exchangeability property of the particles.
In Section \eqref{SPDE} we show that $u:= u^{m_0}$ verifies
 $u := K \ast \bar{v}$, 
where $\bar{v}$  solves the PIDE~\eqref{epide}.
 In Section~\ref{S8}, we propose an Euler discretization of the particle system
 dynamics and prove (Proposition \ref{prop:DiscretTime}) the convergence of this discrete time approximation to the continuous
 time interacting particle system by following the same lines of the propagation of chaos analysis, see Theorem \ref{TPC}.  
\end{enumerate}

The paper is organized as follows. After this introduction, 
we formulate the basic assumptions valid along the paper.
 Section \ref{S3}  is devoted to the existence and
uniqueness problem  when $\Phi, g, \Lambda$ are Lipschitz.
The propagation of chaos is discussed in Section \ref{SChaos}.
Sections \ref{S4} and \ref{S5} discuss the case when 
the coefficients are non-Lipschitz. Section \ref{SPDE}
establishes the link between \eqref{eq:NSDE} and
the integro partial-differential equation \eqref{epide}.
Finally, Section \ref{S8} provides  numerical simulations  illustrating the performances of the
 interacting particle system in approximating the PDE  \eqref{epde}, in a specific case where the solution
is explicitely known.

\section{Notations and assumptions}

\label{S2}

\setcounter{equation}{0}

Let us consider $\shc^d:=\mathcal{C}([0,T],\R^d)$ metrized by the supremum norm $\Vert \cdot \Vert_{\infty}$, equipped with its Borel $\sigma-$ field $\mathcal{B}(\shc^d) = \sigma(X_t,t \geq 0)$ (and $\mathcal{B}_t(\shc^d) := \sigma(X_u,0 \leq u \leq t)$ the canonical filtration) and endowed with the topology of uniform convergence, $X$ the canonical process on $\shc^d$ and $\mathcal{P}_k(\shc^d)$ the set of Borel probability measures on $\shc^d$ admitting a moment of order $k \in \N$. For $k=0$, $\mathcal{P}(\shc^d) := \mathcal{P}_0(\shc^d)$
 is naturally the Polish space (with respect to the weak convergence topology) of Borel probability measures on $\shc^d$ naturally equipped with its Borel $\sigma$-field $\mathcal{B}(\shp(\shc^d))$.
When $d=1$, we often omit it and we simply denote 
 $\mathcal{C} :=  \mathcal{C}^1$.\\
We recall that the Wasserstein distance of order $r$ 
 and respectively the \textit{modified Wasserstein distance of order $r$}
for $r \ge 2$,
  between $m$ and  $m'$ in $\mathcal{P}_r(\mathcal{C}^d) $, denoted by $W^r_{T}(m,m')$ (and resp. $\widetilde{W}^r_{T}(m,m')$) 
are such that
\begin{eqnarray}
\label{eq:Wasserstein}
(W^r_{t}(m,m')) ^{r} & := & \inf_{\mu\in \Pi(m,m')} \left \{ \int_{\mathcal{C}^d \times \mathcal{C}^d} \sup_{0 \leq s\leq t} \vert X_{s}(\omega) - 
X_{s}(\omega')\vert ^{r}  d\mu(\omega,\omega')   \right \}, \ t \in [0,T] \ , \\
\label{eq:WassersteinTilde}
 ( \widetilde{W}^r_{t}(m,m')) ^{r} & := & \inf_{\mu\in \widetilde{\Pi}(m,m')} \left \{ \int_{\mathcal{C}^d \times \mathcal{C}^d} \sup_{0 \leq s\leq t} \vert X_{s}(\omega) - X_{s}(\omega')\vert ^{r} \wedge 1 \; d\mu(\omega,\omega')   \right \}, \ t \in [0,T] \ , 
\end{eqnarray}
where $\Pi(m,m')$ (resp. $\widetilde{\Pi}(m,m')$) denotes the set of Borel probability measures in $\shp(\shc^d \times \shc^d)$ 
with fixed marginals $m$ and $m'$ belonging to $\shp_r(\shc^d)$ (resp. 
$\shp(\shc^d)$ ).
In this paper we will use very frequently the Wasserstein distances
of order $2$. For that reason, we will simply denote
$W_t: = W^2_t$ (resp. ${\tilde W}_t: = {\tilde W}^2_t$).
 \\
Given $N \in \N^{\star}$, $l \in \shc^d$, $l^1, \cdots, l^N \in \shc^d$, a significant role in this paper will be played by the Borel measures on $\shc ^d$ given by $\delta_l$ and $\displaystyle{ \frac{1}{N} \sum_{j=1}^N \delta_{l^j}}$. 
\begin{rem}
\label{RADelta}
Given $l^1,\cdots,l^N, \tilde{l}^1,\cdots, \tilde{l}^N \in \shc^{d}$, by definition of the Wasserstein distance we have, for all $t \in [0,T]$
\begin{eqnarray}
W_t \left(\frac{1}{N} \sum_{j=1}^N \delta_{l^j},\frac{1}{N} \sum_{j=1}^N \delta_{\tilde{l}^j} \right) \leq \frac{1}{N} \sum_{j=1}^N \sup_{0 \leq s \leq t} \vert l^j_s - \tilde{l}^j_s \vert^2 \nonumber \ .
\end{eqnarray}
\end{rem}
$\shc_b(\shc^d)$ will denote the space of bounded, continuous real-valued functions on $\shc^d$, for which the supremum norm will also be denoted by $\Vert \cdot \Vert_{\infty}$.
In the whole paper, $\R^d$ will be equipped with the scalar product $\cdot $ and $\vert x\vert $  will denote the induced Euclidean norm for $x \in \R^d$. 
$\mathcal{M}_f(\R^d)$ is the space of finite, Borel measures on $\R^d$. $\mathcal{S}(\R^d)$ is the space of Schwartz fast decreasing test functions and $\mathcal{S}'(\R^d)$ is its dual. $\mathcal{C}_b(\R^d)$ is the space of bounded, continuous functions on $\R^d$, $\mathcal{C}^{\infty}_0(\R^d)$ is the space of smooth functions with compact support. $\mathcal{C}^{\infty}_b(\R^d)$ is the space of bounded and smooth functions. $\mathcal{C}_0(\R^d)$ will represent the space of continuous functions with compact support in $\R^d$. 
 $W^{r,p}(\R^d)$ is the Sobolev space of order $r \in \N$ in $(L^p(\R^d),||\cdot||_{p})$, with $1 \leq p \leq \infty$. We denote
 by $(\phi^d_{n})_{n \geq 0}$
 an usual sequence of mollifiers
 $ \phi^d_{n}(x) = \frac{1}{\epsilon_n^d}\phi^d(\frac{x}{\epsilon_n})$ where,
$\phi^d$ is a non-negative function, belonging to the 
Schwartz space whose integral is $1$ and $(\epsilon_n)_{n \geq 0}$ a sequence of strictly positive reals verifying $\epsilon_n \xrightarrow[\text{$n \longrightarrow \infty$}]{\text{}} 0$. 
When $d=1$, we will simply write $\phi_n := \phi^1_n,  \phi := \phi^1.$\\
 $\mathcal{F}(\cdot): f \in \mathcal{S}(\R^d) \mapsto \mathcal{F}(f) \in \mathcal{S}(\R^d)$ will denote the Fourier transform on the classical Schwartz space $\mathcal{S}(\R^d)$ such that for all $\xi \in \R^d$,
 $$
 \mathcal{F}(f)(\xi) = \frac{1}{\sqrt{2 \pi}} \int_{\R^d} f(x) e^{-i \xi \cdot x} dx \ .
 $$ 
 We will denote in the same manner the corresponing Fourier transform on $\mathcal{S}'(\R^d)$ .\\
\begin{def} 
\label{DUnifCont}
A function $F:[0,T] \times \R^d \times \R \rightarrow \R$ will be said 
{\bf continuous with respect to} $(y,z) \in \R^d \times \R$ {(the \textit{space variables}) {\bf uniformly with
respect to} $t \in [0,T]$} if  for every $\varepsilon > 0$,
 there is $\delta > 0$, such that $\forall (y,z),(y',z') \in \R^d\times \R $
\begin{equation} 
\label{DUC}
|y-y'| + |z-z'| \leq \delta \Longrightarrow \; \forall t \in [0,T], \; \vert F(t,y,z)-F(t,y',z')\vert \leq \varepsilon.
\end{equation}
\end{def}

\begin{def} \label{def_r.m}
For any Polish space $E$, we will denote by $\shb(E)$ its Borel $\sigma$-field. It is well-known that $\shp(E)$ is also a Polish space with respect to the weak convergence topology, whose Borel $\sigma$-field will be denoted by $\shb(\shp(E))$ (see Proposition 7.20 and Proposition 7.23, Section 7.4 Chapter 7 in \cite{BertShre}). \\
For any fixed measured space $(\Omega,\shf)$, a map 
$\eta : (\Omega,\shf) \longrightarrow (\shp(E),\shb(\shp(E)))$ will be
 called {\bf{random measure}}  (or random kernel) if it is measurable. We will denote by $\shp_2^{\Omega}(E)$ the space of square integrable random measures, i.e., the space of random measures $\eta$ such that $\E[\eta(E)^2] < \infty$.
\end{def}
\begin{rem} \label{RMeasure}
As highlighted in Remark 3.20 in \cite{crauel} (see also Proposition 7.25 in \cite{BertShre}),
 previous definition is equivalent to the two following conditions:
\begin{itemize}
\item for each $\bar{\omega} \in \Omega$, $\eta_{\bar{\omega}} \in \shp(E)$,
\item for all Borel set $A \in \shb(\shp(E))$, $\bar{\omega} \mapsto \eta_{\bar{\omega}}(A)$ is $\shf$-measurable.
\end{itemize}
\end{rem}
\begin{rem} \label{RDelta}
Given $\R^d$-valued continuous processes $Y^1,\cdots, Y^n$, the application $\displaystyle{ \frac{1}{N} \sum_{j=1}^N \delta_{Y^j} }$ is a random measure on $\shp(\shc^d)$.
In fact $\delta_{Y^j}, 1 \le j  \le N$ is a random measure by Remark \ref{RMeasure}.
\end{rem}
As mentioned in the introduction $K: \R^d \rightarrow \R_+$
will be a mollifier such that $\int_{\R^d} K(x) dx = 1$.
Given a finite signed Borel measure $\gamma$ on $\R^d$,
$K * \gamma$ will denote the convolution function
$x \mapsto \gamma(K(x - \cdot))$. In particular if 
$\gamma$ is absolutely continuous with density 
${\dot \gamma}$, then
$(K * \gamma)(x) = \int_{\R^d} K(x-y) \dot \gamma(y)dy$.
In  this article, the following assumptions will be used. 
\begin{ass}
\label{ass:main} 
\begin{enumerate}
	\item   $\Phi$ and  $g$ are Lipschitz functions defined on $[0,T] \times \R^d \times \R$
 taking values respectively in $\R^{d\times p}$ (space of $d \times p$ matrices) and $\R^d$: there exist finite positive reals $L_\Phi$ and $L_g$ such that for any $(t,y,y',z,z')\in [0,T] \times \R^d \times \R^d \times \R \times \R$, we have
$$
\vert \Phi(t,y',z')-\Phi(t,y,z)\vert \leq L_\Phi (\vert z'-z\vert + \vert y'-y\vert) \quad \textrm{and}\quad \vert g(t,y',z')-g(t,y,z)\vert \leq L_g (\vert z'-z\vert + \vert y'-y\vert)\ .
$$
	\item $\Lambda$ is a Borel real valued function defined on $[0,T]\times \R^d\times \R$  Lipschitz w.r.t. the
 space variables: 
	there exists a finite positive real, $L_{\Lambda}$ such that  for any $(t,y,y',z,z')\in [0,T] \times \R^d\times \R^d\times \R \times \R$,
we have
$$
\vert \Lambda(t,y,z)-\Lambda(t,y',z')\vert \leq L_{\Lambda} (\vert y'-y\vert +\vert z'-z\vert )\ .
$$ 

	\item $\Phi$, $g$ and $\Lambda$ are supposed to be uniformly bounded:  there exist finite positive reals $M_\Phi$, $M_g$ and $M_\Lambda$ such that, for any $(t,y,z)\in [0,T] \times \R^d \times \R$
	
	\begin{enumerate}
	\item $$
\vert \Phi(t,y,z)\vert \leq M_\Phi, \quad \vert g(t,y,z)\vert \leq M_g,$$
\item $$ \vert \Lambda(t,y,z)\vert \leq M_\Lambda \ .
$$
\end{enumerate}

	\item $K : \R^d \rightarrow \R_+$ is integrable, Lipschitz, bounded and whose integral is 1: there exist finite positive reals $M_K$ and $L_K$ such that for any $(y,y')\in\R^d\times \R^d$
$$
\vert K(y) \vert \leq M_K, \quad \vert K(y')-K(y)\vert \leq L_K \vert y'-y\vert\ \quad \textrm{and} \quad \int_{\R^d} K(x) dx = 1 \ .
$$
\item  $\zeta_0$ is a fixed Borel probability measure on $\R^d$ admitting a second order moment. 
\end{enumerate}
\end{ass}

To simplify  we introduce the following notations. 
 \begin{itemize}   
 	\item $V\,:[0,T] \times {\mathcal C}^d \times \mathcal{C}\rightarrow \R$ defined for any pair of functions $y\in \mathcal{C}^d$ and $z\in \mathcal{C}$, by 
\begin{equation}
\label{eq:V}
V_t(y,z):=\exp \left ( \int_0^t \Lambda(s,y_s,z_s) ds\right )\quad \textrm{for any} \ t\in [0,T]\ .
\end{equation}
	\item The real valued process $Z$ such that $Z_s=u(s,Y_s)$, for any $s\in [0,T]$, will often be denoted by $u(Y)$. 
\end{itemize}
With these new notations,  the second equation in~\eqref{eq:NSDE} can be rewritten as
\begin{equation}
\label{eq:mu}
 \nu_t(\varphi)=\E[(K\ast \varphi)(Y_t)V_t(Y,u(Y))]\ ,\quad \textrm{for any}\ \varphi\in \mathcal{C}_b(\R^d,\R)\ ,\\ 
\end{equation}
where  $u(t,\cdot)=\frac{d\nu_t}{dx}$.   
\begin{rem}
\label{rem:V}
Under Assumption \ref{ass:main}. 3.(b), $\Lambda$ is bounded.
Consequently
 \begin{equation}
 \label{eq:Vmajor1}
 0\leq V_t(y,z)\leq e^{tM_{\Lambda}}\ ,\quad\textrm{for any}\ (t,y,z)\in [0,T]\times \R^d\times \R\ . 
 \end{equation}
 Under Assumption \ref{ass:main}. 2.
  $\Lambda$ is Lipschitz. Then $V$ inherits in some sense this property. Indeed, observe that for any $(a,b)\in \R^2$, 
 \begin{equation}        \label{EMajor1}
e^{b}-e^{a}=(b-a)\int_0^1 e^{\alpha b+(1-\alpha ) a}d\alpha 
\leq  e^{\sup (a, b)} \vert b-a \vert \ .
\end{equation}
Then for  any continuous functions $y,\,y'\in \mathcal{C}^d =\mathcal{C}([0,T],\R^d)$, and   $z,\,z'\in\mathcal{C}([0,T],\R)$, taking $a=\int_0^t\Lambda(s,y_s,z_s)ds$ and $b=\int_0^{t}\Lambda(s,y'_s,z'_s)ds$ in the above equality yields 
\begin{eqnarray}
\label{eq:Vmajor2}
\vert V_{t}(y',z')-V_t(y,z)\vert &\leq & e^{tM_{\Lambda}}\int_0^t\left \vert \Lambda(s,y'_s,z'_s)-\Lambda(s,y_s,z_s)\right \vert ds \nonumber \\
&\leq & e^{tM_{\Lambda}} L_{\Lambda} \, \int_0^t \left (\vert y'_s-y_s\vert +\vert z'_s-z_s\vert \right )ds \ .
\end{eqnarray}
\end{rem}

In Section \ref{S4}, Assumption \ref{ass:main}. will be replaced by the following.

\begin{ass}
\label{ass:main2} 

\begin{enumerate} 
\item All the items of Assumption \ref{ass:main} are in force excepted 2. which is replaced by the following.
	\item $\Lambda$ is a real valued function defined on $[0,T]\times \R^d\times \R$   continuous w.r.t. 
the space variables uniformly 
with respect to the time variable, see e.g. \eqref{DUC}.
\end{enumerate}
\end{ass}

\begin{rem} \label{RContmain}
 The second item in Assumption \ref{ass:main2}. is fulfilled
if the function $\Lambda$ is continuous with respect to 
$(t,y,z) \in [0,T] \times \R^d \times \R$.
\end{rem}

In Section \ref{S5} we will treat the case when only weak solutions (in law)
exist. In this case we will assume the following.

\begin{ass} 
\label{ass:main3} 
All the items of Assumption \ref{ass:main}. are in force excepted 5.
 and the items 1. and 2. 
which are replaced by the following. 
 $\Phi : [0,T] \times \R^{d} \times \R 
 \longrightarrow \R^{d \times p}, \quad
g : [0,T] \times \R^d \times \R \longrightarrow \R^d$ and 
 $\Lambda: [0,T]\times \R^d\times \R  \rightarrow \R$
are   continuous with respect to
  the space variables uniformly with respect to the time variable.

\end{ass}


\begin{defi} \label{def-strong-sol}
\begin{enumerate}
\item
We say that  \eqref{eq:NSDE} admits {\bf strong existence}
if for any filtered probability space 
$(\Omega, \mathcal{F},\mathcal{F}_t,\mathbb{P})$ equipped with an
$(\mathcal{F}_t)_{t \geq 0}$-Brownian motion $W$, an $\shf_0$-random variable
$Y_0$ distributed according to $\zeta_0$, there is a couple $(Y,u)$
where  $Y$ is an  $(\mathcal{F}_t)_{t \geq 0}$-adapted process
 and   $u:[0,T] \times \R^d \rightarrow \R $, 
verifies \eqref{eq:NSDE}. 
\item We say that  \eqref{eq:NSDE} admits {\bf pathwise uniqueness}
if for any filtered probability space 
$(\Omega, \mathcal{F},\mathcal{F}_t,\mathbb{P})$ equipped with an
$(\mathcal{F}_t)_{t \geq 0}$-Brownian motion $W$, an $\shf_0$-random variable
$Y_0$ distributed according to $\zeta_0$, the following holds.
Given two  pairs $(Y^1,u^1)$ and $(Y^2,u^2)$
as in item 1., verifying \eqref{eq:NSDE} such that
$Y^1_0 = Y^2_0$ $\P$-a.s. then $u^1 = u^2$ and $Y^1$ and $Y^2$
are indistinguishable.
\end{enumerate}
 \end{defi}

\begin{defi} \label{def-weak-sol}
\begin{enumerate}
\item
We say that  \eqref{eq:NSDE} admits {\bf existence in law} 
(or {\bf weak existence}) if 
there is a filtered probability space 
$(\Omega, \mathcal{F}, \mathcal{F}_t,\mathbb{P})$ equipped with an
$(\mathcal{F}_t)_{t \geq 0}$-Brownian motion $W$, 
 a  pair $(Y,u)$,  verifying \eqref{eq:NSDE}, where 
$Y$ is an $(\mathcal{F}_t)_{t \geq 0}$-adapted process
 and   $u$ is a real valued function defined on $[0,T] \times \R^d$. 
\item We say that  \eqref{eq:NSDE} admits {\bf  uniqueness in law}
(or {\bf  weak uniqueness}), if the following holds.
Let
$(\Omega, \mathcal{F},\mathcal{F}_t,\mathbb{P})$
(resp. $(\tilde{\Omega}, \tilde{\mathcal{F}},\tilde{\mathcal{F}}_t,\tilde{\mathbb{P}})$)
be a filtered probability space.
Let $(Y^1, u^1)$ (resp. $(\tilde{Y}^2, \tilde{u}^2)$)
be a solution of  \eqref{eq:NSDE}. 
Then $u^1=\tilde{u}^2$ and $Y^1$ and $\tilde{Y}^2$ have the same law.
\end{enumerate}
\end{defi}

\section{Existence and uniqueness of the problem in the Lipschitz case}

\setcounter{equation}{0}

\label{S3}

In this section we will fix a probability space $(\Omega, {\mathcal F},{\mathcal F}_t, \P)$
equipped with an $({\mathcal F}_t)$-Brownian motion $(W_t)$.
We will proceed in two steps. We study first in the next section the second equation of \eqref{eq:NSDE} defining $u$. Then we will address the main equation defining the process $Y$. 

Later in this section, Assumption \ref{ass:main} will be in force, in particular $\zeta_0$ will be supposed to have a second order moment.

\subsection{Existence and uniqueness of a solution to the linking equation}

This subsection relies only on items $2.$, $3.(b)$ and $4.$ of Assumption \ref{ass:main}. \\ 
We remind that $X$ will denote the canonical process
$X: {\mathcal C}^d \rightarrow  {\mathcal C^d}$
defined by $X_t(\omega) =\omega(t), \ t \ge 0, \omega \in {\mathcal C^d}$.

For a given probability measure $m\in \mathcal{P}(\mathcal{C}^d)$, let us consider the equation 
\begin{equation}
\label{eq:u}
\left\{
\begin{array}{l}
u(t,y)=\int_{\mathcal{C}^d} K(y- X_t(\omega))V_t(X(\omega),u(X(\omega)))
 dm(\omega)\ ,\quad \textrm{for all}\quad t\in [0,T], y \in \R^d \ ,\quad \textrm{with}\\\\
V_t(X(\omega),u(X(\omega)))=\exp \left ( \int_0^t \Lambda \big (s,
X_s(\omega),u(s,X_s(\omega))\big ) ds\right )\ .
\end{array}
\right .
\end{equation}
This equation will be called {\bf linking equation}: it constitutes  the second line of equation \eqref{eq:NSDE}. When $\Lambda = 0$, 
i.e. in the conservative case, $u(t,\cdot) = K * m_t$, where $m_t$ is 
the marginal law of $X_t$ under $m$. Informally speaking, when $K$
is the Delta Dirac measure, then $u(t,\cdot) = m_t$.


\begin{rem} \label{RLipeq:u}
Since $\Lambda$ is bounded, and $K$ Lipschitz, it is clear that if $u:=u^m$ is a solution of \eqref{eq:u}
then $u$ is necessarily bounded by a constant, only depending on $M_\Lambda,M_K,T$ and it is Lipschitz with respect to the second
argument with Lipschitz constant only depending on $L_K,M_\Lambda,T$.
\end{rem}
%
We aim at proving by a fixed point argument the following result. 
\begin{lem}
\label{lem:u}
We assume the validity of items $2.$, $3.(b)$ and $4.$ of Assumption \ref{ass:main}. \\ 
For a given probability measure $m\in \mathcal{P}(\mathcal{C}^d)$, equation~\eqref{eq:u} admits a unique solution, $u^m$. 
\end{lem}
\begin{proof} 
Let us introduce the linear space $\mathcal{C}_1$ of real valued continuous processes $Z$ on $[0,T]$ (defined on the canonical space $\mathcal{C}^d$)
 such that 
$$
\Vert Z\Vert_{\infty , 1}:=\E^m \,\left[\,\sup_{t\leq T}
 \vert Z_t\vert\, \right] := \int_{\mathcal C^d} \; \sup_{0 \leq t \leq T} 
\vert Z_t(\omega) \vert dm(\omega)
  <\infty\ .
$$
$(\mathcal{C}_1,\Vert \cdot\Vert_{\infty ,1})$ is a Banach space.
For any $M \ge 0$, 
a well-known equivalent norm to $\Vert\cdot\Vert_{\infty ,1}$ is given by
 $\Vert \cdot \Vert^M_{\infty ,1}$, where
$\Vert Z \Vert^M_{\infty ,1}=\E^m\,[\,\sup_{t\leq T}\, e^{-Mt} \vert Z_t\vert\,]. $
%
Let us define the operator $T^m: \mathcal{C}_1\rightarrow \mathcal{C}([0,T] \times \R^d, \R)$ such that for any $Z\in \mathcal{C}_1$,
\begin{equation}
\label{eq:T}
T^m(Z)(t,y):=\int_{\mathcal{C}^d} K\big (y-X_t(\omega)\big ) 
V_t\big (X(\omega),Z(\omega)\big ) dm(\omega)\ .
\end{equation}
 Then we introduce the operator
$\tau : f \in \mathcal{C}([0,T] \times \R^d,\R) \mapsto \tau(f) \in \mathcal{C}_{1}$, where $\tau(f)_t(\omega) = f(t,\omega_t)$.
We observe that $\tau \circ T^m $ is a map
  $\mathcal{C}_1\rightarrow \mathcal{C}_1$.

Notice that equation~\eqref{eq:u} is equivalent to 
\begin{equation}
\label{eq:ubis}
u = (T^m \circ \tau)(u).
\end{equation}


We first admit the existence and uniqueness of a fixed point 
$Z\in {\mathcal C}_1$
for the map $\tau \circ T^m$. In particular we have 
 $Z=(\tau \circ T^m)(Z)$.   
We can now  deduce the existence/uniqueness for the function $u$
for  problem \eqref{eq:ubis}. \\
Concerning  existence, we choose 
 $v^m := T^m(Z)$. 
Since $Z$ is a fixed-point of the map $\tau \circ T^m$, by the definition of $v^m$ 
we have
\begin{equation}
\label{eq:Z}
Z = \tau(T^m(Z)),  
\end{equation}
so that  $v^m$ is a solution of \eqref{eq:ubis}. 
  \\
Concerning uniqueness of \eqref{eq:ubis},
 we consider two solutions of \eqref{eq:u} $\bar{v},\tilde{v}$, i.e.  such that 
$\bar{v} = (T^m \circ \tau)(\bar{v}), \tilde{v} = (T^m \circ \tau)(\tilde{v})  $.
We set $\bar{Z} := \tau(\bar{v}), \tilde{Z} := \tau(\tilde{v})$.
Since $\bar{v} = T^m(\bar{Z})$ we have
$\bar{Z} = \tau(\bar{v}) = \tau(T^m(\bar{Z}))$. Similarly
  $\tilde{Z} = \tau(\tilde{v}) = \tau(T^m(\tilde{Z}))$. Since 
$\bar{Z}$ and $\tilde{Z}$ are fixed points of $\tau \circ T$,
it follows that $\bar{Z} = \tilde{Z} \ dm$ a.e.
Finally $\bar{v} =  T^m(\bar{Z}) =  T^m(\tilde{Z}) = \tilde{v}$.

It remains finally to prove that $\tau \circ T^m $ admits a unique 
fixed point, $Z$. \\
The upper bound~\eqref{eq:Vmajor2} implies that for any pair $(Z,Z')\in\mathcal{C}_1\times \mathcal{C}_1$, for any $(t,y)\in [0,T]\times \R^d$,
\begin{eqnarray*}
\vert T^m(Z')-T^m(Z)\vert (t,y)&=&
\left \vert \int_{\mathcal{C}^d} K(y-X_t(\omega)) \left [V_t(X(\omega),Z'(\omega))-V_t(X(\omega),Z(\omega))\right ] dm(\omega) \right \vert \ \\
&\leq & 
M_Ke^{tM_{\Lambda}}L_{\Lambda}\int_{\mathcal{C}^d} \int_0^t\vert Z'_s(\omega)-Z_s(\omega)\vert ds\,dm(\omega) \\
&\leq & 
M_Ke^{TM_\Lambda}L_{\Lambda}\E \left [\int_0^t e^{Ms}e^{-Ms}\vert Z'_s-Z_s\vert ds\right ]\\
&\leq & 
M_Ke^{TM_\Lambda }L_{\Lambda}\E\left [ \int_0^t e^{Ms}\sup_{r\leq t} e^{-Mr}\vert Z'_r-Z_r\vert ds\right ]\\
&\leq & 
M_Ke^{TM_\Lambda }L_{\Lambda}\frac{e^{Mt}-1}{M}\E \left [\sup_{r\leq t}  e^{-Mr}\vert Z'_r-Z_r\vert \right ]\\
&\leq & 
M_Ke^{TM_\Lambda }L_{\Lambda}\frac{e^{Mt}-1}{M}\Vert Z'-Z\Vert_{\infty , 1}^M \ .
\end{eqnarray*}
Then considering  $(\tau \circ T^m)(Z')_t = 
  T^m((Z')(t,X_t)$ and $(\tau \circ T^m) (Z)_t=T(Z)(t,X_t)$, we obtain 
\begin{eqnarray*}
\sup_{t\leq T} e^{-Mt}\left \vert (\tau \circ T^m) (Z')_t- 
 (\tau \circ T^m)(Z)_t\right \vert 
&=&
\sup_{t\leq T} e^{-Mt}\left \vert T^m(Z')(t,X_t)-T^m(Z)(t,X_t)\right \vert\\
&\leq &
M_Ke^{TM_\Lambda }L_{\Lambda}\frac{1}{M}\Vert Z'-Z\Vert^M_{\infty , 1}\ .
\end{eqnarray*}
Taking the expectation yields $\vert  (\tau \circ T^m)(Z')_t-
    (\tau \circ T^m)(Z)_t\Vert^M_{\infty ,1}\leq M_Ke^{TM_\Lambda }L_{\Lambda}\frac{1}{M}\Vert Z'-Z\Vert^M_{\infty , 1}$. 
Hence, as soon as $M$ is sufficiently  large, $M>M_Ke^{TM_\Lambda }L_{\Lambda}$, 
 $(\tau \circ T^m)$  is a contraction on $(\mathcal{C}_1, \Vert \cdot \Vert^M_{\infty ,1})$ and the proof ends  by a simple application of the Banach fixed point theorem. 
%
%
\end{proof}
\begin{rem} \label{RtCont}
For $(y, m) \in  \R^d \times \mathcal{P}(\mathcal{C}^d)$,
$t \mapsto u^m (t,y)$ is continuous. This follows by an application of Lebesgue dominated convergence theorem 
in \eqref{eq:u}.
\end{rem}

In the sequel, we will need a  stability result on $u^m$ solution of~\eqref{eq:u}, w.r.t. the probability measure $m$.

The fundamental lemma treats this  issue, again  only supposing
 the validity of items $2.$, $3.(b)$ and $4.$ of Assumption \ref{ass:main}. 

\begin{lem}
\label{lem:uu'}
We assume the validity of items $2.$, $3.(b)$ and $4.$ of Assumption \ref{ass:main}. \\ 
Let $u$ be a solution of~\eqref{eq:u}. The following assertions hold. 
\begin{enumerate} 
\item For any measures $(m,m')\in \mathcal{P}_2(\mathcal{C}^d)\times \mathcal{P}_2(\mathcal{C}^d)$, for all $(t,y,y') \in [0,T] \times \shc^d \times \shc^d$,
we have
\begin{eqnarray}
\label{eq:uu'}
\vert u^m\big (t,y\big )-u^{m'}\big (t,y'\big )\vert^2\leq C_{K,\Lambda}(t)\left [ \vert y-y'\vert^2+\vert W_t(m,m')\vert ^2\right ] \ ,
\end{eqnarray}
where 
$C_{K,\Lambda}(t):=2C'_{K,\Lambda}(t)(t+2)(1+e^{2tC'_{K,\Lambda}(t)})$ with $C'_{K,\Lambda}(t)=2e^{2tM_{\Lambda}}(L^2_K+2M^2_KL^2_\Lambda t )$.
In particular the functions $C_{K, \Lambda}$ only depend on $M_K, L_K, M_\Lambda, L_\Lambda$ and $t$ and is increasing with $t$.
\item For any measures $(m,m') \in \shp(C^d) \times \shp(C^d)$,  for all $(t,y,y') \in [0,T] \times \shc^d \times \shc^d$, we have
\begin{eqnarray}
\label{eq:uu'2}
\vert u^m\big (t,y\big )-u^{m'}\big (t,y'\big )\vert^2\leq \mathfrak{C}_{K,\Lambda}(t)\left [ \vert y-y'\vert^2+\vert \widetilde{W}_t(m,m')\vert ^2\right ] \ ,
\end{eqnarray}
where $\mathfrak{C}_{K,\Lambda}(t) := 2e^{2tM_{\Lambda}}(\max(L_K,2M_K)^2 + 2M_{K}^2\max(L_{\Lambda},2M_{\Lambda})^2 t ) $.
 \item The function $ (m,t,x) \mapsto u^m(t,y)$ is continuous on $\shp(\shc^d) \times [0,T] \times \R^d$ where $\shp(\shc^d)$ is endowed with the topology of weak convergence.
\item Suppose 
that $K \in W^{1,2}(\R^d)$. Then for any $(m,m')\in \mathcal{P}_2(\mathcal{C}^d)\times \mathcal{P}_2(\mathcal{C}^d)$, $t \in [0,T]$  
\begin{eqnarray}
\label{uu'L2}
\Vert u^m(t,\cdot)-u^{m'}(t,\cdot)\Vert_2^2 
 \leq   \tilde{C}_{K,\Lambda}(t)(1+2tC_{K,\Lambda}(t)) \vert W_t(m,m') \vert^2 \ ,
\end{eqnarray}
where 
$C_{K,\Lambda}(t):=2C'_{K,\Lambda}(t)(t+2)(1+e^{2tC'_{K,\Lambda}(t)})$ with $C'_{K,\Lambda}(t)=2e^{2tM_{\Lambda}}(L^2_K+2M^2_KL^2_\Lambda t )$ and
$\tilde{C}_{K,\Lambda}(t) := 2 e^{2tM_{\Lambda}}(2M_K L_{\Lambda}^2 t(t+1) 
+ \Vert \nabla K \Vert_2^2 )$, $\Vert\cdot\Vert_2$ being the standard 
$L^2(\R^d)$ or $L^2(\R^d,\R^d)$-norms. \\
In particular the functions $\tilde C_{K, \Lambda}$ 
 only depend on $M_K, L_K, M_\Lambda, L_\Lambda$ and $t$ and is increasing with $t$.
\item Suppose 
that $\shf(K) \in L^1(\R^d)$. 
 Then there exists a constant $\bar{C}_{K,\Lambda}(t) > 0$ (depending only on $t,M_{\Lambda},L_{\Lambda},\Vert \shf(K) \Vert_1$) such that for any random measure $\eta : (\Omega,\shf) \longrightarrow (\shp_2(\shc^d),\shb(\shp(\shc^d)))$, for all $(t,m) \in [0,T] \times \shp_2(\shc^d)$
\begin{eqnarray}
\label{eq:uu'Linf}
  \E{ [  \Vert u^{\eta}(t,\cdot) - u^m(t,\cdot) \Vert_{\infty}^2 ] } & \leq & \bar{C}_{K,\Lambda}(t) \sup_{ \underset{\Vert \varphi \Vert_{\infty} \le 1}{\varphi \in \shc_b(\shc^d)}} 
 \E{ [ \vert \langle \eta - m , \varphi \rangle \vert^2 ]} \ ,
\end{eqnarray}
where we recall that $\shp(\shc^d)$ is endowed with the topology of weak convergence.
We remark that the expectation in both sides of \eqref{eq:uu'Linf} is taken w.r.t. the randomness of the random measure $\eta$.
\end{enumerate}
\end{lem}

\begin{rem}
\label{R25}
\begin{enumerate}[a)]
\item By Corollary 6.13, Chapter 6 in \cite{villani}, $\widetilde{W}_T$ is a metric compatible with the weak convergence on $\shp(\shc^d)$. 
\item The map $\displaystyle{ d_2^{\Omega} : (\nu,\mu) \mapsto \sqrt{\sup_{ \underset{\Vert \varphi \Vert_{\infty} \le 1}{\varphi \in \shc_b(\shc^d)}} \E[ \vert \langle \nu - \mu, \varphi \rangle \vert^2 ] } }$ defines a (homogeneous) distance on $\shp_2^{\Omega}(\shc^d)$.
\item Previous distance satisfies
\begin{equation} \label{dWTVillani}
d_2^{\Omega}(\nu,\mu) \leq \sqrt{\E[\vert W_T^1(\nu,\mu) \vert^2]} \ ,
\end{equation}
where $W^1_T$ is the $1$-Wasserstein distance. \\
Indeed, for fixed $\bar \omega \in \Omega$, taking into account that
 $(\shc^d,\nu_{\bar{\omega}})$ and $(\shc^d,\mu_{\bar{\omega}})$ are Polish probability spaces, the first equality of (i) in the  Kantorovitch duality theorem, see 
 Theorem 5.10 p.70
 in \cite{villani}, which in particular implies  the following. 
For any $\varphi  \in  \shc_b(\shc^d)$ we have 
$$\vert \langle \nu - \mu, \varphi \rangle \vert \le  W_T^1(\nu,\mu),$$ 
which implies \eqref{dWTVillani}.
\item The map $(\nu,\mu) \mapsto \sqrt{\E[\vert W_T^1(\nu,\mu) \vert^2]}$ defines a distance on $\shp_2^{\Omega}(\shc^d)$.
\item Item 1. of Lemma \ref{lem:uu'} is a consequence of item 2. 
For expository reasons, we have decided to start with the less general case.
\end{enumerate}
\end{rem}
\begin{proof}[ Proof of Lemma~\ref{lem:uu'}]
We will prove successively the inequalities \eqref{eq:uu'}, \eqref{eq:uu'2}, \eqref{uu'L2} and \eqref{eq:uu'Linf}. \\
Let us consider $(t,y,y')\in [0,T]\times \R^d\times \R^d$.
\begin{itemize}
\item {\bf {Proof of \eqref{eq:uu'} }}.
Let $(m,m')\in \mathcal{P}_2(\mathcal{C}^d)\times \mathcal{P}_2(\mathcal{C}^d)$. \\
We have
\begin{equation}
\label{eq:uu'major}
\vert u^m(t,y)-u^{m'}(t,y')\vert^2 \leq 2\vert u^m(t,y)-u^m(t,y')\vert ^2+2\vert u^m(t,y')-u^{m'}(t,y')\vert ^2\ .
\end{equation}
The first term on the r.h.s. of the above equality is bounded using the Lipschitz property of $u^m$ that derives straightforwardly from the Lipschitz property of the mollifier $K$ and the boundedness property of $V_t$~\eqref{eq:Vmajor1}: 
\begin{eqnarray}
\label{eq:uLip}
\vert u^m(t,y')-u^m(t,y)\vert 
&=&
\left \vert \int_{\mathcal{C}^d} \left [ K(y-X_t(\omega))-K(y'-X_t(\omega))\right ] V_t(X(\omega),u^m(X(\omega))) dm(\omega)\right \vert \nonumber \\
&\leq &
L_K e^{tM_{\Lambda}}\vert y-y'\vert \ .
\end{eqnarray}
Now let us consider the second term on the r.h.s of~\eqref{eq:uu'major}.
By Jensen's inequality we get 
\begin{eqnarray}
\label{eq:uy'u'y'}
\vert u^m(t,y')-u^{m'}(t,y')\vert^2 
&= &\left \vert \int_{\mathcal{C}^d} K(y'-X_t(\omega))V_t\big (X(\omega),u^m(X(\omega))\big )dm(\omega)\right . \nonumber \\
&&\left .-\int_{\mathcal{C}^d} K(y'-X_t(\omega'))V_t\big (X(\omega'),u^{m'}(X(\omega'))\big )dm'(\omega')\right \vert^2 \nonumber\\
&\leq &
\int_{\mathcal{C}^d \times \mathcal{C}^d} \Big \vert K(y'-X_t(\omega))V_t\big (X(\omega),u^m(X(\omega))\big ) \nonumber \\
&& - \; K(y'-X_t(\omega'))V_t\big (X(\omega'),u^{m'}(X(\omega'))\big )\Big \vert^2 d\mu(\omega , \omega')\ ,
\end{eqnarray}
for any $\mu \in \Pi(m,m')$. 
Let us consider four  continuous functions $x,\,x' \in \mathcal{C}([0,T],\R^d)$ and   $z,z' \in \mathcal{C}([0,T],\R)$. We have 
\begin{eqnarray}
\label{eq:KVKV'}
\left \vert K(y'-x_t)V_t(x,z)-K(y'-x'_t)V_t(x',z')\right \vert^2
&\leq &2 \left \vert K(y'-x_t)-K(y'-x'_t)\right \vert^2 \vert V_t(x,z)\vert ^2 \nonumber \\
&&+2\left \vert V_t(x,z)-V_t(x',z')\right \vert^2 \vert K(y'-x'_t)\vert ^2\ .
\end{eqnarray}
Then, using the Lipschitz property of $K$ and 
the upper bound~\eqref{eq:Vmajor2} gives
\begin{eqnarray} \label{eq:KVKV''}
\left \vert K(y'-x_t)V_t(x,z)-K(y'-x'_t)V_t(x',z')\right \vert^2
&\leq &2 L^2_Ke^{2tM_{\Lambda}} \vert x_t-x'_t\vert ^2   \nonumber \\
&+&4M^2_KL^2_\Lambda e^{2tM_{\Lambda}} t \int_0^t \left [\vert x_s- x'_s\vert^2 +\vert z_s-z'_s\vert ^2\right ]\,ds\\
&\leq & C'_{K,\Lambda}(t)\left [(1+t) \sup_{s\leq t}\vert x_s- x'_s\vert^2 +\int_0^t\vert z_s-z'_s\vert ^2\,ds\right]   \nonumber \ ,
\end{eqnarray}
where  $C'_{K,\Lambda}(t)=2e^{2tM_{\Lambda}}(L^2_K+2M^2_KL^2_\Lambda t )$. 
Injecting the latter inequality in~\eqref{eq:uy'u'y'} yields 
\begin{eqnarray*}
\vert u^m(t,y')-u^{m'}(t,y')\vert^2 
& \leq & C'_{K,\Lambda}(t) \int_{\mathcal{C}^d \times \mathcal{C}^d} \left [(1+t)\sup_{s\leq t}\vert X_s(\omega)- X_s(\omega')\vert^2 \right. \\
&& \left. +\int_0^t\vert u^m(s,X_s(\omega))-u^{m'}(s,X_s(\omega'))\vert ^2\,ds\right] d\mu(\omega , \omega')\ ,
\end{eqnarray*}
Injecting the above inequality in~\eqref{eq:uu'major} and using~\eqref{eq:uLip} yields
\begin{eqnarray}
\label{eq:uu'bis}
\vert u^m(t,y)-u^{m'}(t,y')\vert^2 &\leq & 2C'_{K,\Lambda}(t)\left 
[\vert y-y'\vert^2+ (1+t) \int_{\mathcal{C}^d \times \mathcal{C}^d}
  \sup_{s\leq t}\vert X_s(\omega)- X_s(\omega')\vert^2 d\mu(\omega , \omega')\right .\nonumber \\
&&\left . +\int_{\mathcal{C}^d \times \mathcal{C}^d} \int_0^t\vert u^m(s,X_s(\omega))-u^{m'}(s,X_s(\omega'))\vert ^2\,ds\,d\mu(\omega , \omega')\right ]\ ,
\end{eqnarray}
Replacing $y$ (resp. $y'$) with $X_t(\omega)$ (resp. $X_t(\omega')$) in~\eqref{eq:uu'bis}, we get for all $\omega \in \mathcal{C}^d$ (resp. $\omega' \in \mathcal{C}^d$), 
\begin{eqnarray}
\label{um-X}
|u^m(t,X_t(\omega)) - u^{m'}(t,X_t(\omega'))|^2 & \leq & 2C'_{K,\Lambda}(t) \Big[ |X_t(\omega)-X_t(\omega')|^2  \nonumber \\
&&  + (1+t) \int_{\mathcal{C}^d \times \mathcal{C}^d} \displaystyle{\sup_{s \leq t}} |X_s(\omega)-X_s(\omega')|^2 d\mu(\omega,\omega') \nonumber \\
& &  +\int_{\mathcal{C}^d \times \mathcal{C}^d} \int_0^t\vert u^m(s,X_s(\omega))-u^{m'}(s,X_s(\omega'))\vert ^2\,ds\,d\mu(\omega , \omega') \Big]. \nonumber \\
\end{eqnarray}
Let us introduce the following notation
$$
\gamma(s):= \int_{\mathcal{C}^d \times \mathcal{C}^d} \vert u^m(s,X_s(\omega))-u^{m'}(s,X_s(\omega'))\vert ^2\,d\mu(\omega , \omega')\ ,\quad \textrm{for any}\ s\in
 [0,T].\  $$ 
Integrating each side of  inequality~\eqref{um-X} w.r.t. 
the variables $(\omega,\omega')$ according to $\mu$, implies 
$$
\gamma(t)\leq 2 C'_{K,\Lambda}(t)\int_0^t\gamma(s)ds+ 2(t+2)
C'_{K,\Lambda}(t)\int_{\mathcal{C}^d \times \mathcal{C}^d} \sup_{s\leq t}\vert X_s(\omega)- X_s(\omega')\vert^2 d\mu(\omega , \omega')\ ,
$$
for all $t \in [0,T]$. In particular, observing that $C'_{K,\Lambda}(a)$ is increasing in $a$, we have for fixed $t \in ]0,T]$ and all $a \in [0,t]$
$$
\gamma(a)\leq 2 C'_{K,\Lambda}(t)\int_0^a\gamma(s)ds+ 2(t+2)
C'_{K,\Lambda}(t)\int_{\mathcal{C}^d \times \mathcal{C}^d} \sup_{s\leq t}\vert X_s(\omega)- X_s(\omega')\vert^2 d\mu(\omega , \omega')\ .
$$
Using Gronwall's lemma yields
\begin{eqnarray*}
\gamma(t) &:=& \int_{\mathcal{C}^d \times \mathcal{C}^d} \vert u^m(t,X_t(\omega))-u^{m'}(t,X_t(\omega'))\vert ^2\,d\mu(\omega , \omega') \nonumber \\
& \leq & 2(t+2)C'_{K,\Lambda}(t)e^{2tC'_{K,\Lambda}(t)}\int_{\mathcal{C}^d \times \mathcal{C}^d} \sup_{s\leq t}\vert X_s(\omega)- X_s(\omega')\vert^2 d\mu(\omega , \omega')\ .
\end{eqnarray*}
Injecting the above inequality in~\eqref{eq:uu'bis} implies
$$
\vert u^m(t,y)-u^{m'}(t,y')\vert^2 
\leq 
2C'_{K,\Lambda}(t)(t+2)(1 + e^{2tC'_{K,\Lambda}(t)})
\left [\vert y-y'\vert^2+ \int_{\mathcal{C}^d \times \mathcal{C}^d} \sup_{s\leq t}\vert X_s(\omega)- X_s(\omega')\vert^2 d\mu(\omega , \omega')\right ]
\ .
$$
The above inequality holds for any $\mu\in \Pi (m,m')$, hence taking the infimum over $\mu\in \Pi (m,m')$ concludes the proof of \eqref{eq:uu'}.  \\
\item {\bf{Proof of \eqref{eq:uu'2}}}.
Let $(m,m')\in \mathcal{P}(\mathcal{C}^d)\times \mathcal{P}(\mathcal{C}^d)$. 
The proof of \eqref{eq:uu'2} follows at the beginning the same lines as the one of \eqref{eq:uu'}, but the inequality  
     \eqref{eq:KVKV''} is replaced by
	\begin{eqnarray}
\left \vert K(y'-x_t)V_t(x,z)-K(y'-x'_t)V_t(x',z')\right \vert^2
&\leq & 2 \left \vert K(y'-x_t)-K(y'-x'_t)\right \vert^2 \vert V_t(x,z)\vert ^2 \nonumber \\
&&+2\left \vert V_t(x,z)-V_t(x',z')\right \vert^2 \vert K(y'-x'_t)\vert ^2 \nonumber \\
& \leq & 2 e^{2tM_{\Lambda}} \max(L_K,2M_K)^2 (\left \vert x_t-x'_t\right \vert^2 \wedge 1) \nonumber \\
&& + 4M_{K}^2 e^{2tM_{\Lambda}} \max(L_{\Lambda},2M_{\Lambda})^2 t \int_0^t \Big( \vert x_s'-x_s \vert^2 \wedge 1  \nonumber \\
& & + \; \vert z_s-z'_s  \vert^2 \Big) ds  \nonumber \\ 
& \leq & \mathfrak{C}_{K,\Lambda}(t)\left [(1+t) (\sup_{s\leq t}\vert x_s- x'_s\vert^2 \wedge 1) + \int_0^t\vert z_s-z'_s\vert ^2\,ds\right], \nonumber \\
\end{eqnarray}
which implies
\begin{eqnarray}
\vert u^m(t,y)-u^{m'}(t,y')\vert^2
& \leq & 2\mathfrak{C}_{K,\Lambda}(t)(t+2)(1 + e^{2t\mathfrak{C}_{K,\Lambda}(t)}) \Big[ \; \vert y-y'\vert^2 \nonumber \\
&& + \int_{\mathcal{C}^d \times \mathcal{C}^d} \sup_{s\leq t}\vert X_s(\omega)- X_s(\omega')\vert^2 \wedge 1 \; d\mu(\omega , \omega') \Big] \ ,
\end{eqnarray}
where $\mathfrak{C}_{K,\Lambda}(t) := 2e^{2tM_{\Lambda}}(\max(L_K,2M_K)^2 + 2M_{K}^2\max(L_{\Lambda},2M_{\Lambda})^2 t ) $.
This gives the analogue of  \eqref{eq:uu'bis}
 and we conclude in the same way as for the previous item.
\item {\bf{Proof of the continuity of $(m,t,x) \mapsto u^m(t,x)$.}} \\
$\shp(\shc^d) \times [0,T] \times \R^d$ being a separable metric space, we characterize the continuity through converging sequences. We also recall that $\widetilde{W}_T$ is a metric compatible with the weak convergence on $\shp(\shc^d)$, 
see Remark \ref{R25} a). \\
By \eqref{eq:uu'}, the application is continuous with respect to $(m,x)$ uniformly with respect to time.
Consequently it remains to show that the map $t \mapsto u^m(t,x)$ is continuous for fixed
$(m,x) \in \shp(\shc^d) \times \R^d$. \\
 Let us fix $(m,t_0,x) \in \shp(\shc^d) \times [0,T] \times \R^d$. Let $(t_n)_{n \in \N}$ be a sequence in $[0,T]$ converging to $t_0$. \\
We define $F_n$ as the real-valued sequence of measurable functions  on $\shc^d$ such that for all $\omega \in \shc^d$, 
\begin{eqnarray}
\label{def_Fn}
F_n(\omega) := K(x-X_{t_n}(\omega)) \exp \left( \int_0^{t_n} \Lambda(r,X_r(\omega),u^{m}(r,X_r(\omega)) dr \right) \ .
\end{eqnarray}
Each $\omega \in \shc^d$ being continuous, $F_n$ converges pointwise to $F : \shc^d \rightarrow \R$ defined by 
\begin{eqnarray}
\label{def_F}
F(\omega) := K(x-X_{t_0}(\omega)) \exp \left( \int_0^{t_0} \Lambda(r,X_r(\omega),u^{m}(r,X_r(\omega)) dr \right) \ .
\end{eqnarray}
Since $K$ and $\Lambda$ are uniformly bounded, $M_K e^{TM_{\Lambda}}$ is a uniform upper bound of the functions $F_n$. By Lebesgue dominated convergence theorem, we conclude that 
$$
\vert u^{m}(t_n,x) - u^{m}(t_0,x) \vert = \left \vert \int_{\shc^d} F_n(\omega) dm(\omega) - \int_{\shc^d} F(\omega) dm(\omega) \right \vert \xrightarrow[n \rightarrow + \infty]{} 0 \ .
$$
This ends the proof.

\item {\bf{Proof of \eqref{uu'L2}}}.
Let $(m,m')\in \mathcal{P}_2(\mathcal{C}^d)\times \mathcal{P}_2(\mathcal{C}^d)$. \\
Since $K \in L^2(\R^d)$, by Jensen's inequality, it follows easily  
that the functions $x  \mapsto u^m(r,x)$ and $x \mapsto u^{m'}(r,x)$ belong
to $L^2(\R^d)$,
 for every $r \in [0,T]$. 
Then, for any $\mu\in\Pi(m,m')$,
\begin{eqnarray}
\label{majorL2norm}
\Vert u^m(t,\cdot)-u^{m'}(t,\cdot)\Vert_2^2
&=&\int_{\R^d} \vert u^m(t,y)-u^{m'}(t,y)\vert^2 dy \nonumber \\
&= &
\int_{\R^d} \left \vert \int_{\mathcal{C}^d \times\mathcal{C}^d} \Big[ K(y-X_t(\omega)) \, V_t(X(\omega),u^m(X(\omega))) -  \right . \nonumber \\ 
& &   \left . K(y-X_t(\omega'))\,V_t(X(\omega'),u^{m'}(X(\omega'))) \Big] \, d\mu(\omega,\omega') \, \right \vert^2 dy 
\nonumber \\
&\leq & \int_{\R^d} \int_{\mathcal{C}^d \times\mathcal{C}^d} \Big \vert K(y-X_t(\omega)) \, V_t(X(\omega),u^m(X(\omega))) -   \nonumber \\ 
& &  K(y-X_t(\omega')) \, V_t(X(\omega'),u^{m'}(X(\omega'))) \Big \vert^2 \, d\mu(\omega,\omega') \,  dy \nonumber \\
&= & \int_{\mathcal{C}^d \times\mathcal{C}^d} \int_{\R^d} \Big \vert K(y-X_t(\omega)) \, V_t(X(\omega),u^m(X(\omega))) -   \nonumber \\ 
& & K(y-X_t(\omega')) \, V_t(X(\omega'),u^{m'}(X(\omega'))) \Big \vert^2 \, dy \,  d\mu(\omega,\omega') \ ,
\end{eqnarray}
where the third inequality follows by Jensen's and the latter inequality is justified by Fubini theorem.
\\
%
We integrate now both sides of \eqref{eq:KVKV'},  with respect to the state variable  $y$ over $\R^d$, 
for all $(x,x') \in \mathcal{C}^d \times \mathcal{C}^d, (z,z') \in \mathcal{C} \times \mathcal{C}$,
\begin{eqnarray}
\label{eq:KVKV'L2}
\int_{\R^d} \vert K(y-x_t)V_t(x,z)-K(y-x_t')V_t(x',z')\vert^2\,dy
\leq 2\int_{\R^d} \vert K(y-x_t)-K(y-x_t')\vert^2\vert V_t(x,z)\vert^2\,dy \nonumber \\
+2\int_{\R^d} \vert V_t(x,z)-V_t(x',z')\vert^2 \vert K(y-x_t')\vert^2\,dy.
\end{eqnarray}
We remark now that by classical properties of 
 Fourier transform, since $K \in L^2(\R^d)$, we have
$$
\forall \; (x,\xi) \in \R^d \times \R^d, \; \shf(K_x)(\xi) = e^{-i \xi \cdot x} \shf(K)(\xi) \ ,
$$
where in this case, the Fourier transform operator $\; \shf \;$ acts from $L^2(\R^d)$ to $L^2(\R^d)$ and $K_x : \bar{y} \in \R^d \mapsto K(\bar{y}-x)$.
Since $K \in L^2(\R^d)$, Plancherel's theorem gives, for all $(\bar{y},x,x') \in \R^d \times \shc^d \times \shc^d$,
\begin{eqnarray}
\label{KK'Fourier}
\int_{\R^d} \vert K(\bar{y}-x_t)-K(\bar{y}-x'_t)\vert^2 d\bar{y}
& = & \int_{\R^d} \vert K_{x_t}(\bar{y})-K_{x_t'}(\bar{y})\vert^2 d\bar{y}   \nonumber \\
& = & \int_{\R^d} \vert e^{-i \xi \cdot x_t} \shf(K)(\xi)-e^{-i \xi \cdot x'_t} \shf(K)(\xi)\vert^2 d \xi \nonumber \\
& = & \int_{\R^d} \vert \shf(K)(\xi) \vert^2 \; \vert e^{-i \xi \cdot x_t}-e^{-i \xi \cdot x'_t} \vert^2 d \xi \nonumber \\ 
& \leq & \int_{\R^d} \vert \shf(K)(\xi) \vert^2 \; \vert \xi \cdot (x_t-x'_t) \vert^2 d \xi \nonumber \\ 
& \leq & \vert x_t-x'_t \vert^2 \int_{\R^d} \vert \shf(K)(\xi) \vert^2 \; \vert \xi \vert^2 d \xi \nonumber \\
& = & \vert x_t-x'_t \vert^2 \int_{\R^d}  \vert \shf (K) (\xi) \xi \vert ^2 d \xi \nonumber \\
& = & \vert x_t-x'_t \vert^2 \int_{\R^d}  \vert \shf ( \nabla K) (\xi) \vert ^2 d \xi \nonumber \\
& = & \vert x_t-x'_t \vert^2 \Vert \nabla K \Vert_{2}^2\ . 
\end{eqnarray}
Injecting this bound into~\eqref{eq:KVKV'L2}, taking into account \eqref{eq:Vmajor2} yields 
\begin{eqnarray}
\label{eq:majorKVKV'}
\int_{\R^d} \vert K(y-x_t)V_t(x,z)-K(y-x'_t)V_t(x',z')\vert^2\,dy
&\leq & 
2\Vert \nabla K \Vert_{2}^2 \,\vert x_t-x'_t\vert^2 \exp(2tM_{\Lambda}) \nonumber \\
&& + \; 2 M_K \vert V_t(x,z)-V_t(x',z')\vert^2 \nonumber \\
&\leq &
2e^{2tM_{\Lambda}} \Vert \nabla K \Vert_{2}^2 \vert x_t-x'_t\vert^2 \nonumber \\
&&+
4M_K L_{\Lambda}^2e^{2t M_{\Lambda}} t \int_0^t \left [\vert x_s-x'_s\vert^2+\vert z_s-z'_s\vert^2\right ]ds \nonumber \\
&\leq & 
 2 e^{2tM_{\Lambda}}(2M_K L_{\Lambda}^2 t^2 + \Vert \nabla K \Vert_{2}^2) \sup_{0 \leq r \leq t}\vert x_r- x'_r \vert^2  \nonumber \\
 && + 4M_K L_{\Lambda}^2e^{2t M_{\Lambda}} t \int_0^t\vert z_s-z'_s\vert ^2\,ds \nonumber \\ 
 & \leq &  \tilde{C}_{K,\Lambda}(t) \left [ \sup_{0 \leq r \leq t}\vert x_r- x'_r \vert^2 +\int_0^t\vert z_s-z'_s\vert ^2\,ds  \right] \ , \nonumber \\
\end{eqnarray}
for all $(x,x') \in \shc^d \times \shc^d$ and $(z,z') \in \shc \times \shc$, with $\tilde{C}_{K,\Lambda}(t) := 2 e^{2tM_{\Lambda}}(2M_K L_{\Lambda}^2 t(t+1) + \Vert \nabla K \Vert_{2}^2 )$. \\
Inserting  \eqref{eq:majorKVKV'} into \eqref{majorL2norm}, 
after substituting $X(\omega)$ with $x$,   $X(\omega')$ with $x'$,
$z$ with $u^m(X(\omega))$ and $z'$ with $u^{m'}(X(\omega'))$,
 for any  $\mu \in \Pi(m,m')$, we obtain the inequality
\begin{eqnarray}
\label{eq:umoinsu'}
\Vert u^m(t,\cdot)-u^{m'}(t,\cdot)\Vert_2^2
&\leq &
\tilde{C}_{K,\Lambda}(t) \left\{ \int_{\mathcal{C}^d \times \mathcal{C}^d}   \sup_{0 \leq r \leq t} \vert X_r(\omega)-X_r(\omega')\vert^2  \, d\mu(\omega,\omega') \right .\nonumber \\
&&+\left .\int_{\mathcal{C}^d \times \mathcal{C}^d} \int_0^t \vert u^m(s,X_s(\omega))-u^{m'}(s,X_s(\omega'))\vert^2\,ds  \, d\mu(\omega,\omega')  \right \}  \ .
\end{eqnarray}
Since  inequality \eqref{eq:uu'} is verified for all $y \in \R^d \, ,s \in [0,T]$, we obtain for all $\omega, \omega' \in \mathcal{C}^d$ 
\begin{eqnarray*} 
\vert u^m\big (s,X_s(\omega)\big )-u^{m'}\big (s,X_s(\omega')\big )\vert^2 & \leq & C_{K,\Lambda}(s)\left [ \vert X_s(\omega)-X_s(\omega')\vert^2 \, + \vert W_s(m,m')\vert ^2\right ]  \\
& \leq & C_{K,\Lambda}(s) \left[ \sup_{0 \le r \le s}\vert X_r(\omega)-X_r(\omega')\vert^2  + \vert W_s(m,m')\vert ^2 \right]\ .
\end{eqnarray*}
Integrating each side of the above inequality with respect to the time variable $s$ and the measure $\mu\in\Pi(m,m')$ and observing that $C_{K,\Lambda}(s)$ is increasing in $s$ yields 
\begin{eqnarray}
\label{eq:32}
I  & := & \int_{\mathcal{C}^d \times \mathcal{C}^d} \int_0^t \vert u^m(s,X_s(\omega))-u^{m'}(s,X_s(\omega'))\vert^2\,ds \,d\mu(\omega,\omega') \nonumber \\
&\leq & C_{K,\Lambda}(t) t \left[ \int_{\mathcal{C}^d \times \mathcal{C}^d} \sup_{0 \le r \le t}\vert X_r(\omega)-X_r(\omega')\vert^2 \, d \mu(\omega,\omega')   +   \vert W_t(m,m')\vert ^2 \right] \ .
\end{eqnarray}
By injecting inequality \eqref{eq:32} in the right-hand side of inequality \eqref{eq:umoinsu'}, we obtain 
\begin{eqnarray}
\label{umoinsu'2}
\Vert u^m(t,\cdot)-u^{m'}(t,\cdot)\Vert_2^2 & \leq & 
\tilde{C}_{K,\Lambda}(t)(1 + tC_{K,\Lambda}(t)) 
 \int_{\mathcal{C}^d \times \mathcal{C}^d}  \sup_{0 \le r \leq t} \vert X_r(\omega)-X_r(\omega')\vert^2 d\mu(\omega,\omega')  \nonumber \\ 
&  &+ t \tilde C_{K,\Lambda}(t) C_{K,\Lambda}(t) \vert W_t(m,m')\vert ^2 \ .
\end{eqnarray}
By taking the infimum over $\mu \in \Pi(m,m')$ on the right-hand side, we obtain
\begin{eqnarray}
\label{eq:majoruu'}
\Vert u^m(t,\cdot)-u^{m'}(t,\cdot)\Vert_2^2 
& \leq &  \tilde{C}_{K,\Lambda}(t)(1+2tC_{K,\Lambda}(t)) \vert W_t(m,m') \vert^2 \ .
\end{eqnarray}

\item {\bf{Proof of \eqref{eq:uu'Linf}}}. \\
By the hypothesis $4.$ in Assumption \ref{ass:main}, $K \in L^1(\R^d)$.
Given a function $g:[0,T] \times \R^d \rightarrow \C, \ 
(s,x) \mapsto g(s,x) $ , we will 
often denote its Fourier transform in the space variable $x$
by $(s, \xi) \mapsto \shf (g) (s,\xi)$ instead of
$\shf g (s,\cdot) (\xi)$.
 Then for $(\bar{\omega},s,\xi) \in \Omega \times [0,T ] \times \R^d$, the Fourier transform of the functions $u^{\eta_{\bar{\omega}}}$ and $u^m$ are given by
\begin{eqnarray}
\label{Four_u1}
  \shf(u^{\eta_{\bar{\omega}}})(s,\xi)  & = & \shf(K)(\xi) \; \int_{\mathcal{C}^d} e^{-i \xi \cdot X_s(\omega)}\exp \left ( \int_0^s \Lambda \big (r,
  X_r(\omega),u^{\eta_{\bar{\omega}}}(r,X_r(\omega))\big ) dr\right ) d\eta_{\bar{\omega}}(\omega)   \\
  \label{Four_u2}
 \shf(u^{m})(s,\xi)  & = & \shf(K)(\xi) \; \int_{\mathcal{C}^d} e^{-i \xi \cdot X_s(\omega)}\exp \left ( \int_0^s \Lambda \big (r,
   X_r(\omega),u^{m}(r,X_r(\omega))\big ) dr\right ) dm(\omega) \ .
\end{eqnarray}
To simplify notations in the sequel, we will often use the convention
$$
V_r^{\nu}(y) := V_r(y,u^{\nu}(y)) = \exp \left ( \int_0^r \Lambda \big ( \theta,y_{\theta},u^{\nu}(\theta,y_{\theta})\big ) d \theta \right )  \ ,
$$
where $u^{\nu}$ is defined in  \eqref{eq:u}, with $m = \nu$. 
 \\
In this way, relations \eqref{Four_u1} and \eqref{Four_u2} can be re-written as
\begin{eqnarray}
\label{Four_ubis}
  \shf(u^{\eta_{\bar{\omega}}})(s,\xi)  & = & \shf(K)(\xi) \; \int_{\mathcal{C}^d} e^{-i \xi \cdot X_s(\omega)} V_s^{\eta_{\bar{\omega}}}(X(\omega)) d\eta_{\bar{\omega}}(\omega) 
\nonumber  \\
&& \\
 \shf(u^{m})(s,\xi)  & = & \shf(K)(\xi) \; \int_{\mathcal{C}^d} e^{-i \xi \cdot X_s(\omega)} V_s^m(X(\omega)) dm(\omega) \ , \nonumber
\end{eqnarray}
for $(\bar{\omega},s,\xi) \in \Omega \times [0,T] \times \R^d$. \\
For a function $f \in L^1(\R^d)$ such that $\shf(f) \in L^1(\R^d)$, the inversion formula of the Fourier transform is valid and implies 
\begin{eqnarray}
\label{E330}
f(x) = \frac{1}{\sqrt{2 \pi}} \int_{\R^d} \shf(f)(\xi) e^{i \xi \cdot x} d \xi, \ x \in \R^d \ .
\end{eqnarray} 
$f$ is obviously bounded and continuous taking into account  Lebesgue dominated convergence theorem.
Moreover 
\begin{equation} \label{F1infty}
 \Vert f \Vert_\infty \le \frac{1}{\sqrt{2 \pi}} \Vert \shf (f) \Vert_1,  
\end{equation}
where we remind that  $\Vert \cdot \Vert_1$ denotes the $L^1(\R^d)$-norm.
As $\shf(K)$ belongs to $L^1(\R^d)$, from \eqref{F1infty} applied to the function $
f = u^{\eta_{\bar{\omega}}}(s, \cdot)-u^m(s,\cdot)$ with fixed $\bar{\omega} \in \Omega, s \in [0,T]$, we get
\begin{eqnarray}
\label{majorLinf}
\E{[ \Vert u^{\eta}(s,\cdot) - u^m(s,\cdot) \Vert^2_{\infty} ] } & \leq & 
 \frac{1}{\sqrt{2 \pi}} \E{[\Vert \shf(u^{\eta})(s,\cdot) - \shf(u^m)(s,\cdot)  \Vert_1^2 ] } \nonumber \\
& \leq & \frac{1}{\sqrt{2 \pi}} \E { \left[\left( \int_{\R^d} \vert \shf(u^{\eta_{\bar{\omega}}})(s,\xi) - \shf(u^m)(s,\xi) \vert d \xi   \right)^2 \right] } \ ,
\end{eqnarray}
where we recall that $\E$ is taken w.r.t. to $d \mathbb{P}(\bar{\omega})$.

The terms intervening in the expression above are measurable.
This can be justified by Fubini-Tonelli theorem and
the fact that $(\bar{\omega},s,x) \mapsto u^{\eta_{\bar{\omega}}}(s,x)$ is measurable from $(\Omega \times [0,T] \times \R^d, \shf \otimes \shb([0,T]) \otimes \shb(\R^d))$ to $(\R,\shb(\R))$.
We prove the latter point.
By  item 3. of this Lemma, we recall that the function $(m,t,x) \mapsto u^m(t,x)$ is continuous on $\shp(\shc^d) \times [0,T] \times \R^d$ and so measurable from $(\shp(\shc^d) \times [0,T] \times \R^d, \shb(\shp(\shc^d)) \otimes \shb([0,T]) \otimes \shb(\R^d))$ to $(\R,\shb(\R))$. The application $(\bar{\omega},t,x) \mapsto (\eta_{\bar{\omega}},t,x)$ being measurable from $(\Omega \times [0,T] \times \R^d,\shf \otimes \shb([0,T] \otimes \shb(\R^d)))$ to $(\shp(\shc^d) \otimes \shb([0,T]) \otimes \shb(\R^d))$, by composition the map 
$(\bar{\omega},s,x) \mapsto u^{\eta_{\bar{\omega}}}(s,x)$ is measurable.
By Fubini-Tonelli theorem 
 $(\bar{\omega},s,\xi) \mapsto \shf(u^{\eta_{\bar{\omega}}})(s,\xi))$ is measurable from $(\Omega \times [0,T] \times \R^d, \shf \otimes \shb([0,T]) \otimes \shb(\R^d)$ to $(\C,\shb(\C))$ and
$(s,\xi) \mapsto u^m(s,\xi)$ is measurable from $([0,T]\times \R^d, \shb([0,T] \otimes \R^d) $ to  $(\R,\shb(\R)$.

We are now ready to bound the right-hand side of \eqref{majorLinf}.
 For all $(\bar{\omega},s) \in \Omega \times [0,T ]$, by \eqref{Four_ubis}
\begin{eqnarray} 
\label{E331}
\vert \shf(u^{\eta_{\bar{\omega}}})(s,\xi) - \shf(u^m)(s,\xi) \vert & \leq & \left \vert \shf(K)(\xi) \right \vert \left \vert \int_{\mathcal{C}^d} e^{-i \xi \cdot X_s(\omega)} V_s^{\eta_{\bar{\omega}}}(X(\omega)) d\eta_{\bar{\omega}}(\omega) - \int_{\mathcal{C}^d} e^{-i \xi \cdot X_s(\omega)} V_s^{\eta_{\bar{\omega}}}(X(\omega)) dm(\omega)   \right \vert \nonumber \\
&& + \left \vert \shf(K)(\xi) \right \vert  \left \vert \int_{\mathcal{C}^d} e^{-i \xi \cdot X_s(\omega)} V_s^{\eta_{\bar{\omega}}}(X(\omega)) dm(\omega) - \int_{\mathcal{C}^d} e^{-i \xi \cdot X_s(\omega)} V_s^{m}(X(\omega)) dm(\omega)   \right \vert , \nonumber \\
\end{eqnarray}
which implies
\begin{eqnarray}
\label{decomp_AB}
\left( \int_{\R^d} \vert \shf(u^{\eta_{\bar{\omega}}})(s,\xi) - \shf(u^m)(s,\xi) \vert d \xi \right)^2 & \leq & \left( \int_{\R^d} \vert \shf(K)(\xi) \vert \vert A_{s,\bar{\omega}}(\xi) \vert d\xi + \int_{\R^d} \vert \shf(K)(\xi) \vert \vert B_{s,\bar{\omega}}(\xi) \vert d\xi \right)^2 \nonumber \\
& \leq & 2(I^1_{s,\bar{\omega}} + I^2_{s,\bar{\omega}}) \ ,
\end{eqnarray}
where
\begin{equation}
\label{def_I}
\left \{
\begin{array}{l}
I^1_{s,\bar{\omega}} := \left( \int_{\R^d} \vert \shf(K)(\xi) \vert \vert A_{s,\bar{\omega}}(\xi) \vert d\xi \right)^2 \\
I^2_{s,\bar{\omega}} := \left( \int_{\R^d} \vert \shf(K)(\xi) \vert \vert B_{s,\bar{\omega}}(\xi) \vert d\xi \right)^2 \ ,
\end{array}
\right .
\end{equation}
and for all $\bar{\omega} \in \Omega$, $s \in [0,T]$ 
\begin{equation}
\label{def_AB}
\left \{
\begin{array}{lll}
A_{s,\bar{\omega}}(\xi) := \int_{\mathcal{C}^d} e^{-i \xi \cdot X_s(\omega)} V_s^{\eta_{\bar{\omega}}}(X(\omega)) d\eta_{\bar{\omega}}(\omega) - \int_{\mathcal{C}^d} e^{-i \xi \cdot X_s(\omega)} V_s^{\eta_{\bar{\omega}}}(X(\omega)) dm(\omega) \\
B_{s,\bar{\omega}}(\xi) := \int_{\mathcal{C}^d} e^{-i \xi \cdot X_s(\omega)} V_s^{\eta_{\bar{\omega}}}(X(\omega)) dm(\omega) - \int_{\mathcal{C}^d} e^{-i \xi \cdot X_s(\omega)} V_s^{m}(X(\omega)) dm(\omega) \ .
\end{array}
\right .
\end{equation} 
We observe that  $(\bar{\omega},s, \xi) \mapsto A_{s,\bar{\omega}}(\xi)$
 and $(\bar{\omega},s,\xi) \mapsto B_{s,\bar{\omega}}(\xi)$ are measurable.
Indeed, the map \\
$(\omega,\bar \omega, \xi) \mapsto 
e^{-i \xi \cdot X_s(\omega)} V_s^{\eta_{\bar{\omega}}}(X(\omega))$
is Borel. 
By Remark \ref{RMeasure} we can easily show that
$(\bar \omega, s,\xi)  \mapsto \eta_{\bar \omega}(\omega) $
is (still) a random measure when $\Omega$ is replaced by  
 $[0,T] \times \R^d \times \Omega$.
  Proposition 3.3, Chapter 3. of \cite{crauel} tell us that 
$ (\bar \omega, s,\xi) \mapsto \int_{\mathcal{C}^d} e^{-i \xi \cdot X_s(\omega)} V_s^{\eta_{\bar{\omega}}}(X(\omega)) d\eta_{\bar{\omega}}(\omega)$ is measurable.
By use of Fubini's theorem mentioned, measurability of $A, B$ follows.

Regarding $A_{s,\bar{\omega}}$, let $\varphi_{s,\xi,\bar{\omega}}$ denote the function defined by $y \in \shc^d \mapsto e^{-i \xi \cdot y_s} V_s^{\eta_{\bar{\omega}}}(y)$. Then, one can write  $ A_{s,\bar{\omega}} = \langle \eta_{\bar{\omega}} - m , \varphi_{s,\xi,\bar{\omega}} \rangle $, where $\langle \cdot , \cdot \rangle $ denotes the pairing between measures and bounded, continuous functionals.
 $\varphi_{s,\xi,\bar{\omega}}$ is clearly  bounded by $e^{sM_{\Lambda}}$; inequalities \eqref{eq:Vmajor2} and \eqref{eq:uu'} imply the continuity of $\varphi_{s,\xi,\bar{\omega}}$ on $(\shc^d,\Vert \cdot \Vert_{\infty})$, for fixed $(\bar{\omega},s,\xi) \in \Omega \times [0,T] \times \R^d$.
By Cauchy-Schwarz inequality we obtain for all $\bar{\omega} \in \Omega$, $s \in [0,T]$
\begin{eqnarray}
\label{major_I1}
I^1_{s,\bar{\omega}} 
& \leq &  \Vert \shf(K) \Vert_{1} \left(\int_{\R^d} \vert \shf(K)(\xi) \vert \vert A_{s,\bar{\omega}} \vert^2 d\xi \right) \nonumber \\
& \leq & \Vert \shf(K) \Vert_{1} \left(\int_{\R^d} \vert \shf(K)(\xi) \vert  \vert \langle \eta_{\bar{\omega}} - m , \varphi_{s,\xi,\bar{\omega}} \rangle \vert^2 d\xi \right) \ .
\end{eqnarray}
Since the right-hand side of \eqref{major_I1} is measurable, taking expectation w.r.t. $d \P(\bar{\omega})$ in both sides yields
 \begin{eqnarray}
 \label{major_I1bis}
 \E{[I^1_s]} & \leq & \Vert \shf(K) \Vert_{1} \left(\int_{\R^d} \vert \shf(K)(\xi) \vert \;  \E[ \vert \langle \eta - m , \varphi_{s,\xi,\cdot} \rangle \vert^2 ] \; d\xi \right) \nonumber \\
 & \leq & e^{2sM_{\Lambda}} \Vert \shf(K) \Vert_{1} \left(\int_{\R^d} \vert \shf(K)(\xi) \vert \; \sup_{\underset{\Vert \varphi \Vert_{\infty} \leq 1}{\varphi \in \shc_b(\shc^d)}} \E{[\vert \langle \eta - m , \varphi \rangle \vert^2]} d\xi \right) \nonumber \\
  & \leq & e^{2sM_{\Lambda}} \Vert \shf(K) \Vert_{1}^2  \sup_{\underset{\Vert \varphi \Vert_{\infty} \leq 1}{\varphi \in \shc_b(\shc^d)}} \E{[\vert \langle \eta - m , \varphi \rangle \vert^2]} \ .
 \end{eqnarray}
Concerning the second term $B_{s,\bar{\omega}}$, for all $(s,\xi) \in [0,T] \times \R^d$, 
\begin{eqnarray}
\label{BoundB}
\vert B_{s,\bar{\omega}}(\xi) \vert^2 & = & \left | \int_{\mathcal{C}^d} e^{-i \xi \cdot X_s(\omega)} \big( V_s^{\eta_{\bar{\omega}}}(X(\omega)) - V_s^{m}(X(\omega)) \big) dm(\omega) \right|^2 \nonumber \\
& \leq &  \int_{\mathcal{C}^d} \left | V_s^{\eta_{\bar{\omega}}}(X(\omega)) - V_s^{m}(X(\omega)) \right|^2  dm(\omega) \nonumber \\
& \leq & e^{2sM_{\Lambda}} L_{\Lambda}^2 \int_{\mathcal{C}^d} \left |\int_0^s u^{\eta_{\bar{\omega}}}(r,X_r(\omega)) - u^{m}(r,X_r(\omega))  dr \right|^2  dm(\omega)  \quad \textrm{, by } \eqref{eq:Vmajor2} \nonumber \\
& \leq & se^{2sM_{\Lambda}} L_{\Lambda}^2 \int_{\mathcal{C}^d} \int_0^s \left|u^{\eta_{\bar{\omega}}}(r,X_r(\omega)) - u^{m}(r,X_r(\omega))\right|^2 dr \; dm(\omega) \nonumber \\
& \leq & se^{2sM_{\Lambda}} L_{\Lambda}^2  \int_0^s  \left \Vert u^{\eta_{\bar{\omega}}}(r,\cdot) - u^{m}(r,\cdot)\right \Vert_{\infty}^2 dr \ ,
\end{eqnarray}
where we remind that functions $(r,x,\bar{\omega}) \in [0,T] \times \R^d \times \Omega \mapsto u^{\eta_{\bar{\omega}}}(r,x)$ and $(r,x) \in [0,T] \times \R^d \mapsto u^{m}(r,x) $ are uniformly bounded. \\
Taking into account \eqref{BoundB}, the measurability of the function $(\bar{\omega},r) \in \Omega \times [0,T] \mapsto \Vert u^{\eta_{\bar{\omega}}}(r,\cdot) - u^{m}(r,\cdot) \Vert_{\infty}^2$ and the Fubini theorem imply 
\begin{eqnarray}
\label{Bound_I2} 
\E{[I^2_s]} & \leq & \E \left[ \left( \int_{\R^d} \vert \shf(K)(\xi) \vert \; \sup_{\xi \in \R^d} \vert B_{s,\cdot}(\xi) \vert d\xi \right)^2 \right] \nonumber \\
& \leq & \E[ \sup_{\xi \in \R^d} \vert B_{s,\cdot}(\xi) \vert^2 \; \Vert \shf(K) \Vert_1^2 ] \nonumber \\
& \leq & se^{2sM_{\Lambda}} L_{\Lambda}^2 \Vert \shf(K) \Vert_{1}^2  \int_0^s  \E{ [ \left \Vert u^{\eta}(r,\cdot) - u^{m}(r,\cdot) \right \Vert_{\infty}^2 ] } dr \ .
\end{eqnarray}
Taking the expectation of both sides in \eqref{decomp_AB},    we inject
 \eqref{major_I1bis} and \eqref{Bound_I2} in the expectation of the right-hand side of \eqref{decomp_AB}
so that \eqref{majorLinf} gives  for all $s \in [0,T]$
\begin{eqnarray}
\label{major_I}
\E{\left[ \Vert u^{\eta}(s,\cdot) - u^m(s,\cdot) \Vert_{\infty}^2 \right]} 
& \leq & C_2(s)\int_0^s  \E{ [ \left \Vert u^{\eta}(r,\cdot) - u^{m}(r,\cdot) \right \Vert_{\infty}^2 ] } dr  \nonumber \\ 
&  & +C_1(s) \sup_{\underset{\Vert \varphi \Vert_{\infty} \leq 1}{\varphi \in \shc_b(\shc^d)}} \E{[\vert \langle \eta - m , \varphi \rangle \vert^2]}, 
\end{eqnarray}
where $C_1 (s):=  \frac{1}{\sqrt{2 \pi}} e^{sM_{\Lambda}} \Vert \shf(K) \Vert_{1}^2$ and 
$C_2(s) :=  \frac{1}{\sqrt{2 \pi}} s e^{2sM_{\Lambda}} L_{\Lambda}^2 \Vert \shf(K) \Vert_{1}^2$. On the one hand, since the functions $u^{\eta}$ and $u^{m}$ are uniformly bounded, $\E{[\Vert u^{\eta}(s,\cdot) - u^m(s,\cdot) \Vert_{\infty}}^2] $ is finite. On the other hand, observing that $a \mapsto C_1(a)$ and $a \mapsto C_2(a)$ are increasing, we have for all $s \in ]0,T], a \in [0,s]$
$$
\E{\left[ \Vert u^{\eta}(a,\cdot) - u^m(a,\cdot) \Vert_{\infty}^2 \right]} \leq C_2(s)\int_0^a  \E{ [ \left \Vert u^{\eta}(r,\cdot) - u^{m}(r,\cdot) \right \Vert_{\infty}^2 ] } dr + C_1(s) \sup_{\underset{\Vert \varphi \Vert_{\infty} \leq 1}{\varphi \in \shc_b(\shc^d)}} \E{[\vert \langle \eta - m , \varphi \rangle \vert^2]} \ .
$$
By Gronwall's lemma, we finally obtain
\begin{eqnarray}
\label{majorFinal}
\forall \; s \in [0,T], \;\E{\left[ \Vert u^{\eta}(s,\cdot) - u^m(s,\cdot) \Vert_{\infty}^2 \right]} & \leq & C_1(s) e^{sC_2(s)} \sup_{\underset{\Vert \varphi \Vert_{\infty} \leq 1}{\varphi \in \shc_b(\shc^d)}} \E{[\vert \langle \eta - m , \varphi \rangle \vert^2]} \ .
\end{eqnarray}
\end{itemize}
\end{proof}

To conclude this part, we want to highlight some properties of the function $u^{m}$, which  will be used in Section \ref{SChaos}.  
In fact, the map $(m,t,x) \in \mathcal{P}_2(\mathcal{C}^d) \times [0,T] \times \R^d \mapsto u^m(t,x)$ has an important non-anticipating property. We begin by defining the notion of induced measure.
For the rest of this section, we fix $t \in  [0,T]$, $m_t \in \mathcal{P}(\mathcal{C}_t^d)$. 

\begin{defi} \label{D38}
Given a non-negative Borel measure $m$ on $(\mathcal{C}^d,\shb(\shc^d))$. 
From now on,  $m_t$ will denote 
 the (unique) induced measure on $(\mathcal{C}_t^d,\shb(\shc^d_t))$ (with $\mathcal{C}_t^d := \mathcal{C}([0,t],\R^d)$) by
$$
\int_{\mathcal{C}_t^d} F(\phi)m_t(d\phi) = \int_{\mathcal{C}^d} F(\phi_{|_{[0,t]}})m(d\phi),
$$
where $F : \mathcal{C}_t^d \longrightarrow \R$ is bounded and continuous.
\end{defi}

\begin{rem} \label{R38}
Let $t \in  [0,T], m = \delta_{\xi} \;, \xi \in \mathcal{C}^d$. 
The induced measure, $m_t$, on $\shc_t^d$ is $\delta_{(\xi_r | 0 \leq r \leq t)} $.
\end{rem}
 The same construction as the one carried on in Lemma \ref{lem:u} allows us to define the unique solution to
\begin{equation}
\label{u_mt}
\begin{array}{l}
u^{m_t}(s,y) = \int_{\mathcal{C}_t^d} K(y-X_s(\omega)) 
\exp\left(\int_0^s \Lambda(r,X_r(\omega),u^{m_t}(r,X_r(\omega))) dr \right)
m_t(d\omega) \quad  \ \forall s\in [0,t]\ .
\end{array}
\end{equation}
\begin{prop}
\label{non-anticip}
Under the assumption of Lemma \ref{lem:u}, we have 
$$
\forall (s,y) \in [0,t] \times \R^d, \; u^{m}(s,y) = u^{m_t}(s,y).
$$
\end{prop}

\begin{proof}
By definition of $m_t$, it follows that $u^m(s,y)|_{[0,t] \times \R^d}$ is a solution of \eqref{u_mt}. Invoking the uniqueness of \eqref{u_mt} ends the proof.
\end{proof}
\begin{corro} \label{Canticip}
Let $N \in \N$, $\xi^1, \cdots, \xi^i, \cdots, \xi^N$ be $(\mathcal{G}_t)$-adapted continuous processes, where $\mathcal{G}$ is a filtration 
(defined on some probability space) 
fulfilling the usual conditions. Let $m(d\omega) = \frac{1}{N} \sum_{i=1}^{N} \delta_{\xi^i}(d\omega)$.
Then, $(u^m(t,y))$ is a $(\mathcal{G}_t)$-adapted random field, i.e.
for any $(t,y) \in [0,T] \times \R^d$, the process is $(\mathcal{G}_t)$-adapted.

\end{corro}

\subsection{Existence and uniqueness of the solution to the stochastic differential equations}

\label{S32}

For a given $m\in \mathcal{P}_2(\mathcal{C}^d)$, $u^m$ is well-defined according to Lemma~\ref{lem:u}.
 Let $Y_0 \sim \zeta_0$.  Then we can consider the SDE 
\begin{equation}
 \label{eq:Ym}
 \begin{array}{l}
 Y_t=Y_0+\int_0^t\Phi(s,Y_s,u^m(s,Y_s))dW_s+\int_0^t g(s,Y_s,u^m(s,Y_s))ds\ ,\quad \textrm{for any}\ t\in [0,T],
 \end{array}
 \end{equation}
which constitutes the first equation of \eqref{eq:NSDE}.
 Thanks to Assumption~\ref{ass:main} and Lemma~\ref{lem:uu'} implying the Lipschitz property of $u^m$ w.r.t. the space variable (uniformly in time),  \eqref{eq:Ym} admits a unique strong solution $Y^m$. 
We define the application $\Theta :\mathcal{P}_2(\mathcal{C}^d)\rightarrow \mathcal{P}_2(\mathcal{C}^d)$ such that $\Theta (m):=\mathcal{L}(Y^m)$. 
The aim of the present section is to prove, following Sznitman~\cite{sznitman},
 by a fixed point argument on $\Theta $ the following result. 
\begin{thm}
\label{prop:NSDE}
Under Assumption~\ref{ass:main}, the 
McKean type SDE~\eqref{eq:NSDE} admits the following properties.
\begin{enumerate}
\item Strong existence and pathwise uniqueness;
\item existence and uniqueness in law.
\end{enumerate} 
\end{thm}
The proof of the theorem needs the lemma below.
Given two reals $a, b$ we will denote in the sequel
$a \wedge b := {\rm min}(a,b)$.
\begin{lem}
\label{lem:yy'}
Let $r:\,[0,T]\mapsto [0,T]$ be a non-decreasing function such that $r(s)\leq s$ for any $s\in [0,T]$. 
Let $\shu : (t,y) \in [0,T] \times \shc^d \rightarrow \R $ (respectively $\shu' : (t,y) \in [0,T] \times \shc^d \rightarrow \R$),
 be a given Borel function  
such that for all $t \in [0,T]$, there is a Borel map $\shu_t: \shc ([0,t],\R^d) \rightarrow \R$ (resp. $\shu^{'}_t: \shc ([0,t],\R^d) 
\rightarrow \R$) such that $\shu(t,\cdot) = \shu_t(\cdot)$ (resp. $\shu'(t,\cdot) = \shu'_t(\cdot)$).

Then the following two assertions hold. 
\begin{enumerate}
\item Consider $Y$ (resp. $Y'$) a solution of the following SDE for $v=\shu$ (resp. $v=\shu'$):
\begin{equation}
 \label{eq:YY'}
 \begin{array}{l}
 Y_t=Y_0+\int_0^t\Phi(r(s),Y_{r(s)},v(r(s),Y_{\cdot \wedge r(s)}))dW_s+\int_0^t g(r(s),Y_{r(s)},v(r(s),Y_{\cdot \wedge r(s)}))ds\ ,\quad \textrm{for any}\ t\in [0,T] \ ,
 \end{array}
 \end{equation}
 where, we emphasize that for all $\theta \in [0,T]$, $Z_{\cdot \wedge \theta} := \{Z_u, 0 \leq u \leq \theta \} 
\in \shc([0,\theta],\R^d)$ for any continuous process $Z$. 
For any $ a \in [0,T]$, we have
\begin{equation}
 \label{eq:YY'Stab}
\E [\sup_{t\leq a} \vert Y'_t-Y_t\vert ^2]
\leq 
C_{\Phi,g}(T)\E\left [\int_0^a \vert \shu(r(t),Y_{\cdot \wedge r(t)})-\shu'(r(t),Y'_{\cdot \wedge r(t)})
\vert^2dt \right ]\ ,
\end{equation}
where $C_{\Phi,g}(T)=12(4L^2_{\Phi}+TL^2_g)e^{12T(4L^2_{\Phi}+TL^2_g)}$.
\item Suppose moreover that $\Phi$ and $g$ are globally Lipschitz w.r.t. the time and space variables i.e. there exist some positive constants $L_{\Phi}$ and $L_g$ such that for any $(t,t',y,'y',z,z')\in [0,T]^2\times \R^{2d}\times \R^2$ 
\begin{equation}
\label{eq:PhigLipt}
\left \{
\begin{array}{l}
\vert \Phi(t,y,z)-\Phi(t',y',z')\vert \leq L_{\Phi} (\vert t-t'\vert +\vert y-y'\vert +\vert z-z'\vert)\\
\vert g(t,y,z)-g(t',y',z')\vert \leq L_{g} (\vert t-t'\vert +\vert y-y'\vert +\vert z-z'\vert)\ .
\end{array}
\right . 
\end{equation}
Let  $r_1, r_2:\,[0,T]\mapsto [0,T]$ being two non-decreasing functions verifying $r_1(s)\leq s$ and $r_2(s)\leq s$ for any $s\in [0,T]$. 
Let $Y$ (resp. $Y'$) be a solution  of~\eqref{eq:YY'} for $v=\shu$ and $r=r_1$ (resp. $v=\shu'$ and $r=r_2$). 
Then for any $ a \in [0,T]$, the following inequality holds: 
\begin{eqnarray}
 \label{eq:YY'Stabr'}
\E [\sup_{t\leq a} \vert Y'_t-Y_t\vert ^2]
&\leq &
C_{\Phi,g}(T)\left (\Vert r_1-r_2\Vert_{2}^2 +\int_0^a\E[\vert Y'_{r_1(t)}-Y'_{r_2(t)}\vert^2] dt\right .\nonumber \\
&&
\left .+\E \left[\int_0^a\vert  \shu(r_1(t),Y_{\cdot \wedge r_1(t)})-\shu'(r_2(t),Y'_{\cdot \wedge r_2(t)})\vert^2dt´\right]\right )\ ,
\end{eqnarray}
where $\Vert \cdot \Vert_2$ is the $L^2([0,T])$-norm.  
\end{enumerate}
\end{lem}
\begin{proof}
\begin{enumerate}
\item Let us consider the first assertion of Lemma~\ref{lem:yy'}. 
 Let $Y$ (resp. $Y'$) is solution of~\eqref{eq:YY'} with associated function $\shu$ (resp.  $\shu'$).
	Let us fix $a \in ]0,T]$. We have
	\begin{equation}
\label{eq:ab}
Y_{\theta}-Y'_{\theta}  =  \alpha_\theta + \beta_\theta, \ \theta \in [0,a],
\end{equation}
where 
\begin{eqnarray*}
\alpha_{\theta} & := & \int_0^{\theta} \left ( \Phi(r(s),Y_{r(s)},\shu(r(s),Y_{\cdot \wedge r(s)}))-
\Phi(r(s),Y'_{r(s)},\shu'(r(s),Y'_{\cdot \wedge r(s)}))\right )dW_s \\
\beta_{\theta} & := & \int_0^{\theta} \left ( g(r(s),Y_{r(s)},\shu(r(s),Y_{\cdot \wedge r(s)}))-g(r(s),Y'_{r(s)},\shu'(r(s),Y'_{\cdot \wedge r(s)}))\right )ds\ .
\end{eqnarray*}
By Burkholder-Davis-Gundy (BDG) inequality, we obtain 
\begin{eqnarray}
\label{eq:alpha}
\E \sup_{ \theta \leq a} \vert \alpha_{\theta} \vert ^2
&\le & 4\E \left[\int_0^a\left \vert  \Phi(r(s),Y_{r(s)},\shu(r(s),Y_{\cdot \wedge r(s)}))-\Phi(r(s),Y'_{r(s)},\shu'(r(s),Y'_{\cdot \wedge r(s)}))\right  \vert ^2ds\right] 
\nonumber \\
&= & 4\int_0^a\E\left[\left \vert  \Phi(r(s),Y_{r(s)},\shu(r(s),Y_{\cdot \wedge r(s)}))
-\Phi(r(s),Y'_{r(s)},\shu'(r(s),Y'_{\cdot \wedge r(s)}))\right  \vert ^2\right]ds \nonumber\\
&\leq & 8L^2_{\Phi}\int_0^a\E\left[\left \vert  \shu(r(s),Y_{\cdot \wedge r(s)})-\shu'(r(s),Y'_{\cdot \wedge r(s)})\right  \vert ^2\right]ds+8L^2_{\Phi}\int_0^a \E\left [\vert Y_{r(s)}-Y'_{r(s)}\vert^2 \right ]\,ds \ . \nonumber \\
\end{eqnarray}
Concerning $\beta$ in \eqref{eq:ab}, 
by Cauchy-Schwarz inequality, we get 
\begin{eqnarray}
\label{eq:beta}
\E \sup_{ \theta \leq a} \vert \beta_{\theta} \vert ^2
&\leq & a \E\left [\int_0^a \vert g(r(s),Y_{r(s)},\shu(r(s),Y_{\cdot \wedge r(s)}))-g(r(s),Y'_{r(s)},\shu'(r(s),Y'_{\cdot \wedge r(s)}))\vert ^2ds \right ]\nonumber\\
&\leq & 2a L_g^2\E\left [\int_0^a \vert \shu(r(s),Y_{\cdot \wedge r(s)})-\shu'(r(s),Y'_{\cdot \wedge r(s)})\vert ^2 ds \right ]+2aL^2_{g}\int_0^a \E\left [\vert Y_{r(s)}-Y'_{r(s)}\vert^2 \right ]\,ds\ . \nonumber \\
\end{eqnarray}
Gathering~\eqref{eq:beta} together with~\eqref{eq:alpha} and using the fact that $r(s)\leq s$, implies 
\begin{eqnarray*}
\E[\sup_{ \theta \leq a}\vert Y'_{\theta}-Y_{\theta}\vert^2]
&\leq &
4(4L^2_{\Phi}+TL^2_g)\left (\E[\int_0^a\vert \shu(r(s),Y_{\cdot \wedge r(s)})-\shu'(r(s),Y'_{\cdot \wedge r(s)})\vert^2 ds] \right . \nonumber \\
& & \left . + \; \int_0^a \E [\vert Y_{r(s)}-Y'_{r(s)}\vert^2  ]\,ds\right  )\\
&\leq &
4(4L^2_{\Phi}+TL^2_g)\left (\E[\int_0^a\vert \shu(r(s),Y_{\cdot \wedge r(s)})-\shu'(r(s),Y'_{\cdot \wedge r(s)})\vert^2 ds] \right . \nonumber \\
& & \left . + \;  \int_0^a \E [\sup_{ \theta \leq s}\vert Y_{\theta}-Y'_{\theta}\vert^2  ]\,ds\right  )\ ,
\end{eqnarray*}
for any $a \in [0,t]$. \\
We conclude the proof by applying Gronwall's lemma.

\item Consider now the second assertion of Lemma~\eqref{lem:yy'}. Following the same lines as the proof of assertion 1. and using the Lipschitz property of $\Phi$ and $g$ w.r.t. to both the time and space variables~\eqref{eq:PhigLipt}, we obtain the inequality
\begin{eqnarray*}
\E[\sup_{t\leq a}\vert Y'_t-Y_t\vert^2]
&\leq &
12(4L^2_{\Phi}+TL^2_g)\left (\int_0^a\vert r_1(t)-r_2(t)\vert^2 dt+\int_0^a\E[\vert Y'_{r_1(t)}-Y'_{r_2(t)}\vert^2] dt\right .\\ 
&&+\E[\int_0^a\vert  \shu(r_1(t),Y_{\cdot \wedge r_1(t)})-\shu'(r_2(t),Y'_{\cdot \wedge r_2(t)})\vert^2dt]
\left .+  \int_0^a \E [\vert Y_{r_1(t)}-Y'_{r_1(t)}\vert^2  ]\,dt\right  )\\
&\leq &
12(4L^2_{\Phi}+TL^2_g)\left (\Vert r_1-r_2\Vert_{2}^2 +\int_0^a\E[\vert Y'_{r_1(t)}-Y'_{r_2(t)}\vert^2] dt\right .\\ 
&&+\E[\int_0^a\vert  \shu(r_1(t),Y_{\cdot \wedge r_1(t)})-\shu'(r_2(t),Y'_{\cdot \wedge r_2(t)})\vert^2dt]
\left .+  \int_0^a \E [\sup_{s\leq t}\vert Y_{s}-Y'_{s}\vert^2  ]\,dt\right  )\ .
\end{eqnarray*}
Applying again Gronwall's lemma concludes the proof.
\end{enumerate}
\end{proof}
\begin{proof}[Proof of Theorem \ref{prop:NSDE}]
Let us consider two probability measures $m$ and $m'$ in $\mathcal{P}_2(\mathcal{C}^d)$. We are interested in proving that $\Theta$ is a contraction for the Wasserstein metric. 
Let $u:= u^m, u':= u^{m'}$ be solutions of \eqref{eq:u} related to $m$ and $m'$.
Let $Y$ be the solution of \eqref{eq:Ym} and  $Y'$ be the solution of \eqref{eq:Ym}
with $m'$ replacing $m$. \\
By definition of the Wasserstein metric~\eqref{eq:Wasserstein}
\begin{equation}
\label{eq:Wasserstein2}
\vert W_T(\Theta(m),\Theta(m'))\vert ^2\leq \E [\sup_{t\leq T} \vert Y'_t-Y_t\vert ^2]\ .
\end{equation}
Hence, we control $\vert Y'_t-Y_t\vert$. To this end, we will use Lemma \ref{lem:yy'}.

By usual BDG and Cauchy-Schwarz  inequalities, as  for example Theorem 2.9, Section 5.2, Chapter 5 in \cite{karatshreve}
there exists a positive real 
  $ C_0$ depending on $(T,M_{\Phi},M_{g})$ such that 
$
\E[  \sup_{t\leq T}  \vert Y_{t}\vert ^{2} ] \leq C_0 \left (1 + \E[ |Y_{0}|^{2}] \right) \ .
$

Using Lemma~\ref{lem:yy'} and    Lemma~\ref{lem:uu'} by applying successively inequality~\eqref{eq:YY'Stab} and inequality~\eqref{eq:uu'} yields
\begin{equation}
\label{eq:Y}
\E [\sup_{t\leq a} \vert Y'_t-Y_t\vert ^2]
\leq 
C \left [\int_0^a \E[\sup_{s\leq t}\vert Y'_s-Y_s\vert^2]dt+\int_0^a\vert W_t(m,m')\vert ^2dt\right ]\ ,
\end{equation}
for any $a \in [0,T]$,
where $C=C_{\Phi , g}(T)C_{K,\Lambda}(T)$. \\
Applying Gronwall's lemma to~\eqref{eq:Y} yields 
\begin{equation}
\label{eq:uepsilon1}
\E [ \sup_{t\leq a} |Y_{t} - Y'_{t}|^{2} ]  \leq  
C e^{CT} \int_{0}^{a} \vert W_{s}(m,m')\vert ^2 ds \ .
\end{equation}
Then recalling~\eqref{eq:Wasserstein2}, this finally gives  
\begin{equation}
\label{eq:majorfinal}
\vert W_{a}(\Theta(m),\Theta(m'))\vert ^2  \leq  Ce^{CT} \int_{0}^{a} 
\vert W_{s}(m,m')\vert^2 ds, \ a \in [0,T].
\end{equation}
We end the proof of item 1. by classical fixed point argument, similarly
to the one of Chapter 1, section 1 of Sznitman~\cite{sznitman}. \\
Concerning item 2. it remains to show uniqueness in law for \eqref{eq:NSDE}.
Let $(Y^1, m^1), (Y^2, m^2)$ be two solutions of \eqref{eq:NSDE} on possibly 
different probability spaces and different Brownian motions, and different
initial conditions distributed according to $\zeta_0$. Given $m \in \shp_2(\shc^d)$, we denote by $\Theta(m)$ the law of $\bar{Y}$, where $\bar{Y}$ is the (strong) solution of 
\begin{eqnarray}
\bar{Y}_t = \bar{Y}_0^1 + \int_0^t \Phi(s,\bar{Y}_s,u^{m}(s,\bar{Y}_s)) dW_s + \int_0^t g(s,\bar{Y}_s,u^{m}(s,\bar{Y}_s)) ds \ ,
\end{eqnarray}
 on the same probability space and same Brownian motion on which $Y^1$ lives. Since $u^{m^2}$ is fixed, $\bar{Y}^2$ is solution of a classical SDE with Lipschitz coefficients for which pathwise uniqueness holds. By Yamada-Watanabe theorem, $Y^2$ and $\bar{Y}^2$ have the same distribution. Consequently, $\Theta(m^2) = \shl(\bar{Y^2}) = \shl(Y^2) = m^2$. It remains to show that $Y^1 = \bar{Y}^2$ in law, i.e. $m^1 = m^2$. By the same arguments as for the proof of $1.$, we get \eqref{eq:majorfinal}, i.e. for all $a \in [0,T]$, 
 $$
 \vert W_{a}(\shl(Y^1),\shl(\bar{Y}^2))\vert ^2 = \vert W_{a}(\Theta(m^1),\Theta(m^2))\vert ^2  \leq  Ce^{CT} \int_{0}^{a} 
 \vert W_{s}(m^1,m^2)\vert^2 ds .
 $$
 Since $\Theta(m^1) = m^1$ and $\Theta(m^2) = m^2$, by Gronwall's lemma $m^1 = m^2$ and finally $Y^1 = \bar{Y}^2$ (in law). This concludes the proof of Proposition \ref{prop:NSDE}. 
\end{proof}

\section{Strong Existence under weaker assumptions}
\label{S4}

\setcounter{equation}{0}

Let us fix a filtered probability space $(\Omega, \shf,(\shf_t)_{t \geq 0}, \P)$ equipped with a $p$ dimensional $(\shf_t)_{t \geq 0}$-Brownian motion $(W_t)_{t \geq 0}$. \\
In this section  Assumption \ref{ass:main2} will be in force.
In particular, we suppose that 
 $\zeta_0$ is a Borel probability measure having a second order moment.





Before proving the main result of this part, we remark  that in this case, uniqueness fails for \eqref{eq:NSDE}.
 To illustrate this, we consider the following counterexample, which is even 
valid for $d=1$.

\begin{example} \label{E41}
 Consider the case $\Phi = g = 0$, $X_{0} = 0$ so that $\zeta_0 =\delta_0$.
This implies that $X_t \equiv 0$ is a strong solution of the first
line of \eqref{eq:NSDE}. Since $u(0,.) = (K * \zeta_{0})(\cdot)$, we have $u(0,\cdot) = K$. \\
A solution  $u$ of the second line equation of \eqref{eq:NSDE}, will be of the
 form
\begin{equation} \label{EqCounter} 
 u(t,y) = K(y)\exp \left( \int_{0}^{t} \Lambda(r,0,u(r,0)) dr \right),
\end{equation}
for some suitable $\Lambda$ fulfilling Assumption \ref{ass:main2}. 2.
 We will in fact consider $\Lambda$ independent of the time and
 $\beta(u) := \Lambda(0,0,u)$.
Without restriction of generality we can suppose $K(0) =1$.
We will show that the second line equation of~\eqref{eq:NSDE} 
is not well-posed for some particular choice of $\beta$. \\
Now  \eqref{EqCounter} becomes
\begin{equation} \label{EqCounter1} 
 u(t,y) = K(y)\exp \left( \int_{0}^{t} \beta(u(r,0))  dr \right).
\end{equation}
By setting $y = 0$, we get $\phi(t) := u(t,0)$ and  in particular,
necessarily we have
  \begin{equation}
\label{eq:example_phi}
\phi(t) = \exp \left( \int_{0}^{t} \beta(\phi(r)) dr \right). 
\end{equation}
A solution $u$ given in  \eqref{EqCounter1} is determined by 
setting $u(t,y) = K(y) \phi(t)$. 
Now, we choose the function $\beta$ such that for given constants 
 $\alpha \in  (0,1)$ and $C > 1$, 
\begin{equation}
\label{eq:example_beta}
\beta(r) = \left \{
\begin{array}{lll}
 & |r - 1|^{\alpha} &\;  ,  \ \textrm{if}\ r \in [0,C] \\
 & |C - 1|^{\alpha} & \;  ,  \ \textrm{if}\ r \geq C  \\
     1  \; & , & \ \textrm{if}\ r \leq 0 \ .\\
\end{array}
\right.
\end{equation}

$\beta$ is clearly a bounded, uniformly continuous function verifying $\beta(1) = 0$ and
 $\beta(r) \ne 0$, for all $r \ne 1$.

We define $F: \mathbb{R} \longrightarrow \mathbb{R}$, by 
$F(u) = \int_{1}^{u} \frac{1}{r \beta(r)} dr $. $F$ is strictly positive
 on $ (1,+\infty)$, and it is a homeomorphism from $[1, +\infty)$ 
to $\R_+$,
 since $ \int_{1}^{+\infty} \frac{1}{r \beta(r)} dr = \infty$. 

On the one hand, by setting $\phi(t) := F^{-1}(t)$, for $t > 0$, we observe
that $\phi$ verifies
 $ \phi'(t)   =  \phi(t) \beta(\phi(t)), t > 0$
and so $\phi$ 
is a solution of \eqref{eq:example_phi}.

On the other hand, the function $\phi \equiv 1$ also satisfies ~\eqref{eq:example_phi}, with the same choice of $\Lambda$,
 related to $\beta$.
This shows
 the non-uniqueness for  the second equation of \eqref{eq:NSDE}. 
\end{example}

The main theorem of this section 
states however the existence (even though non-uniqueness) 
for \eqref{eq:NSDE}, when the coefficients $\Phi$ and $g$ of the SDE
are Lipschitz in $(x,u)$.

\begin{thm} 
\label{thm_semi_weak}
We suppose the validity of Assumption \ref{ass:main2}.
\eqref{eq:NSDE} admits strong existence. 
\end{thm}

The proof goes through several steps. We begin by recalling a lemma,
stated in Problem 4.12 and Remark 4.13 page 64 in \cite{karatshreve}.
%
\begin{lem}
\label{karatzas} 
 Let $(\mathbb{P}_n)_{n \geq 0}$ be a sequence of probability measures on $\mathcal{C}^d$ converging in law to some probability $\mathbb{P}$. Let 
 $(f_n)_{n \geq 0}$ be a uniformly bounded sequence of real-valued, continuous functions defined on $\mathcal{C}^d$, converging uniformly on every 
 compact subset to some continuous $f$. Then 
 ${\displaystyle 
 \int_{\mathcal{C}^d} f_n(\omega)d\mathbb{P}_{n}(\omega) \xrightarrow[\text{$n \rightarrow +\infty$}]{\text{}} \int_{\mathcal{C}^d} f(\omega)d\mathbb{P}(\omega)\ .
 }$
\end{lem}
\begin{rem} \label{Rkaratzas}
We will apply several times Lemma \ref{karatzas}.
We will verify its assumptions 
 showing that
the sequence $(f_n)$ converges uniformly on each bounded ball
of ${\mathcal C}^d$. This will be enough since
every compact of  ${\mathcal C}^d$ is bounded.
\end{rem}

We emphasize that the hypothesis of uniform convergence in Lemma \ref{karatzas} is crucial,
see Remark \ref{counter-example}.

We formulate below an useful Remark, which follows by a simple application
of Lebesgue dominated convergence theorem. It will be often used in the sequel.
\begin{rem}
\label{continuite}
Let  $\Lambda : [0,T] \times \R^{d} \times \R \longrightarrow \R$ be
a Borel bounded function such that for almost all $t \in [0,T]$
$\Lambda(t, \cdot, \cdot)$ is continuous. 
 The function $F : [0,T] \times \mathcal{C}^d \times \mathcal{C} \longrightarrow \R$, $x_{0} \in \R$,
 defined by $F(t,y,z) = K(x_{0}-y_{t}) \exp\left( \int_{0}^{t} \Lambda(r,y_r,z_r) dr\right)$ is continuous. 
\end{rem}

\begin{lem}
\label{cvu}
 Let $(\Lambda_{n})_{n\in \N}$ be a sequence of Borel uniformly bounded functions
 defined on $[0,T] \times \R^d \times \R $,
such that for every $n$, $\Lambda_n(t, \cdot, \cdot)$ 
is continuous. 
Assume the
existence of a Borel function $\Lambda: [0,T] \times \R^d \times \R 
\rightarrow \R$ such that,
 for almost all $t \in [0,T]$,
  $\big [\Lambda_{n}(t,.,.) - \Lambda(t,.,.)\big ]  \xrightarrow[\text{$n \rightarrow +\infty$}]{\text{}}0 $, uniformly on each compact of $\R^d \times \R$. 
Let $x_0 \in \R^d$, 
we denote by $F_n, F: [0,T] \times \mathcal C^d \times \mathcal C  \rightarrow \R $, the maps
$$
  F_{n}(t,y,z) := K(x_{0}-y_{t}) \exp\left( \int_{0}^{t} 
\Lambda_{n}(r,y_r,z_r)\right) \quad\textrm{and}\quad 
 F(t,y,z):= K(x_{0}-y_{t}) \exp\left( \int_{0}^{t} \Lambda(r,y_r,z_r)\right).
$$
Then for every $M > 0$, $F_n$ converges to $F$
when $n$ goes to infinity uniformly with respect to
 $(t,y,z) \in [0,T] \times B_d(O,M) \times B_1(O,M) 
$, with $B_k(O,M) := \{ y \in \mathcal{C}^k,
 ||y||_{\infty} := \sup_{u \in [0,T]} |y_u| \leq M \}$ for $k \in \N^{\star}$.
\end{lem}

\begin{proof}
 We want to evaluate $ \displaystyle{ ||F_{n} - F||_{\infty,M} := \sup_{(t,y,z) \in [0,T] \times B_d(O,M) \times B_1(O,M)} |F_{n}(t,y,z) - F(t,y,z)| }$.

Since $(\Lambda_{n})_{n \geq 0}$ are uniformly bounded, there is a constant
$M_\Lambda$ such that
$$
\forall r \in [0,T], \sup_{(y',z') \in B_d(O,M) \times B_1(O,M) }|\Lambda_{n}(r,y'_r,z'_r) - \Lambda(r,y'_r,z'_r)| \leq 2 M_{\Lambda}.
$$

 By use of \eqref{EMajor1}, we obtain for all $(t,y,z) \in [0,T] \times B_d(O,M) \times B_1(O,M)$, 
 \begin{equation}
  |F_{n}(t,y,z) - F(t,y,z)| \leq M_{K} \exp( 
M_{\Lambda}) \int_{0}^{t} \sup_{(y',z') \in B_d(O,M) \times B_1(O,M)}|\Lambda_{n}(r,y'_r,z'_r) - \Lambda(r,y'_r,z'_r)| dr,
 \end{equation}
 which implies
 
 \begin{equation} \label{E535}
  ||F_{n} - F||_{\infty,M} \leq M_{K} \exp(M_{\Lambda}) \int_{0}^{T} \sup_{(y',z') \in B_d(O,M) \times B_1(O,M)}|\Lambda_{n}(r,y'_r,z'_r) - \Lambda(r,y_r',z'_r)| dr.
\end{equation}
 By Lebesgue's dominated convergence theorem, we have 

$$ \int_{0}^{T} \sup_{(y',z') \in B_d(O,M) \times B_1(O,M)}|\Lambda_{n}(r,y'_r,z'_r) - \Lambda(r,y'_r,z'_r)|\; dr \longrightarrow 0 \ ,$$
 which concludes the proof.
\end{proof}

\begin{lem}
\label{limit}
Let $\Lambda_n, \Lambda$ be as stated in Lemma \ref{cvu}. 
 Let $(Y^{n})_{n\in \N}$ be a sequence of continuous processes.
 We set $Z^{n} := u_{n}(.\,,Y^{n})$ where for any $(t,x)\in [0,T]\times \R^d$ 
 %
 \begin{equation} \label{E536} 
 \left\{
 \begin{array}{l}
  u_{n}(t,x) := \int_{\mathcal{C}^d} K(x-X_t(\omega)) \, 
\exp \left \{\int_{0}^{t}\Lambda_{n} \big 
(r,X_r(\omega),u_{n}(r,X_r(\omega))\big )dr\right \} dm^{n}(\omega) \, \\
m^n:=\mathcal{L}(Y^n)\ .
  \end{array}
  \right .
 \end{equation}
 Suppose moreover that $\nu^n := \mathcal{L}((Y^{n},Z^{n}))$ converges in law to some probability 
measure $\nu$ on $\mathcal{C}^d\times\mathcal{C}$. 
 Then $u_{n}$ pointwisely converges as $n$ goes to infinity to some limiting function $u : [0,T] \times \R^d \rightarrow \R$ such that for all $(t,x) \in [0,T] \times \R^d$,
 \begin{equation}
 \label{eq:ulimit}
 u(t,x) := \int_{\mathcal{C}^d  \times \mathcal{C}}
 K(x-X_t(\omega)) \, \exp \left \{\int_{0}^{t}\Lambda 
\big (r,X_r(\omega),X'_r(\omega')\big )dr\right \} d\nu(\omega,\omega') \ .
 \end{equation}
\end{lem}

\begin{rem}
\label{rem_un}
 $(u_{n})_{n \geq 0}$ is uniformly bounded by $M_K \exp(M_\Lambda T)$.
\end{rem}

\begin{proof}
 Observe that  
$ u_{n}(t,x) = \int_{\mathcal{C}^d \times \mathcal{C}} K(x-X_t(\omega)) \,
 \exp \left \{\int_{0}^{t}\Lambda_{n} \big (r,X_r(\omega),X_r'
(\omega')\big )dr\right \} d\nu^{n}(\omega,\omega'). $ 
Let us fix  $t \in [0,T], x \in \R^d$.
 Let us introduce the sequence of real valued functions $(f_n)_{n\in \N}$ and $f$ defined on $\mathcal{C}^d \times \mathcal{C}$ such that  
 $$
 f_{n}(y,z) =  K(x-y_t) \, \exp \left \{\int_{0}^{t}\Lambda_{n} \big 
(r,y_r,z_r\big )dr\right \}  \quad \textrm{and}\quad f(y,z) = K(x-y_t) \, \exp \left \{\int_{0}^{t}\Lambda \big (r,y_r,z_r\big )dr\right \}\ .
$$
By Remark \ref{continuite}, $f_n$ and  $f$ are continuous. \\
By Lemma \ref{cvu}, it follows that $f_{n} \xrightarrow[\text{$n \longrightarrow +\infty$}]{} f$ uniformly on each closed ball 
(and therefore also for each
 compact subset) of $\mathcal{C}^{d} \times \mathcal{C}$.
Then applying Lemma \ref{karatzas} and Remark \ref{Rkaratzas},
  with $\shc^d \times \shc$,
  $\mathbb{P} = \nu$, $\mathbb{P}^n = \nu^n$ allows to conclude.
\end{proof}


 In fact, the pointwise convergence of $(u_{n})_{n \geq 0}$ can be reinforced.

\begin{prop}  \label{cvuoncompact}
Suppose that the same assumptions as in Lemma \ref{limit} hold.
 

Then the sequence $(u_n)$ introduced in Lemma \ref{limit} also converges 
 uniformly to $u: [0,T] \times \R^d \rightarrow \R$ defined in \eqref{eq:ulimit},
on each compact of $[0,T] \times \R^d$.
In particular $u$ is continuous.
\end{prop}

\begin{proof} 
We fix a compact $C$ of $\R^d$. The restrictions of $u_n$ to $[0,T] \times C$
are uniformly bounded. Provided we prove  that
the sequence $u_{n} \vert_{[0,T] \times C}$ is equicontinuous,
 Ascoli-Arzela theorem would imply 
 that 
the set of    restrictions of $u_n$ to $[0,T] \times C$
is relatively compact with respect to uniform convergence norm topology. \\ 
To conclude, given a subsequence $(u_{n_k})$ it is enough to extract
a subsubsequence converging to $u$. 
Since the  set of 
   restrictions of $u_{n_k}$ to $C$
is relatively compact, there is a function $v:[0,T] \times C \rightarrow \R$
to which  $u_{n_k}$ converges uniformly on $[0,T] \times C$.
Since, by Lemma \ref{limit}, $u_n$ converges pointwise to $u$, obviously $v$ coincides with $u$ on $[0,T] \times C$.

It remains to show the equicontinuity of the sequence $(u_n)$
on $[0,T] \times C$. 
We do this below.



Let $ \varepsilon' > 0$. We need to prove that 
$\; \exists \delta,\eta > 0, \forall (t,x),(t',x') \in [0,T] \times C, $
\begin{equation} \label{udecomp1}
|t-t'|<\delta,\;|x-x'|<\eta \Longrightarrow \forall n \in \N,\;|u_{n}(t,x) - u_{n}(t',x')|<\varepsilon'. 
\end{equation}

We start decomposing as follows:

\begin{equation} \label{udecomp}
 |u_{n}(t,x) - u_{n}(t',x')| \le |(u_{n}(t,x) - u_{n}(t,x')) \vert +
\vert (u_{n}(t,x')- u_{n}(t',x'))|.
\end{equation}
As far as the first term 
in the right-hand side of \eqref{udecomp} is concerned, we have

\begin{equation}
 \label{eq:equicontinuite}
 \begin{array}{lll}
   |u_{n}(t,x) - u_{n}(t,x')| & \leq & \int_{\mathcal{C}^d} |K(x-X_t(\omega))-K(x'-X_t(\omega))|
 \exp(M_{\Lambda} T) dm^{n}(\omega), \\
  & \leq & \exp(M_{\Lambda}T) L_K|x-x'|,
 \end{array}
\end{equation}
where the constant $M_{\Lambda}$
 is an uniform upper bound of $(\vert \Lambda_n\vert, n \geq 0)$. \\
We choose $\eta = \frac{\varepsilon'}{3 \exp (M_\Lambda T) 
L_K}$ to obtain that 
\begin{equation} \label{Exprime}
 |(u_{n}(t,x) - u_{n}(t,x')) \vert \leq \frac{\varepsilon'}{3},
\end{equation}
for $x, x' \in C$ such that $ \vert x - x'\vert < \eta$ and $t \in [0,T]$.

Regarding the second one we have
\begin{equation} \label{EB1B2} 
  |u_{n}(t,x') - u_{n}(t',x')|   \leq B_1 + B_2,
\end{equation}
where 
\begin{equation} 
\label{defB}
\begin{array}{lll}
B_1 & := & \left|\int_{\mathcal{C}^d} \big[ K(x'-X_t(\omega))-K(x'-X_{t'}(\omega)) \big] \, \exp \left \{\int_{0}^{t}\Lambda_n \big (r,X_r(\omega),u_{n}(r,X_r(\omega))\big )dr\right \} dm^{n}((\omega)) \right| \\
B_2 & := & \left| \int_{\mathcal{C}^d} K(x'-X_{t'}(\omega)) \, \big[ \exp \left \{\int_{0}^{t}\Lambda_n \big (r,X_r(\omega),u_{n}(r,X_r(\omega))\big )dr\right \} - \right . \\
& & \left . \exp \left \{\int_{0}^{t'}\Lambda_n \big (r,X_r(\omega),u_{n}(r,X_r(\omega))\big )dr\right \} \big] dm^{n}(\omega) \right|
\end{array}
\end{equation}

%

We first estimate   $B_{1}$. We fix $\varepsilon > 0$. Let us introduce the continuous functional
 $\shc^d \longrightarrow \shc([0,T] \times C, \R)$ given by
$$
F : \eta \mapsto \Big( (t,x') \in [0,T] \times \R^d \mapsto K(x' - \eta_t)  \Big) \ ,
$$
where we denote by $\shc([0,T] \times C, \R)$ the linear space of real valued continuous functions on $[0,T] \times C$, equipped with the usual sup-norm topology. Since $(Y^n)_{n \in \N}$ converges in law, the sequence of r.v. \\
$(R^n_{t,x'} := F(Y^n)(t,x'), (t,x') \in [0,T] \times C)_{n \in \N}$ indexed on $[0,T] \times C$, also converges in law. In particular, the sequence of their corresponding laws are tight. \\
An easy adaptation of Theorem 7.3 page 82 in \cite{billingsley} (and the first part of its proof) to the case of random fields shows the existence of $\delta_{\varepsilon} > 0$ such that 
\begin{eqnarray}
\label{tight_mn}
\forall n \in \N, \; \P(\Omega^n_{\varepsilon,\delta_{\varepsilon}}) \leq \varepsilon \ ,
\end{eqnarray}
where $\displaystyle{ \Omega^n_{\varepsilon,\delta_{\varepsilon}} := \left \{ \bar{\omega} \in \Omega \; \Big \vert \; \sup_{ \underset{(x,x') \in C^2, \vert x-x' \vert \leq \delta_{\varepsilon}}{(t,t') \in [0,T]^2, \vert t-t' \vert \le \delta_{\varepsilon}} } \left \vert K(x-Y^n_t(\bar{\omega})) - K(x'-Y^n_{t'}(\bar{\omega})) \right \vert \geq \varepsilon  \right \} }$. \\
In the sequel of the proof, for simplicity we will simply write $\Omega^n_{\varepsilon} := \Omega^n_{\varepsilon,\delta_{\varepsilon}}$. Suppose that $\vert t-t' \vert \leq \delta_{\varepsilon}$. \\
Then, for all $x' \in C$
\begin{eqnarray}
\label{E416}
B_1 & = & \Big \vert \E \Big[ \Big( K(x'-Y^n_t) - K(x'-Y^n_{t'}) \Big) \exp \Big\{ \int_0^t \Lambda(r,Y^n_r,u^n(r,Y^n_r)) \Big \}  \Big] \Big \vert \nonumber \\
& \leq & \exp(M_{\Lambda}T) \E \Big[ \vert K(x'-Y^n_t) - K(x'-Y^n_{t'}) \vert \Big] \nonumber \\
& = & \exp(M_{\Lambda}T) \left( I_1(\varepsilon,n) + I_2(\varepsilon,n) \right) \ ,
\end{eqnarray}
where 
\begin{eqnarray}
\label{E417}
I_1(\varepsilon,n) & := & \E \Big[ \vert K(x'-Y^n_t) - K(x'-Y^n_{t'}) \vert \; 1_{\Omega^n_{\varepsilon}} \Big] \\
I_2(\varepsilon,n) & := & \E \Big[ \vert K(x'-Y^n_t) - K(x'-Y^n_{t'}) \vert \; 1_{(\Omega^n_{\varepsilon})^{c}} \Big] \ .
\end{eqnarray}
We have
\begin{eqnarray}
I_1(\varepsilon,n) \leq 2M_K \P(\Omega^n_{\varepsilon}) \leq 2 M_K \varepsilon \ ,
\end{eqnarray}
and
\begin{eqnarray}
I_2(\varepsilon,n) \leq \varepsilon \P((\Omega^n_{\varepsilon})^{c}) \leq \varepsilon \ .
\end{eqnarray}
At this point, we have shown that for $\vert t-t' \vert \leq \delta_{\varepsilon}$, $x' \in C$,
\begin{eqnarray}
B_1 \leq \varepsilon (2M_K + 1) \exp(M_{\Lambda}T) \ .
\end{eqnarray}
We can now choose $ \varepsilon := \frac{\varepsilon'}{3(2M_K +1)} \exp(-M_{\Lambda}T)$ so that $B_1 \leq \frac{\varepsilon'}{3}$. \\

Concerning the term $B_{2}$, using \eqref{EMajor1},
we have
\begin{equation}
\label{eq:termB1}
 \begin{array}{lll}
  B_{2} & \leq & \int_{\mathcal{C}^d} | K(x'-X_{t'}(\omega)) | \Big| \; e^{ \left \{\int_{0}^{t}\Lambda_n \big (r,X_r(\omega),u_{n}(r,X_r(\omega))\big )dr\right \} } - e^{ \left \{\int_{0}^{t'}\Lambda_n \big (r,X_r(\omega),u_{n}(r,X_r(\omega))\big )dr\right \} } \Big| dm^{n}(\omega) \\
  & \leq & M_K \exp(M_{\Lambda}) 
\int_{\shc^d} dm^n(\omega)   \left| \int_{t}^{t'} \Lambda_n \big (r,X_r(\omega),u_{n}(r,X_r(\omega))\big )dr \right| \\
  & \leq & M_K \exp(M_{\Lambda})M_{\Lambda}|t-t'| \ .
 \end{array}
\end{equation}
We choose $ \delta = \min({\delta}_{\epsilon}, \frac{\varepsilon'}{3 M_K M_{\Lambda} \exp(M_\Lambda) } )$.
For $\vert t - t'\vert < \delta$, we have $B_2 \le \frac{\varepsilon'}{3}$.
By additivity $B_1 +B_2 \le \frac{2\varepsilon'}{3}$
and finally, taking into account  \eqref{Exprime} and \eqref{EB1B2},
 \eqref{udecomp1} is verified.
This concludes the proof of Proposition \ref{cvuoncompact}.

\end{proof}



Regarding the limit in  Lemma \ref{limit}, we can be more precise by using once again Lemma \ref{karatzas}  and Remark \ref{Rkaratzas}.

\begin{prop}
\label{limit_in_u}
Let $\Lambda_n, \Lambda$ be as in Lemma \ref{cvu}.
Let $(Y^n)$ be a sequence of $\R^d-$valued continuous processes,
 whose law is denoted by 
$m^n$.
Let $u_n:[0,T] \times \R^d \rightarrow \R$ as in \eqref{E536}. 
Let $Z^n_t := u_n(t,Y^n_t), t \in [0,T]$. \\  
We suppose that  $(Y^n,Z^n)$ converges in law. \\ 
Then $(u_n)$ converges uniformly on each compact to some continuous $u: [0,T] \times \R^d \rightarrow \R$
which fulfills
\begin{equation} \label{EuFinal}
 u(t,\eta) = \int_{\mathcal{C}^d} K(\eta-X_{t}(\omega)) \exp \big(\int_{0}^{t} \Lambda(r,X_r(\omega),u(r,X_r(\omega)))dr\big) dm(\omega),
\end{equation}
where $m$ is the limit of $(m^n)_{n \geq 0}$. 
\end{prop}

\begin{proof}

Without loss of generality, the proof below is written with $d=1$. By Lemma \ref{limit}, the left-hand side of \eqref{E536} converges pointwise  to $u$, 
where $u$ is defined in \eqref{eq:ulimit}. 
By Proposition \ref{cvuoncompact} the convergence holds uniformly on each compact and
 $u$ is continuous.
 It remains to show that  $u$ fulfills \eqref{EuFinal}.
For this we will take 
 the limit of the right-hand side (r.h.s) of \eqref{E536}  and we will show that it gives 
the r.h.s of \eqref{EuFinal}.
For $n \in \N$, $(r,x) \in [0,T] \times \R$, we set 
\begin{eqnarray}
 \tilde{\Lambda}_n(r,x) & := & \Lambda_n(r,x,u_n(r,x)) \\
 \tilde{\Lambda}(r,x) & := & \Lambda(r,x,u(r,x)).
\end{eqnarray}
We fix $(t,\eta) \in [0,T] \times \R$.
In view of applying Lemma \ref{karatzas}, we define $f_n, f : \mathcal{C} \rightarrow \R$ such that
\begin{eqnarray*}
 f_{n}(y) &=&  K(\eta-y_{t}) \exp \big(\int_{0}^{t} \tilde{\Lambda}_{n}(r,y_r)dr\big) \\
 f(y) &=&  K(\eta-y_{t}) \exp \big(\int_{0}^{t} \tilde{\Lambda}(r,y_r)dr\big) \ .
\end{eqnarray*}

We also set $\mathbb{P}^n := m^{n}$. Since $(Y^n,Z^n)$ converges in law to $\nu$, $m^n$ converges weakly to $m$. 
Moreover, since $\vert \tilde{\Lambda}_n \vert$ are uniformly bounded with
 upper bound $M_{\Lambda}$,
 $(f_n)$ are also uniformly bounded.

The maps $(f_n)$ are continuous by Remark \ref{continuite}, and also the function $f$ since, 
 $u$ is continuous on $[0,T] \times \R$. 
Taking into account Remark \ref{Rkaratzas}, we will
 show that $f_n \rightarrow f$ 
uniformly on each ball of $ \mathcal{C}$. \\

Let us fix $ M > 0$ and denote $B_1(O,M) := \{ y \in \mathcal{C}, ||y||_{\infty} := \sup_{s \in [0,T]}|y_s| \leq M \}$. 
For any locally bounded function $\phi : [0,T] \times \R  \rightarrow \R$,
we set   $ ||\phi||_{\infty,M} := \sup_{s \in [0,T], \xi \in [-M,M]}|\phi(s,\xi)|$. 
Let $\varepsilon > 0$. \\
Since $u_n \rightarrow u$ uniformly on $[0,T] 
\times [-M,M]$, there exists $n_0 \in \N$ such that,
 \begin{equation}
  \label{ECU}
 n \geq n_{0} \Longrightarrow ||u_{n} - u||_{\infty,M} < \varepsilon \ .
 \end{equation}
The sequence $u_n|_{[0,T] \times [-M,M]}$ is uniformly bounded. 
Let $I_{M}$ be a compact interval including the subset \\ $\{u_n(s,x) \; | \; 
(s,x) \in [0,T] \times [-M,M] \}$. \\ 
For all $(s,x) \in [0,T] \times [-M,M]$, 
\begin{equation}
\label{EA-1}
 \begin{array}{lll}
  |\tilde{\Lambda}_n(s,x) - \tilde{\Lambda}(s,x)| & = & |\Lambda_n(s,x,u_n(s,x)) - \Lambda(s,x,u(s,x))| \\
  & \leq & |\Lambda_n(s,x,u_n(s,x)) - \Lambda(s,x,u_n(s,x))| + |\Lambda(s,x,u_n(s,x)) - \Lambda(s,x,u(s,x))| \\
  & := & I_1(n,s,x) + I_2(n,s,x) \ .
 \end{array}
\end{equation}
Concerning $I_1$, since for almost all $s \in [0,T]$, $\Lambda_n(s,\cdot,\cdot) \xrightarrow[\text{$n \rightarrow \infty$}]{\text{}} \Lambda(s,\cdot,\cdot)$ uniformly on $[-M,M] \times I_M$, we have for $x \in [-M,M]$,
$$
0 \leq I_1(n,s,x) \leq \sup_{x \in [-M,M] , \xi \in I_M} |\Lambda_n(s,x,\xi) - \Lambda(s,x,\xi)| \xrightarrow[\text{$n \rightarrow \infty$}]{\text{}} 0 \; ds \textrm{-a.e. ,}
$$
from which we deduce
\begin{equation}
 \label{I1cvg}
 \sup_{x \in [-M,M]} I_1(n,s,x) \xrightarrow[\text{$n \rightarrow \infty$}]{\text{}} 0 \; ds \textrm{-a.e.}
\end{equation}
Now, we treat the term $I_2$. Taking into account \eqref{ECU}, we get for $n \geq n_0$ ($n_0$ depending on $\varepsilon$), 
\begin{equation}
 \label{I2major}
 0 \leq \sup_{s \in [0,T],x \in [-M,M]} I_2(n,s,x) \leq \sup_{s \in [0,T],x \in [-M,M],
 |\xi_1 - \xi_2| \leq \varepsilon} |\Lambda(s,x,\xi_{1}) - \Lambda(s,x,\xi_{2})|
\end{equation}
We take the $\limsup$ on both sides of \eqref{I2major}, which gives,
\begin{equation}
 \label{I2cvg}
 \limsup_{n \longrightarrow \infty} \sup_{s \in [0,T],x \in [-M,M]} I_2(n,s,x) \leq S(\varepsilon),
\end{equation}
where $S(\varepsilon) := \sup_{s \in [0,T],x \in [-M,M] , |\xi_1 - \xi_2| \leq \varepsilon} 
|\Lambda(s,x,\xi_{1}) - \Lambda(s,x,\xi_{2})|$.
Summing up \eqref{I1cvg}, \eqref{I2cvg} and taking into account \eqref{EA-1}, we get,
\begin{equation}
 \label{cvgEA-1}
 0 \leq \limsup_{n \rightarrow \infty} \sup_{x \in [-M,M]}|\tilde{\Lambda}_n(s,x) - \tilde{\Lambda}(s,x)| \leq S(\varepsilon) \; ds\textrm{-a.e.}
\end{equation}
Since $\Lambda$ satisfies Assumption \ref{ass:main2}, 
the uniform continuity of $(x,\xi) \in \R \times \R \mapsto \Lambda(s,x,\xi)$
 (uniformly with respect to $s$) holds and $\lim_{\varepsilon \longrightarrow 0} S(\varepsilon) = 0$. \\

Finally, 
\begin{equation}
\label{Eqcvg}
\sup_{x \in [-M,M]}|\tilde{\Lambda}_n(s,x) - \tilde{\Lambda}(s,x)| \xrightarrow[\text{$n \rightarrow \infty$}]{\text{}} 0 \; ds \textrm{-a.e.}
\end{equation}
 Now, for   $n > n_{0}$ we obtain 
 \begin{equation} \label{E548}
  \begin{array}{lll}
  \sup_{y \in B_1(0,M)} |f_{n}(y) - f(y)| & \leq & M_K 
\exp(M_{\Lambda}T) \int_{0}^{T} \sup_{x \in [-M,M]} 
|\tilde  \Lambda_n(r,x) - \tilde \Lambda(r,x) | dr. 
  \end{array}
 \end{equation}
Since $(\tilde{\Lambda}_n),\Lambda$ are uniformly bounded, \eqref{Eqcvg} and Lebesgue's dominated convergence theorem, the right-hand side of \eqref{E548} goes to $0$ when $n \longrightarrow 0$.
This shows that $f_n \longrightarrow f$ uniformly on $B_1(O,M)$. \\


We can now apply Lemma \ref{karatzas} (with $\mathbb{P}_{n}$ and $f_{n}$ defined above) to obtain, for $n \longrightarrow \infty$,
$$
\int_{\mathcal{C}} K(\eta-X_{t}(\omega)) \exp \left(\int_{0}^{t} \Lambda_n(r,X_r(\omega),u_{n}(r,X_r(\omega)))dr\right) dm^{n}(\omega) 
$$
converges to
$$
\int_{\mathcal{C}} K(\eta-X_{t}(\omega)) \exp \left(\int_{0}^{t} \Lambda(r,X_r(\omega),u(r,X_r(\omega)))dr\right) dm(\omega),
$$
which finally proves \eqref{EuFinal} and concludes the proof of Proposition \ref{limit_in_u}.
\end{proof}
At this point we state simple technical lemma concerning
strong convergence of solutions of stochastic differential equations.

\begin{lem}
Let $R_0$ be a square integrable random variable on some filtered probability space,
equipped with a $p$ dimensional Brownian motion $W$.
 \label{approxSDE}
 Let $a_{n} : [0,T] \times \R^d \longrightarrow \R^{d \times p}$ and $b_{n} : [0,T] \times \R^d \longrightarrow \R^{d}$ Borel functions verifying the following.
 \begin{itemize}
  \item  $\exists L > 0$, for all $(x,y) \in \R^d \times \R^d$,  $\sup_{n \geq 0} |a_{n}(t,x) - a_{n}(t,y)| + \sup_{n \geq 0} |b_{n}(t,x) - b_{n}(t,y)| \le  L |x-y| $;
  \item  $\exists c > 0$, for all $x \in \R^d $,  $\sup_{n \geq 0} (|a_{n}(t,x) | +  |b_{n}(t,x)|) \le c(1 +|x|) $;
  \item $(a_{n}), (b_{n})$ converge pointwise respectively  to  Borel functions $a : [0,T] \times \R^d \rightarrow \R^{d \times p}$ and $b : [0,T] \times \R^d \rightarrow \R^{d}$.
\end{itemize}
Then there exists a unique strong solution of
 \begin{equation} \label{ESDEClass}
   \left \{
   \begin{array}{lll}
    dY_{t} = a(t,Y_t)dW_t + b(t,Y_t)dt \\
    Y_0 = R_0,\ .
   \end{array}
   \right .
  \end{equation}
Moreover,  let for each $n$, let
  the strong solution $X^{n}$ (which of course exists) of
 
   \begin{equation}
   \left \{
   \begin{array}{lll}
    dY^{n}_{t} = a_{n}(t,Y^{n}_t)dW_t + b_{n}(t,Y^{n}_t)dt \\
    Y^{n}_0 = R_0.
   \end{array}
   \right .
  \end{equation}
 Then, 
$$
\sup_{t \leq T} |Y^{n}_{t} - Y_{t}| \xrightarrow[\text{$n \longrightarrow + \infty$}]{\text{$L^{2}$}} 0.
$$
 
\end{lem}

\begin{proof} The existence and uniqueness of $Y$ follows because $a, b$ are Lipschitz with linear growth. \\
 The proof of the convergence is classical: it relies on BDG and Jensen inequalities together with    Gronwall's lemma.
\end{proof}

Now, we are able to prove the main result of this section.

\begin{proof}[Proof of Theorem \ref{thm_semi_weak}]

Let $Y_0$ be a r.v. distributed according to $\zeta_0$.
We set 
\begin{equation}
 \label{def-lambda-n}
 \Lambda_{n} : (t,x,\xi) \in [0,T] \times \R^d \times \R 
\mapsto \Lambda_{n}(t,x,\xi) := \int_{\R^d \times \R}
 \phi_{n}^d(x - x_{1})\phi_{n}(\xi - \xi_{1})\Lambda(t,x_{1},\xi_{1})dx_1d\xi_1,
\end{equation}
where $(\phi_{n})_{n \geq 0}$ is a usual mollifier sequence
 converging (weakly) to the Dirac measure. 
Thanks to the classical properties of the convolution, we know that $\Lambda$ 
being bounded implies          \\
$ \forall n \in \N, ||\Lambda_{n}||_{\infty} \le
 ||\phi_{n}^d||_{L^{1}}  ||\phi_{n}||_{L^{1}}  ||\Lambda||_{\infty} = 
||\Lambda||_{\infty} $.
For fixed $n \in \N$, $\phi_n$ is Lipschitz so that   \eqref{def-lambda-n} says that $\Lambda_n$ is also Lipschitz.
Then, for fixed $n \in \N$,  according to Assumption \ref{ass:main2}, $\Phi$, $g$, $\Lambda_{n}$ are Lipschitz and uniformly bounded, we 
can apply the results of Section \ref{S3} (see Theorem \ref{prop:NSDE}) to obtain the existence of a  pair $(Y^{n},u_{n})$ such that
\begin{equation}
\label{eq:NSDEregul}
 \left \{
 \begin{array}{lll}
  dY^{n}_{t} = \Phi(t,Y^{n}_t,u_{n}(t,Y_{t}^ {n}))dW_t + g(t,Y^{n}_t,u_{n}(t,Y_{t}^ {n}))dt \\
  Y^{n}_0 = Y_0, \; \\
  u_{n}(t,x) = \E{[K(x-Y^{n}_{t}) \exp\big( \int_{0}^{t} \Lambda_{n}(r,Y^{n}_r,u_{n}(r,Y^{n}_r))dr \big)]}. 
\end{array}
\right .
\end{equation}

Since $\Lambda_{n}$ is uniformly bounded and $ \{ Y_0^n \}$ are obviously tight,
Lemma \ref{tightness} in the Appendix 
gives the existence of 
a subsequence $(n_{k})$ 
such that $(Y^{n_{k}},u_{n_{k}}(\cdot,Y^{n_{k}}_{\cdot})) $ converges
in law to some probability measure $\nu$ on $\mathcal{C}^d \times \mathcal{C}$.
By Assumption \ref{ass:main2}, for all $t \in [0,T]$, $\Lambda_n(t,\cdot,\cdot)$ converges to $\Lambda(t,\cdot,\cdot)$, uniformly on every compact subset of $\R^d \times \R$. 
 
In view of applying Proposition \ref{limit_in_u}, we set $Z^{n_{k}}_t := u_{n_k}(t,Y^{n_k}_t)$ and $m^{n_k} := \mathcal{L}(Y^{n_k})$.
We know that $(\Lambda_{n_k}),\Lambda$ satisfy the hypotheses of Lemma \ref{cvu}. 
On the other hand $(Y^{n_k},Z^{n_{k}})$ converges in law to $\nu$.
So we can apply  Proposition \ref{limit_in_u}, which says  that $ (u_{n_k}) $ converges uniformly on each compact to some $u$  
which verifies \eqref{EuFinal}, where  $m$ is the first marginal of $\nu$.
In particular we emphasize that the sequence $(Y^{n_k})$ converges in law to $m$.

We continue the proof of Theorem \ref{thm_semi_weak} concentrating on the first line of \eqref{eq:NSDE}.

We set, for all $(t,x) \in [0,T] \times \R^d$, $k \in \N$,

\begin{equation}
 \label{coeffs}
 \begin{array}{lll}
  a_{k}(t,x) & := & \Phi(t,x,u_{n_{k}}(t,x)) \\
  b_{k}(t,x) & := & g(t,x,u_{n_{k}}(t,x)) \\
  a(t,x) & := & \Phi(t,x,{u}(t,x)) \\
  b(t,x) & := & g(t,x,{u}(t,x)) \ .
 \end{array}
\end{equation}
Here, the functions $u_n$  being fixed, the first equation of \eqref{eq:NSDEregul} is
 a classical SDE, whose coefficients depend on the (deterministic) continuous function $u_n$.
By Remark \ref{RLipeq:u},  the functions $u_{n}$ appearing in \eqref{eq:NSDEregul} are Lipschitz with respect to
the second argument and bounded. This implies that the coefficients  $a_k, b_k$ 
are  Lipschitz (with constant not depending on $k$) and
uniformly bounded.


Since $(u_{n_k})$ converges pointwise to $u$, then $(a_k)$, $(b_k)$ converges 
pointwise respectively to $a, b$ where 
$ a(t,x) = \Phi(t,x,u(t,x)), 
   b(t,x) = g(t,x,u(t,x))$.

Consequently, we can apply Lemma \ref{approxSDE} with the  sequence of classical SDEs 
\begin{equation}
 \label{SequenceSDEs}
 \left \{
 \begin{array}{lll}
 dY^{n_k}_{t} = a_{k}(t,Y^{n_k}_t)dW_t + b_{k}(t,Y^{n_k}_t )dt \\
 Y^{n_k}_0 = Y_0,
 \end{array}
 \right .
\end{equation}
to obtain 
$$\sup_{t \leq T} |Y^{n_k}_{t} - Y_{t}| \xrightarrow[\text{$k \longrightarrow + \infty$}]{\text{$L^{2}(\Omega)$}} 0, $$
where $Y$ is the (strong) solution to the classical SDE 
   \begin{equation}
   \label{LimitX}
   \left \{
   \begin{array}{lll}
    dZ_{t} = a(t,Z_t)dW_t + b(t,Z_t)dt \\
    Z_0 = Y_0 \\
   a(t,x) = \Phi(t,x,{u}(t,x)) \\
   b(t,x) = g(t,x,{u}(t,x)) \ .
   \end{array}
   \right .
  \end{equation}
We remark that $Y$ verifies the first equation of \eqref{eq:NSDE} and the corresponding $u$
fulfills \eqref{EuFinal}.   To conclude the proof of Theorem \ref{thm_semi_weak} 
it remains to identify the law of $Y$ with $m$.
Since $Y^{n_k}$ converges strongly, 
then the laws $m^{n_k}$ of  $Y^{n_k}$ converge to the law of $Y$,
which by Proposition \ref{limit_in_u},
coincides  necessarily to $m$.  

\end{proof}

\section{Weak Existence when the coefficients are continuous}

\label{S5}

In this section
we consider again \eqref{eq:NSDE} i.e.
problem
 \begin{equation}
\label{eq:NSDEweak}
\left\{
\begin{array}{l}
Y_t=Y_0+\int_0^t \Phi(r,Y_r,u(r,Y_r)) dW_r+\int_0^tg(r,Y_r,u(r,Y_r))dr\ ,\quad\textrm{with}\quad Y_0\sim  \zeta_0\ ,\\
u(t,x)= \int_{\shc^d} dm(\omega) \left[ K(x-X_t(\omega)) \, 
\exp \left \{\int_{0}^{t}\Lambda \big (r,X_r(\omega),
u(r,X_r(\omega))\big )dr\right \}\right ] \ , \quad\textrm{for}\quad (t,x) \in [0,T] \times \R^d \\
\shl(Y) = m \ ,
\end{array}
\right .
\end{equation}
but  under weaker conditions on the coefficients
$\Phi, g, \Lambda$ and initial condition $\zeta_0$.
 In that case the existence or the
well-posedness will only be possible in the weak sense,
i.e., not on a fixed (a priori)
 probability space. \\
The aim of this section is to show weak existence for problem
\eqref{eq:NSDEweak}, in the sense of Definition \ref{def-weak-sol}
under Assumption \ref{ass:main3}.
The idea consists here  in regularizing  the functions $\Phi$ and $g$
and truncating the initial condition $\zeta_0$ to use 
 existence result stated in Section \ref{S4}, i.e. Theorem \ref{thm_semi_weak}.

%

\begin{thm}
\label{thm_weak}
Under Assumption \ref{ass:main3}, the problem \eqref{eq:NSDE}
admits existence in law, i.e. there is a solution
$(Y,u)$ of  \eqref{eq:NSDEweak} on a suitable probability space
equipped with a Brownian motion.
\end{thm}

\begin{proof}
We consider the following mollifications (resp. truncations) of
the coefficients (resp. the initial condition).
\begin{equation}
 \label{eq:coeffsregul}
 \begin{array}{lll}
 \Phi_{n} : (t,x,\xi) \in [0,T] \times \R^d \times \R \mapsto \int_{\R^d \times \R} \phi^d_{n}(x-r')\phi_{n}(\xi-r) \Phi(t,r',r)dr'dr \\
 g_{n} : (t,x,\xi) \in [0,T] \times \R^d \times \R \mapsto \int_{\R^d \times \R}
 \phi^d_{n}(x-r')\phi_{n}(\xi-r) g(t,r',r)dr'dr  \\ 
 \forall \varphi \in \shc_b(\R^d), \; \int_{\R^d} \zeta_0^n (dx) \varphi(x) =  \int_{\R^d} \zeta_0 (dx) 
(-n \vee \varphi(x)) \wedge n \ .
 \end{array}
\end{equation} 
We fix a filtered probability space $(\Omega, \shf, \P)$
equipped with an $(\mathcal F_t)_{t \geq 0}$-Brownian motion $W$.
 First of all, we point out the fact that the function $\Lambda$ satisfies the same assumptions as in Section \ref{S4}. 
 On the one hand, by \eqref{eq:coeffsregul}, since $\phi_{n}$ belongs to $\shs(\R^d)$, $\Phi_n$ and $g_n$ are uniformly bounded and Lipschitz with respect to $(x,\xi)$ uniformly w.r.t. $t$ for each $n \in \N$.
Also $\zeta_0^n$ admits a second moment and $(\xi^n_0)$ weakly converges to $\xi_0$.
On the other hand 
For each $n$, let  $Y^n_0$ be a (square integrable) r.v. distributed
according to $\zeta^n_0$.
By Theorem \ref{thm_semi_weak}, there is a  pair $(Y^n,u^n)$
fulfilling \eqref{eq:NSDE} with $\Phi, g, \zeta_0$ replaced by 
$\Phi_n, g_n, \zeta_0^n$.
In particular we have
 \begin{equation}
\label{eq:NSDEweakregul}
\left\{
\begin{array}{l} 
Y^n_t=Y_0^n+\int_0^t \Phi_n(r,Y^n_r,u_n(r,Y^n_r)) dW_r+\int_0^tg_n(r,Y^n_r,u_n(r,Y^n_r))dr\ ,\quad\textrm{with}\quad Y_0^n\sim  \zeta_0^n\ ,\\
u_n(t,x)= \int_{\shc^d} dm^n(\omega) \left[ K(x-X_t(\omega)) \, 
\exp \left \{\int_{0}^{t}\Lambda \big (r,X_r(\omega),
u_n(r,X_r(\omega))\big )dr\right \}\right ] \ , \quad\textrm{for}\quad (t,x) \in [0,T] \times \R^d  \ ,\\
\shl(Y^n) = m^n.

\end{array}
\right .
\end{equation}

\begin{rem}
\label{un_bounded}
Similarly to Remark \ref{rem_un}
 $(\vert u_n \vert)_{n \geq 0}$ is uniformly bounded by 
 $M_{K} \exp(M_\Lambda T)$. 
\end{rem}



Setting $Z^n:= u_{n}(\cdot,Y^n)))$, in  the sequel, we will denote by $\nu^n$ the Borel probability 
defined by
$ \mathcal{L}(Y^n,Z^n)$.
The same notation will be kept after possible extraction of  subsequences.




Since $(\zeta_0^n)_{n \in \N}$ weakly converges to $\zeta_0$,  it is tight. By Remark \ref{un_bounded} and Lemma \ref{tightness},
 there is a subsequence  $\{ \nu^{n_k}\}$ which weakly converges to
some Borel probability $\nu$ on $\mathcal C^d \times \mathcal C$.
For simplicity we replace in the sequel the subsequence $(n_k)$ by $(n)$. 
Let $(Y^n)$ be the sequence of processes solving \eqref{eq:NSDEweakregul}.
We remind that $m^n$ denote their law. 
The final result will be established once we will have proved the following statements, 
\begin{enumerate}[a)]
\item $u^n$ converges to some (continuous) function $u : [0,T] \times \R^d \rightarrow \R$, uniformly on each compact of $[0,T] \times \R^d$, which verifies 
$$
\forall (t,x) \in [0,T] \times \R^d, \; u(t,x) = \int_{\mathcal{C}^d}  K(x-X_t(\omega))
 \, \exp \left \{\int_{0}^{t}\Lambda \big (r,X_r(\omega),u(r,X_r(\omega)))\big )dr\right \}
 dm(\omega)
$$
where $m$ is the limit of the laws of $Y^n$.
\item The processes $Y^n$ converge in law to $Y$, where $Y$ is a solution, in law, of 
\begin{equation}
\label{eq:limitY}
\left \{
\begin{array}{l}
Y_t = Y_0 + \int_0^t \Phi(r,Y_r,u(r,Y_r)) dW_r + \int_0^t g(r,Y_r,u(r,Y_r)) dr \\
Y_0 \sim \zeta_0 \ .
\end{array}
\right .
\end{equation}
\end{enumerate}
Step a) is a consequence of Proposition \ref{limit_in_u} with for all $n \in \N, \; \Lambda_n \equiv \Lambda$. \\
 To prove the second step b), we will  pass to the limit in the first equation of \eqref{eq:NSDEweakregul}.  To this end, let us denote by $C^{2}_0(\R^d)$, the space of $C^{2}(\R^d)$ functions  with compact support.
Without loss of generality, we suppose $d = 1$. We will prove that $m$ is a solution to the martingale problem  (in the sense of Stroock and Varadhan, see chapter 6 in \cite{stroock}) associated with the first equation of \eqref{eq:NSDEweak}. 
In fact we will show that
\begin{equation}
  \label{eq:MartPb}
  \left \{
  \begin{array}{lll}
   \forall \varphi \in C^{2}_0(\R), t \in [0,T], \; M_{t} := \varphi(X_t) - \varphi(X_0) - 
\int_0^t (\mathcal{A}_{r} \varphi)(X_r) dr, \text{ is a } \mathcal{F}^{X}_{t}\text{-martingale, where } \\
    (\mathcal{F}^X_t, t \in [0,T]) \text{ is the canonical filtration generated by } X, 
  \end{array}
  \right .
 \end{equation}
where we denote $\mathcal{A}_{r} \varphi)(x) =
 \frac{1}{2} \Phi^{2}(r,x,u(r,x))) \varphi''(x) + g(r,x,u(r,x)) \varphi'(x), 
\ r \in [0,T], x \in \R^d.$  \\
 Let $ 0 \leq s < t \leq T$ fixed,
$ F :\shc([0,s],\R) \rightarrow \R$ continuous and bounded.
  Indeed, we will show  
  \begin{equation}
  \label{eq:MartPb_bis}
  \begin{array}{lll}
   \forall \varphi \in C^{2}_0(\R), \; \E^{m}{\left[ \left( \varphi(X_t) - \varphi(X_0) - \int_0^t (\mathcal{A}_{r} \varphi)(X_r) dr \right) F(X_r, r \le s) \right]}= 0 
  \end{array}
 \end{equation}
 We remind that, for $n \in \N$, by definition, $m^n$ is the law of the  strong solution
$Y^n$  of 
$$
Y^n_t=Y_0^n+\int_0^t \Phi_n(r,Y^n_r,u_n(r,Y^n_r)) dW_r + \int_0^tg_n(r,Y^n_r,u_n(r,Y^n_r))dr \ ,
$$
on  a fixed underlying probability space $(\Omega,\mathcal{F},\mathbb{P})$ with related expectation  $\E$.  \\
  Then, by It\^o's formula, we easily deduce  that $\forall n \in \N$,  
 \begin{equation}
 \E{\left[ \left(\varphi(Y^n_t) - \varphi(Y^n_s) - \int_s^t \left(\frac{1}{2} \Phi_{n}^{2}(r,Y^n_r,u_n(r,Y^n_r)) \varphi''(Y^n_r) + g_{n}(r,Y^n_r,u_n(r,Y^n_r)) \varphi'(X^n_r) \right) dr\right) F(Y^n_r,  r \le s)  \right]}=0 \ .
 \end{equation} 
 Transferring this to the canonical space $\mathcal{C}$ and to the probability $m ^n$
gives 
 \begin{equation}
 \label{eq:MartPb2}
 \E^{m^n}{\left[ \left(\varphi(X_t) - \varphi(X_s) - \int_s^t \left(\frac{1}{2} \Phi_{n}^{2}(r,X_r,u_n(r,X_r)) \varphi''(X_r) + g_{n}(r,X_r,u_n(r,X_r)) \varphi'(X_r) \right) dr\right) F((X_u, 0 \leq u \leq s)) \right]}=0.
 \end{equation}
From now on, we are going to pass to the limit when $n \longrightarrow +\infty$ in \eqref{eq:MartPb2} to obtain \eqref{eq:MartPb}. Thanks to the weak convergence
of the sequence $m^n$, for $\varphi \in C^{2}_{0}(\R^d)$,  we have immediately
\begin{equation}
 \label{eq:MartPb3}
 \E^{m_n}{\left[ \left(\varphi(X_t) - \varphi(X_s)\right)F(X_u, 0 \leq u \leq s)\right]} - \E^{m}{\left[ \left(\varphi(X_t) -
 \varphi(X_s)\right)F(X_u, 0 \leq u \leq s)\right]} \xrightarrow[\text{$ n \longrightarrow \infty$}]{\text{}} 0.
\end{equation}
It remains to show, 

 \begin{equation}
 \label{eq:MartPb4}
 \left \{
 \begin{array}{lll}
\lim_{n \longrightarrow \infty} \E^{m_n}{\left[ H^{n}(X) F(X_u, 0 \leq u \leq s)\right]} = \E^{m}{\left[ H(X) F(X_u, 0 \leq u \leq s)\right]}, \\
 \text{with  } H^{n}(\alpha) := \int_s^t (\frac{1}{2} \Phi_{n}^{2}(r,\alpha_r,u_n(r,\alpha_r)) \varphi''(\alpha_r) + g_{n}(r,\alpha_r,u_n(r,\alpha_r)) \varphi'(\alpha_r) dr, \\
 H(\alpha) := \int_s^t (\frac{1}{2} \Phi^{2}(r,\alpha_r,u(r,\alpha_r)) \varphi''(\alpha_r) + g(r,\alpha_r,u(r,\alpha_r)) \varphi'(\alpha_r) dr \ . \\
 \end{array}
 \right .
 \end{equation}
In order to show that
 $\E^{m_n}{\left[ H^{n}(X) F\right]} - \E^{m}{\left[ H(X) F\right]}$   
goes to zero, we will apply again Lemma \ref{karatzas}. \\
As we have mentioned above, $F$ is continuous and bounded.
Similarly as for Remark \ref{continuite}, the proof of the continuity of $H$ (resp. $H_n$) makes use 
of the continuity of $\Phi$, $g$, $\varphi''$, $\varphi'$ (resp. $\Phi_n$, $g_n$,$\varphi''$, $\varphi'$) 
and Lebesgue dominated convergence theorem. 

Taking into account Remark \ref{Rkaratzas}, it is enough to prove
 the uniform convergence of 
 $H^n : \mathcal{C} \longrightarrow \R$ to 
$H : \mathcal{C} \longrightarrow \R$ on each ball of $\mathcal C$. 
This relies on the uniform convergence   of $\Phi_n(t,\cdot,\cdot)$ (resp. $g_n(t,\cdot,\cdot)$ ) to $\Phi(t,\cdot,\cdot)$ (resp. $g(t,\cdot,\cdot)$ ) on every compact subset $\R \times \R$, for fixed $t \in [0,T]$.
Since the sequence $(m^n)$ converges weakly, finally  Lemma \ref{karatzas} 
allows to conclude \eqref{eq:MartPb4}.

 \end{proof}

\section{Link with nonlinear Partial Differential Equation}
\label{SPDE}

\setcounter{equation}{0}

From now on, in all the sequel, to simplify notations, we will often use the notation $f_t(\cdot) = f(t,\cdot)$ for functions $f : [0,T] \times E \longrightarrow \R$, $E$ being some metric space.
\medskip

In the following, we suppose again the validity of Assumption \ref{ass:main3}.
\\
 Here  $\mathcal{F}(\cdot): f \in \mathcal{S}(\R^d) \mapsto \mathcal{F}(f) \in \mathcal{S}(\R^d)$ denotes the Fourier transform on the classical Schwartz space $\mathcal{S}(\R^d)$. We will indicate in the same manner the Fourier transform on $\mathcal{S}'(\R^d)$.
In this section, we want to  link  the nonlinear SDE \eqref{eq:NSDE} 
to a  partial integro-differential equation (PIDE)  that we have 
to determine. 
We start by considering  problem \eqref{eq:NSDE} written under the following form:
\begin{equation}
\label{Y_edp}
\left \{
\begin{array}{l}
Y_t=Y_0+\int_0^t\Phi(s,Y_s,u^{m}_s(Y_s))dW_s+\int_0^tg(s,Y_s,u^{m}_s(Y_s))ds, \quad Y_0 \sim \zeta_0\\
u^m_t(x) = \int_{\mathcal{C}^d}  K(x-X_t(\omega))
 \, \exp \left \{\int_{0}^{t}\Lambda \big (r,X_r(\omega),u^m_r(X_r(\omega)))\big )dr\right \}
 dm(\omega) \\
 \shl(Y) = m \ ,
\end{array}
\right .
\end{equation}
recalling that $V_t\big (Y,u^{m}(Y)\big 
)=\exp \big( \int_0^t\Lambda_s(Y_s,u^{m}_s(Y_s))ds \big)$. \\
Suppose that $K$ is formally the Dirac measure at zero. In this case,
the solution of \eqref{Y_edp} is also a solution of \eqref{ENLSDELambda}.
Let $\varphi \in \shs(\R^d)$. Applying It\^o formula to
$\varphi(Y_t)$ we can easily show that the function $v$, density of the measure $\nu$ defined in \eqref{ENLSDELambda}, is a solution in the distributional sense of the PDE \eqref{epde}.
For  $K$ being a mollifier of the Dirac measure, 
applying the same strategy, we cannot easily identify 
the deterministic problem solved by $u^m$, e.g. PDE or PIDE.

For that reason we begin by establishing a 
 correspondence between \eqref{Y_edp} and another 
McKean type stochastic differential equation, i.e. 
\begin{equation}
\label{Y_edp2}
\left \{
\begin{array}{l}
Y_t=Y_0+\int_0^t\Phi(s,Y_s,(K*\gamma^{m})(s,Y_s))dW_s+\int_0^tg(s,Y_s,(K*\gamma^{m})(s,Y_s))ds, \quad Y_0 \sim \zeta_0\\
\gamma^m_t \text{ is the measure defined by, for all } \varphi \in 
 \mathcal{C}_{b}(\R^d) \\
\gamma^m_t(\varphi) := \langle \gamma^m_t, \varphi \rangle :=  \int_{\shc^d} \varphi(X_t(\omega)) V_t(X,(K * \gamma^m)(X)) dm(\omega) \\
\shl(Y) = m \ ,
\end{array}
\right .
\end{equation}
where  we recall the notations $(K * \gamma)(s,\cdot) := (K * \gamma_s)(\cdot)$ and $ \int_{\R^d} \varphi(x) \gamma_t^m(dx) := \gamma^m_t(\varphi) $ .

\begin{thm}
\label{thm-PIDE}
We suppose the validity of Assumption \ref{ass:main3}.
The existence of the McKean type stochastic differential equation
 \eqref{Y_edp} is equivalent to the one of \eqref{Y_edp2}.
More precisely, given a solution  $(Y, \gamma^m)$ of \eqref{Y_edp2},  
$(Y,u^m),$ with $u^m = K * \gamma^m$,
  is a solution of \eqref{Y_edp}  and if $(Y,u^m)$ is a solution \eqref{Y_edp}, there exists a Borel measure $\gamma^m$ such that $(Y,\gamma^m)$ is solution of $\eqref{Y_edp2}$.  \\
In addition, if the measurable set $\{ \xi \in \R^d \vert \shf(K)(\xi) = 0  \}$ is Lebesgue negligible, \eqref{Y_edp} and \eqref{Y_edp2} are equivalent, i.e., the measure $\gamma^m$ is uniquely determined by $u^m$ and conversely.
\end{thm}

\begin{proof}
Let $((Y_t, t \geq 0),u^m)$ be a solution of \eqref{Y_edp}. Let us fix $t \in [0,T]$. \\
Since $K \in L^1(\R^d)$, the Fourier transform applied to the function $u^m(t,\cdot)$ gives
\begin{eqnarray}
\label{E61}
\shf(u^m)(t,\xi) = \shf(K)(\xi) \int_{\shc^d} e^{-i \xi \cdot X_t(\omega)} \exp \left( \int_0^t \Lambda(r,X_r(\omega),u^m_r(X_r(\omega))) \right) dm(\omega) \ .
\end{eqnarray} 
By Lebesgue dominated convergence theorem, the function
$$f^m : \xi \in \R^d \mapsto f(\xi) := \int_{\shc^d} e^{-i \xi \cdot X_t(\omega)} \exp \left( \int_0^t \Lambda(r,X_r(\omega),u^m_r(X_r(\omega))) \right) dm(\omega) \ ,
$$
 is clearly continuous on $\R^d$. Since $\Lambda$ is bounded, $f^m$ is also bounded. Let $(a_k)_{k=1,\cdots,d}$  be a sequence of complex numbers and $(x_k)_{k=1\cdots,d} \in (\R^d)^d$. Remarking that for all $\xi \in \R^d$
$$
\sum_{k=1}^d \sum_{p=1}^d a_k \bar{a}_p e^{-i \xi \cdot (x_k-x_p)} = \left( \sum_{k=1}^d a_k e^{-i \xi \cdot x_k} \right) \overline{\left( \sum_{p=1}^d a_p e^{-i \xi \cdot x_p} \right )} = \left \vert \sum_{k=1}^d a_k e^{-i \xi \cdot x_k} \right \vert^2 \ , 
$$
it is clear that $f^m$ is non-negative definite. Then, by Bochner's theorem (see Theorem 24.9 Chapter I.24 in \cite{rogers_v1}), there exists a finite non-negative Borel measure $\mu_t$ on $\R^d$ such that for all $\xi \in \R^d$
\begin{eqnarray} \label{E928}
f^m(\xi) = \frac{1}{\sqrt{2 \pi}}
 \int_{\R^d} e^{-i \xi \cdot \theta} \mu_t^m( d\theta) \ .
\end{eqnarray}
We want to show that $\gamma^m_t := \mu_t^m$ fulfills the third line equation of \eqref{Y_edp2}. \\
Since $\mu_t^m$ is a finite (non-negative) Borel measure, it is a Schwartz (tempered) distribution such that 
$$
\shf^{-1}(f^m) = \mu_t^m \quad \textrm{ and } \quad \forall \psi \in \shs(\R^d), \; \left \vert \int_{\R^d} \psi(x) \mu_t^m(dx) \right \vert \leq \Vert \psi \Vert_{\infty} \mu_t^m(\R^d) < \infty \ .
$$
On the one hand,  equalities \eqref{E61} and \eqref{E928} give
\begin{eqnarray} 
\label{E929}
\shf(u^m)(t,\cdot) = \shf(K) \shf(\mu_t^m) \Longrightarrow u^m(t,\cdot) = K * \mu_t^m \ .
\end{eqnarray}
On the other hand, for all $\psi \in \shs(\R^d)$, 
\begin{eqnarray}
\langle \mu_t^m,\psi \rangle & = & \langle \shf^{-1}(f^m),\psi \rangle \nonumber \\
& = & \langle f^m,\shf^{-1}(\psi) \rangle \nonumber \\
& = & \int_{\R^d} \shf^{-1}(\psi)(\xi)  \left( \int_{\shc^d} e^{-i \xi \cdot X_t(\omega)} \exp ( \int_0^t \Lambda(r,X_r(\omega),u^m_r(X_r(\omega))) ) dm(\omega) \right) d \xi \nonumber \\
& = & \int_{\shc^d} \left( \int_{\R^d} \shf^{-1}(\psi)(\xi) e^{-i \xi \cdot X_t(\omega)} d \xi \right)   \exp \left( \int_0^t \Lambda(r,X_r(\omega),u^m_r(X_r(\omega))) \right) dm(\omega) \nonumber  \\ 
& = & \int_{\shc^d} \left( \int_{\R^d} \shf^{-1}(\psi)(\xi) e^{-i \xi \cdot X_t(\omega)} d \xi \right)   \exp \left( \int_0^t \Lambda(r,X_r(\omega),  (K* \mu_r^m)(X_r(\omega))) \right) dm(\omega) \nonumber \\
& = & \int_{\shc^d} \psi(X_t(\omega)) \exp \left( \int_0^t \Lambda(r,X_r(\omega),  (K* \mu_r^m)(X_r(\omega))) \right) dm(\omega) \ ,
\end{eqnarray}
where the third equality is justified by Fubini theorem
and the fourth equality follows by \eqref{E929}.
 This allows to conclude the necessary part of the first statement of the lemma. \\  

Regarding the converse, let $(Y,\gamma^m)$ be a solution of
 \eqref{Y_edp2}. We set directly ${u}_t^m(x) := (K*\gamma^m_t)(x)$. 
Obviously the first equation  in \eqref{Y_edp} is satisfied for $(Y,u^m)$.
Since $\mu_t^m$ is finite, the second equation follows directly by \eqref{Y_edp2} 
to $\varphi = K(x-\cdot)$. \\
To establish the second statement of the theorem, it is enough to observe that from the r.h.s of \eqref{E929} we have
$$
\textrm{Leb}(\{ \xi \in \R^d \vert \shf(K)(\xi) = 0\}) = 0 \Longrightarrow \shf(\mu_t^m) = \frac{\shf(u^m)(t,\cdot)}{\shf(K)} \; \textrm{a.e. }\ , t \in [0,T],
$$
where Leb denotes the Lebesgue measure on $\R^d$. This shows effectively that $\gamma^m$ (resp. $u^m$) is uniquely determined by $u^m$ (resp. $\gamma^m$) and ends the proof.
\end{proof}
Now, by applying It\^o's formula 
, we can show that the associated measure $\gamma^m$ (second equation in \eqref{Y_edp2}) satisfies a PIDE.
\begin{thm} \label{T64}
The measure $\gamma^m_t$, defined in the second equation of \eqref{Y_edp2}, satisfies the  PIDE
\begin{equation}
\label{PIDE}
\left \{
\begin{array}{lll}
\partial_t \gamma^m_t & = & \frac{1}{2} \displaystyle{\sum_{i,j=1}^d} \partial^2_{ij}\left( (\Phi \Phi^t)_{i,j}(t,x,(K*\gamma^m_t)) \gamma^m_t \right) - div \left( g(t,x,K*\gamma^m_t) \gamma^m_t \right) + \gamma^m_t \Lambda(t,x,(K*\gamma^m_t)) \\
\gamma^m_0(dx) & = & \zeta_0(dx) \; , 
\end{array}\right .
\end{equation}
in the sense of distributions.
\end{thm}

\begin{proof}
It is enough to use the definition of $\gamma^m_t$ and, as mentioned above,
 apply It\^o's formula to the process \\
$\varphi(Y_t)V_t(Y,(K * \gamma^m)(Y))$, for $\varphi \in \mathcal{C}_0^{\infty}(\R^d)$ and $Y$ (defined in the first equation of \eqref{Y_edp2}). Indeed, for $\varphi \in \mathcal{C}_0^{\infty}(\R^d)$, It\^o's formula gives,
\begin{eqnarray}
\label{ito3}
\E[ \varphi(Y_t) V_t(Y,(K * \gamma^m)(Y))] & = & \E[\varphi(Y_0) ] \nonumber \\
& & + \; \int_0^t \E \left[\varphi(Y_s) \Lambda(s,Y_s,(K*\gamma^m)(s,Y_s) ) V_s(Y,(K * \gamma^m)(Y)) \right ] ds \nonumber \\
& &  + \; \int_0^t \sum_{i=1}^d \E \left[ \partial_i \varphi(Y_s) g_i(s,Y_s,(K*\gamma^{m})(s,Y_s)) V_s(Y,(K * \gamma^m)(Y)) \right] ds \nonumber \\
& & + \; \frac{1}{2}\int_0^t \sum_{i,j=1}^d \E[ \partial^2_{ij} \varphi(Y_s) (\Phi \Phi^t)_{i,j}(s,Y_s,(K*\gamma^{m})(s,Y_s)) V_s(Y,(K * \gamma^m)(Y))  ] ds  \nonumber \ . \\ 
\end{eqnarray}
By  the definition of the measure $\gamma^m$ in \eqref{Y_edp2}, we have 
\begin{eqnarray}
\int_{\R^d} \varphi(x) \gamma^m_t(dx) & = & \int_{\R^d} \varphi(x) \zeta_0(dx) \nonumber \\
& & + \; \int_0^t \int_{\R^d} \varphi(x) \Lambda(s,x,(K*\gamma^m)(s,x) ) \gamma^m_s(dx) ds \nonumber \\
& & + \; \int_0^t \int_{\R^d}  
 \nabla \varphi(x) \cdot  g(s,x,(K*\gamma^{m})(s,x)) \gamma^m_s(dx)  ds
 \nonumber \\
& & + \; \frac{1}{2} \sum_{i,j=1}^d \int_0^t \int_{\R^d} \partial_{ij}^2  \varphi(x) (\Phi \Phi^t)_{i,j}(s,x,(K*\gamma^{m})(s,x)) \gamma^m_s(dx) ds \ . 
\end{eqnarray}
This concludes the proof of Theorem \ref{T64}.
\end{proof}

\section{Particle systems approximation and propagation of chaos  }
\label{SChaos}

\setcounter{equation}{0}

In this section, we focus on the propagation of chaos for an interacting particle system
$\xi = (\xi^{i,N})_{i=1,\cdots, N}$ 
associated with the McKean type equation \eqref{eq:NSDE} 
when the coefficients $\Phi, g, \Lambda$ are bounded and Lipschitz. 
We remind that the propagation of chaos consists in the property of 
asymptotic independence of the components of $\xi$ when the size $N$ of the particle system 
goes to $\infty$. That property  was introduced in
\cite{McKean} and further developed and popularized by
 \cite{sznitman}. Moreover, we propose a particle approximation of $u$, solution of~\eqref{eq:NSDE}.

 We suppose here the validity of Assumption \ref{ass:main}. For the simplicity 
of formulation we suppose that $\Phi$ and $g$  only depend on the last variable $z$.
Let $(\Omega, \mathcal F, \P)$ be a fixed probability space,
and $(W^i)_{i=1,\cdots ,N}$ be a sequence of independent $\R^p$-valued Brownian motions. Let $(Y_0^i)_{i=1,\cdots,N}$ be i.i.d. r.v. according to $\zeta_0$. We consider $\mathbf{Y}:=(Y^i)_{i=1,\cdots ,N}$ the
 sequence of processes such that  $(Y^i, u^{m^i})$ are solutions to
\begin{equation}
\label{eq:Yi}
\left \{
\begin{array}{l}
Y^i_t=Y_0^i+\int_0^t\Phi(u^{m^i}_s(Y^i_s))dW^i_s+\int_0^tg(u^{m^i}_s(Y^i_s))ds \\
u^{m^i}_t(y)={\displaystyle \E\left [ K(y-Y^i_t)V_t\big (Y^i,u^{m^i}(Y^i)\big ) \right ]}\ ,\quad \textrm{with}\ m^i:=\mathcal{L}(Y^i)\ ,
\end{array}
\right .
\end{equation}
recalling that $V_t\big (Y^i,u^{m^i}(Y^i)\big )=\exp \big( \int_0^t\Lambda_s(Y^i_s,u^{m^i}_s(Y^i_s))ds \big)$. 
The existence and uniqueness of the solution of each equation is ensured by Proposition~\ref{prop:NSDE}. 
We remind that  the map                            $(m,t,y) \mapsto u^m(t,y)$
 fulfills the regularity properties given at the second and third item of Lemma \ref{lem:uu'} .

Obviously the processes  $(Y^i)_{i=1,\cdots ,N}$ are independent.
They are also identically distributed since Proposition \ref{prop:NSDE} also states uniqueness
in law.
\\
So we can define
 $m^0:=m^i$ the common distribution of the processes $(Y^i)_{i=1,\cdots ,N}$. \\
From now on, ${\mathcal C}^{dN}$  will denote $( {\mathcal C}^{d})^N$, which is obviously isomorphic to $\shc([0,T], \R^{dN})$.
We start observing that, for every $\bar{\xi} \in {\mathcal C}^{dN}$ 
the function $(t,x) \mapsto u^{S^N(\mathbf{\bar{\xi}})}_t(x)$
is obtained by composition of $m \mapsto u^m_t(x)$ 
 with $ m = S^N(\mathbf{\bar{\xi}})$.

Now let us introduce  the system of equations

\begin{equation}
\label{eq:XIi}
\left \{
\begin{array}{l}
\xi^{i,N}_t= \xi^{i,N}_0 +\int_0^t\Phi(u^{S^N(\mathbf{\xi})}_s(\xi^{i,N}_s))dW^i_s+\int_0^tg(u^{S^N(\mathbf{\xi})}_s(\xi^{i,N}_s))ds  \\
\xi^{i,N}_0 = Y^i_0 \\
u^{S^N(\mathbf{\xi})}_t(y)={\displaystyle \frac{1}{N}\sum_{j=1}^N  K(y-\xi^{j,N}_t) V_t\big (\xi^{j,N},u^{S^N(\mathbf{\xi})}(\xi^{j,N})\big ) }\ ,
\end{array}
\right .
\end{equation}
with $S^N(\mathbf{\xi})$ standing for the empirical measure associated to $\mathbf{\xi}:=(\xi^{i,N})_{i=1,\cdots ,N}$ i.e. 
\begin{equation}
\label{eq:SampleXi}
S^N(\mathbf{\xi}):=\frac{1}{N} \sum_{i=1}^N \delta_{\xi^{i,N}}\ .
\end{equation}
As for  \eqref{eq:SampleXi}, we set $\displaystyle{ S^N(\mathbf{Y}):=\frac{1}{N} \sum_{i=1}^N \delta_{Y^{i}} }$ is the empirical measure for 
${\bf Y} := (Y^i)_{i=1,\cdots,N},$
where we remind that for each $i \in \{1,\cdots,N\}$, $Y^i$ is solution of \eqref{eq:Yi}.
 We observe that by Remark \ref{RDelta}, $S^N(\mathbf{\xi})$ and $S^N(\mathbf{Y})$ are measurable maps from $(\Omega,\shf)$ to $(\shp(\shc^d),\shb(\shp(\shc^d)))$, and they are such that $S^N(\xi),S^N({\bf{Y}}) \in \shp_2(\shc^d)$ $\P$-a.s.
A solution ${\bf{\xi}}:=(\xi^{i,N})_{i=1,\cdots ,N}$  of \eqref{eq:XIi} is called
{\bf{weakly interacting particle system}}. \\
The first line equation of \eqref{eq:XIi} is in fact a path-dependent stochastic differential equation.
We claim that its coefficients 
 are measurable. Indeed, the map $(t,\bar{\xi}) \mapsto (S^{N}(\bar{\xi}),t,\bar{\xi}_t^i,)$ being continuous from $([0,T] \times \shc^{dN},\shb([0,T]) \otimes \shb(\shc^{dN}))$ to $( \shp(\shc^{d}) \times [0,T] \times \R^d , \shb(\shp(\shc^d)) \otimes \shb([0,T]) \otimes \shb(\R^d)) $ for all $i \in \{1,\cdots,N\}$, by composition with the continuous map $(m,t,y) \mapsto u^m(t,y)$ (see Lemma \ref{lem:uu'} $(3.)$) we deduce the continuity of $(t,\bar{\xi}) \mapsto (u_t^{S^N(\bar{\xi})}(\bar{\xi}^i_t))_{i=1,\cdots,N}$, and so the measurability from $([0,T] \times \shc^{dN},\shb([0,T]) \otimes \shb(\shc^{dN}))$ to $(\R,\shb(\R))$.
\\
In the sequel, for simplicity we denote $\bar{\xi}_{r \le s}:= (\bar{\xi}^i_{r \le s})_{1 \le i \le N}$.  
We remark that,  by Proposition \ref{non-anticip} and Remark \ref{R38}, we have
\begin{eqnarray}
\label{NonAntip}
 \left (u^{S^N(\bar{\xi})}_s(\bar{\xi}^{i}_s)\right )_{i=1,\cdots N} = \left (u^{S^N(\mathbf{\bar{\xi}_{r \le s} })}_s(\bar{\xi}^{i}_s)\right )_{i=1,\cdots N},
\end{eqnarray}
for any $s \in [0,T], \; \bar{\xi} \in {\mathcal C}^{dN}$ and so stochastic integrands of \eqref{eq:XIi} 
are adapted (so progressively measurable being continuous in time) and so the corresponding
 It\^o integral makes sense. We discuss below its well-posedness. \\
The fact that \eqref{eq:XIi} has a unique strong solution $(\xi^{i,N})_{i=1,\cdots N}$ holds true because of the following arguments.
\begin{enumerate}
 \item  $\Phi$ and $g$ are Lipschitz. Moreover 
 the map $\mathbf{\bar{\xi}}_{r\leq s}
\mapsto \left (u^{S^N(\mathbf{\bar{\xi}_{r \le s} })}_s(\bar{\xi}^{i}_s) \right)_{i=1,\cdots,N}$
is Lipschitz. \\
Indeed, for given $(\xi_{r \leq s},\eta_{r \leq s}) \in \shc^{dN} \times \shc^{dN} $, $s \in [0,T]$, by using successively inequality \eqref{eq:uu'} of Lemma \ref{lem:uu'} and Remark \ref{RADelta}, for all $i \in \{1,\cdots,N\}$ we have 
\begin{eqnarray}
\label{Lip_u_PD}
\vert u_s^{S^N(\xi_{r \leq s})}(\xi^i_{t}) - u_s^{S^N(\eta_{r \leq s})}(\eta^i_{t})  \vert & \leq & \sqrt{C_{K,\Lambda}(T)} \left(  \vert \xi^i_s - \eta_s^i \vert + \frac{1}{N} \sum_{j=1}^N \sup_{0 \leq r \leq s}\vert \xi^j_r- \eta^j_r \vert \right) \nonumber \\
& \leq & 2\sqrt{C_{K,\Lambda}(T)} \max_{j=1,\cdots,N} \sup_{0 \leq r \leq s}\vert \xi^j_r-\eta^j_r \vert \ .
\end{eqnarray}
Finally  the functions 
\begin{eqnarray*}
 {\mathbf{\bar{\xi}}}_{r\leq s}  &\mapsto& \left (\Phi(u_s^{S^N(\mathbf{\bar{\xi}_r, r \le s})}
(\bar{\xi}^{i}_s))\right )_{i=1,\cdots N} \\
 \mathbf{\bar{\xi}}_{r\leq s}  &\mapsto& \left (g(u_s^{S^N
(\mathbf{\bar{\xi}_r, r \le s})}
(\bar{\xi}^{i}_s))\right )_{i=1,\cdots N}
\end{eqnarray*}
are uniformly Lipschitz and bounded.
\item 
A classical argument of well-posedness for systems
of path-dependent stochastic differential equations
with Lipschitz dependence on the sup-norm of the path (see Chapter~V, Section~2.11, Theorem~11.2 page 128 in~\cite{rogers_v2}).
\end{enumerate}

\begin{thm} \label{TPC}
Let us suppose the validity of Assumption \ref{ass:main}. Let $N$ be a fixed positive integer.
Let $(Y^i)_{i=1, \cdots,N}$ (resp. ($(\xi^{i,N})_{i=1, \cdots,N}$) be the solution of \eqref{eq:Yi} (resp. \eqref{eq:XIi}),
$m^0$ is defined after \eqref{eq:Yi}. The following assertions hold.
\begin{enumerate}
\item If $\shf(K)$ is in $L^1(\R^d)$, there is a constant $C= C(\Phi, g, \Lambda, K, T)$ such that, for all $i=1,\cdots,N$ and $t \in [0,T]$,  
\begin{eqnarray}
\label{eq:xiYuFinalb1}
 \E[ \Vert u_t^{S^N(\xi)}-u^{m_0}_t\Vert_{\infty} ^2] 
 & \leq & \frac{C}{N} \\
  \label{eq:xiYuFinalb2}
\E[\sup_{0 \leq s\leq t} \vert \xi^{i,N}_s-Y^i_s\vert ^2] & \leq & \frac{C}{N}\ ,
\end{eqnarray}
where $C$ is a finite positive constant only depending on $M_K,M_{\Lambda},L_K,L_{\Lambda},T$.
\item  If $K$ belongs to $W^{1,2}(\R^d)$, there is a constant $C= C(\Phi, g, \Lambda, K, T)$ such that, for all $t \in [0,T]$,
 \begin{eqnarray}
 \label{eq:xiYuFinal3}
 \E[ \Vert u_t^{S^N(\xi)}-u^{m_0}_t\Vert_{2} ^2] & \leq & \frac{C}{N} \ ,
 \end{eqnarray}
where $C$ is a finite positive constant only depending on $M_K,M_{\Lambda},L_K,L_{\Lambda},T$ and $\Vert \nabla K\Vert_2$.
\end{enumerate} 
\end{thm}
The validity of \eqref{eq:xiYuFinalb1} and \eqref{eq:xiYuFinalb2}
 will be the consequence of the significant more general proposition below.
\begin{prop}\label{PMoreGeneral}
Let us suppose the validity of Assumption \ref{ass:main}. Let $N$ be a fixed positive integer. Let $(W^{i,N})_{i=1,\cdots,N}$ be a family of $p$-dimensional standard Brownian motions (not necessarily independent).
We consider the processes $(\bar{Y}^{i,N})_{i=1,\cdots,N}$, such that 
for each $i \in \{1,\cdots,N\}$, $\bar Y^{i,N}$ is the unique strong solution of
\begin{equation}
\label{eq:Yi-general}
\left \{
\begin{array}{l}
\bar{Y}^{i,N}_t=Y_0^i+\int_0^t\Phi(u^{m^{i,N}}_s(\bar{Y}^{i,N}_s))dW^{i,N}_s+\int_0^tg(u^{m^{i,N}}_s(\bar{Y}^{i,N}_s))ds, \quad \textrm{ for all } t \in [0,T] \\
u^{m^{i,N}}_t(y)={\displaystyle \E\left [ K(y-\bar{Y}^{i,N}_t)V_t\big (\bar{Y}^{i,N},u^{m^{i,N}}(\bar{Y}^{i,N})\big ) \right ]}\ ,\quad \textrm{with}\ m^{i,N}:=\mathcal{L}(\bar{Y}^{i,N})\ ,
\end{array}
\right .
\end{equation}
recalling that $V_t\big (Y^{i,N},u^{m^{i,N}}(Y^{i,N})\big )=\exp \big( \int_0^t\Lambda_s(Y^{i,N}_s,u^{m^{i,N}}_s(Y^{i,N}_s))ds \big)$, $({Y}^i_0)_{i=1,\cdots,N}$ being the family of i.i.d. r.v. initializing the system~\eqref{eq:Yi}.
Below, we consider the system of equations
\eqref{eq:XIi},
where the processes  $W^i$ are replaced  by
 $W^{i,N}$, i.e.
\begin{equation}
\label{eq:XIiN}
\left \{
\begin{array}{l}
\xi^{i,N}_t= \xi^{i,N}_0 +\int_0^t\Phi(u^{S^N(\mathbf{\xi})}_s(\xi^{i,N}_s))dW^{i,N}_s+\int_0^tg(u^{S^N(\mathbf{\xi})}_s(\xi^{i,N}_s))ds  \\
\xi^{i,N}_0 = \bar{Y}^{i,N}_0 \\
u^{S^N(\mathbf{\xi})}_t(y)={\displaystyle \frac{1}{N}\sum_{j=1}^N  K(y-\xi^{j,N}_t) V_t\big (\xi^{j,N},u^{S^N(\mathbf{\xi})}(\xi^{j,N})\big ) }\ .
\end{array}
\right .
\end{equation}
Then the following assertions hold.
\begin{enumerate}
\item For any $i=1,\cdots N$, $(\bar{Y}^{i,N}_t)_{t\in [0,T]}$ have the same law $m^{i,N}=m^0$, where $m^0$ is the common law of processes $(Y^i)_{i=1,\cdots,N}$ defined by the system \eqref{eq:Yi}.
\item Equation~\eqref{eq:XIiN} admits a unique strong solution.
	\item Suppose moreover that $\shf(K)$ is in $L^1(\R^d)$. Then, there is a constant $C= C(K,\Phi, g, \Lambda, T)$ such that, for all $t \in [0,T]$,
\begin{eqnarray}
\label{eq:xiYuFinalGeneral}
\sup_{i=1,\dots,N} \E[\sup_{0 \leq s\leq t} \vert \xi^{i,N}_s-\bar{Y}^{i,N}_s\vert ^2] + \E[ \Vert u_t^{S^N(\xi)}-u^{m^0}_t\Vert_{\infty} ^2] 
 & \leq & C\sup_{ \underset{\Vert \varphi \Vert_{\infty} \le 1}{\varphi \in \shc_b(\shc^d)}} \E[ \vert \langle S^N(\mathbf{\bar{Y}})  - m^0 , \varphi \rangle \vert^2 ] \ , \nonumber \\ 
\end{eqnarray}
with $\displaystyle{S^{N}({\bf{\bar{Y}}}) := \frac{1}{N} \sum_{j=1}^N \delta_{\bar{Y}^{j,N}}}$. 
\end{enumerate}
\end{prop}
\begin{rem} \label{RMoreGeneral}
\begin{enumerate}
\item 
The r.h.s. of \eqref{eq:xiYuFinalGeneral} can be easily  bounded if the processes $(\bar{Y}^{i,N})_{i=1,\cdots,N}$ are i.i.d.  
Indeed, as in the proof of the Strong Law of Large Numbers,
\begin{eqnarray}
\label{majorNormFaible}
  \sup_{ \underset{\Vert \varphi \Vert_{\infty} \le 1}{\varphi \in \shc_b(\shc^d)}} \E[ \vert \langle S^N(\mathbf{\bar{Y}})  - m^0 , \varphi \rangle \vert^2 ] 
  & = & \sup_{ \underset{\Vert \varphi \Vert_{\infty} \le 1}{\varphi \in \shc_b(\shc^d)}} \E \left[ \left( \frac{1}{N} \sum_{j=1}^{N} \varphi(\bar{Y}^{j,N}) - \E[ \varphi(\bar{Y}^{j,N})] \right)^2 \right] \nonumber \\
  & = & \sup_{ \underset{\Vert \varphi \Vert_{\infty} \le 1}{\varphi \in \shc_b(\shc^d)}}  Var(\frac{1}{N} \sum_{j=1}^N \varphi(\bar{Y}^{j,N}) ) \nonumber \\
  & = & \sup_{ \underset{\Vert \varphi \Vert_{\infty} \le 1}{\varphi \in \shc_b(\shc^d)}}  Var( \frac{1}{N} \varphi(\bar{Y}^{1,N}) ) \nonumber \\
  & \leq & \frac{1}{N} \ .
\end{eqnarray}
\item 
In fact Proposition \ref{PMoreGeneral} does not require the independence of $(\bar Y^{i,N})_{i=1,\cdots N}$.
Indeed, the convergence of the numerical approximation $u_t^{S^N(\xi)}$ to $u_t^{m_0}$ only requires the convergence of
 $ d_2^\Omega(S^N(\mathbf{\bar Y}), m^0)$ to $0$, where we remind that the distance $ d_2^\Omega$ has been defined at Remark \ref{R25} b). 
 This gives the opportunity to define new numerical schemes for which the convergence of the empirical measure $ S^N(\mathbf{\bar{Y}})$ is verified without i.i.d. particles.
\item Let us come back to the case of independent driving Brownian motions $W^i, i \ge 1$.  Observe that Theorem \ref{TPC} implies the propagation of chaos.
Indeed, for all $k \in \N^{\star}$, \eqref{eq:xiYuFinalb2} implies 
$$
(\xi^{1,N}-Y^1,\xi^{2,N}-Y^2,\cdots,\xi^{k,N}-Y^k) \xrightarrow[\text{$N \longrightarrow + \infty$}]{\text{$L^{2}(\Omega,\shf,\P)$}} 0 \ ,
$$
which gives the convergence in law of the vector $(\xi^{1,N},\xi^{2,N},\cdots,\xi^{k,N})$ to $(Y^1,Y^2,\cdots,Y^k)$. Consequently,
 since $({Y}^{i})_{i=1,\cdots,k}$ are i.i.d. according to $m^0$
\begin{eqnarray}
\label{eq:CVChaos}
(\xi^{1,N},\xi^{2,N},\cdots,\xi^{k,N}) \; \textrm{ converges in law to } (m^0)^{\otimes k} \; \textrm{ when } N \rightarrow +\infty \ .
\end{eqnarray}
\item Proposition~\ref{PMoreGeneral} can be used to provide propagation of chaos results for non exchangeable particle systems. Let us consider $(\bar{Y}^{i,N})_{i=1,\cdots N}$ (resp. $(\xi^{i,N})_{i=1,\cdots N}$) solutions of~\eqref{eq:Yi-general} (resp.~\eqref{eq:XIiN}) where
$$
W^{1,N} := {\displaystyle \frac{\sqrt{N^2-1}}{N}W^1+\frac{1}{N}W^2} \quad \textrm{and} \quad \; \textrm{for } i>1, \; W^{i,N} := W^i \ ,
$$
where we recall that $(W^i)_{i=1,\cdots, N}$ is a sequence of independent $p$ dimensional Brownian motions. In that situation, the particle system $(\xi^{i,N})$ is clearly not exchangeable. However, a simple application of Proposition~\ref{PMoreGeneral} allows us to prove the propagation of chaos. Indeed, let us introduce the sequence of i.i.d processes $(Y^{i})$ solutions of~\eqref{eq:Yi}, Proposition~\ref{PMoreGeneral} yields 
\begin{eqnarray*}
\E[\sup_{s\leq t} \vert \xi^{i,N}_s - Y^i_s\vert^2]&\leq & 
2\E[\sup_{s\leq t} \vert \xi^{i,N}_s - \bar{Y}_s^{i,N}\vert^2]+2\E[\sup_{s\leq t} \vert \bar{Y}^{i,N}_s- Y^i_s\vert^2]\\
&\leq &
C\sup_{ \underset{\Vert \varphi \Vert_{\infty} \le 1}{\varphi \in \shc_b(\shc^d)}} \E[ \vert \langle S^N(\mathbf{\bar{Y}})  - m^0 , \varphi \rangle \vert^2 ]
+\E[\sup_{s\leq t} \vert \bar{Y}^{i,N}_s-Y_s^i\vert^2]\ .
\end{eqnarray*}
To bound the second term on the r.h.s. of the above inequality, observe that $\bar{Y}^{i,N}=Y^i$ for $i>1$ and for $i=1$, notice that simple computations, involving BDG inequality, imply $\E [\sup_{s\leq t} \vert Y^{1,N}_s-Y^1_s\vert^2]\leq \frac{C}{N^2}$. \\
Concerning the first term on the r.h.s. of the above inequality, we first observe that the  decomposition holds 
\begin{eqnarray*}
\langle S^N(\mathbf{\bar{Y}}) - m^0,\varphi \rangle & = & \frac{1}{N}\sum_{i=1}^N \varphi(\bar Y^{i,N}) - \langle m^0,\varphi \rangle \\
& = & \frac{1}{N} \Big( \varphi(\bar Y^{1,N}) - \E[ \varphi(\bar Y^{1,N}) ]\Big)+\frac{N-1}{N} \Big( \frac{1}{N-1}\sum_{i=2}^N \varphi(\bar Y^{i,N}) - \langle m^0,\varphi \rangle \Big) \ ,
\end{eqnarray*}
for all $\varphi \in \shc_b(\shc^d)$. 
We remind that $\bar Y^{1,N},\cdots, \bar Y^{N,N}$ have the same law $m^0$ 
taking into account item 1. of Proposition \ref{PMoreGeneral}. It follows that
\begin{eqnarray} \label{ExRem}
\sup_{ \underset{\Vert \varphi \Vert_{\infty} \le 1}{\varphi \in \shc_b(\shc^d)}} \E \Big[ \vert \langle S^N(\mathbf{\bar{Y}})  - m^0 , \varphi \rangle \vert^2 \Big] \leq \frac{6}{N^2} + \frac{3(N-1)^2}{N^2} \sup_{ \underset{\Vert \varphi \Vert_{\infty} \le 1}{\varphi \in \shc_b(\shc^d)}} \E \Big[ \vert \langle \frac{1}{N-1} \sum_{j=2}^N \delta_{\bar{Y}^{j,N}} - m^0, \varphi \rangle \vert^2  \Big] .
\end{eqnarray}
Since the r.v. $(\bar Y^{2,N},\cdots, \bar Y^{N,N})$ are i.i.d. according to $m^0$, \eqref{ExRem} and item 1. of Remark \ref{RMoreGeneral} give us
$$
\sup_{ \underset{\Vert \varphi \Vert_{\infty} \le 1}{\varphi \in \shc_b(\shc^d)}} \E[ \vert \langle S^N(\mathbf{\bar{Y}})  - m^0 , \varphi \rangle \vert^2 ] \leq \frac{C}{N} \ ,
$$
which leads to a similar inequality as \eqref{eq:xiYuFinalb2} in Theorem \ref{TPC}. The same reasoning as in item 3. above implies propagation of chaos. 
 \end{enumerate} 
\end{rem}
\begin{proof}[Proof of Proposition \ref{PMoreGeneral}]  Let us fix $t \in [0,T]$. 
In this proof, $C$ is a real positive constant ($C= C(\Phi, g, \Lambda, K, T)$)
 which may change from line to line. \\
  Equation~\eqref{eq:Yi-general} has $N$ blocks, numbered by $1 \le i \le N$.
Item 2. of Proposition~\ref{prop:NSDE} gives 
uniqueness in law for each block equation,
  which implies that for any $i=1,\cdots N$, $m^{i,N}=m^0$ and proves the
 first item. \\
 Concerning the strong existence and pathwise uniqueness of \eqref{eq:XIiN}, 
the same argument as for the well-statement of \eqref{eq:XIi} operates. 
 The only difference consists 
in the fact that the Brownian motions may be correlated.
A very close proof as the one of
 Theorem 11.2 page 128 in
 \cite{rogers_v2} works: the main argument is 
  the multidimensional BDG inequality, see e.g. Problem 3.29 of  
\cite{karatshreve}.
From now on let us focus on the proof of inequality \eqref{eq:xiYuFinalGeneral}. On the one hand, since the map $(t,\bar{\xi}) \in [0,T] \times \shc^{dN} \mapsto~(u_t^{S^N(\bar{\xi})}(\bar{\xi}^i_t))_{i=1,\cdots,N}$ is measurable and satisfies the non-anticipative property \eqref{NonAntip}, the first assertion of Lemma \ref{lem:yy'} gives for all $i \in \{1,\cdots,N\}$
\begin{eqnarray}
\label{PropaChaos}
\E[\sup_{0 \leq s\leq t} \vert \xi^{i,N}_s-\bar{Y}^{i,N}_s\vert ^2]
& \leq & C
\E[\int_0^t
\vert u_s^{S^N(\mathbf{\xi})}(\xi^{i,N}_s)-u^{m^0}_s(\bar{Y}^{i,N}_s)\vert ^2
ds ] \nonumber \\
& \leq & C
\int_0^t \E[
\vert u_s^{S^N(\mathbf{\xi})}(\xi^{i,N}_s)-u^{m^0}_s(\xi^{i,N}_s)\vert ^2]
ds + \int_0^t
\E[ \vert u^{m^0}_s(\xi^{i,N}_s)-u^{m^0}_s(\bar{Y}^{i,N}_s)\vert ^2]
ds \nonumber \\
& \leq & C \int_0^t \left( \E[ \Vert u_s^{S^N(\mathbf{\xi})}-u^{m^0}_s \Vert_{\infty}^2 ] + 
\E[ \sup_{0 \leq r \leq s} \vert \xi^{i,N}_r - \bar{Y}^{i,N}_r \vert^2 ] \right) ds,\quad \textrm{ by \eqref{eq:uu'} ,} \nonumber \\
\end{eqnarray}
which implies
\begin{eqnarray}
\label{um0um0}
\sup_{i=1,\cdots,N} \E[\sup_{0 \leq s\leq t} \vert \xi^{i,N}_s-\bar{Y}^{i,N}_s\vert ^2] \leq C \int_0^t \left( \E[ \Vert u_s^{S^N(\mathbf{\xi})}-u^{m^0}_s \Vert_{\infty}^2 ] + \sup_{i=1,\cdots,N} \E[ \sup_{0 \leq r \leq s} \vert \xi^{i,N}_r - \bar{Y}^{i,N}_r \vert^2 ] \right) ds
\end{eqnarray}
On the other hand, using inequalities \eqref{eq:uu'} (applied pathwise with $m = S^N(\xi)(\bar{\omega})$ and $m' = S^N({\bf{\bar{Y}}})(\bar{\omega})$) and \eqref{eq:uu'Linf} (with the random measure $ \eta = S^N({\bf{\bar{Y}}})$ and $m = m^0$) in Lemma \ref{lem:uu'}, yields
\begin{eqnarray}
\label{uSNxium0}
\E[ \Vert u_t^{S^N(\mathbf{\xi})}-u^{m^0}_t \Vert_{\infty}^2 ] & \leq & 2 
\E\left[ \Vert u_t^{S^N(\mathbf{\xi})}-u^{S^N(\mathbf{\bar{Y}})}_t \Vert_{\infty}^2 \right]
+ 2  \E[ \Vert u^{S^N(\mathbf{\bar{Y}})}_t-u_t^{m^0} \Vert_{\infty}^2 ] \nonumber \\
& \leq & 2 C \E[ \vert W_t(S^N(\mathbf{\xi}),S^N(\mathbf{\bar{Y}})) \vert^2] + 2 C \sup_{ \underset{\Vert \varphi \Vert_{\infty} \le 1}{\varphi \in \shc_b(\shc^d)}} \E{ [ \vert \langle S^N(\mathbf{\bar{Y}})  - m^0 , \varphi \rangle \vert^2 ]} \nonumber \\
& \leq & \frac{2 C}{N} \sum_{i=1}^N \E[\sup_{0 \leq s\leq t} \vert \xi^{i,N}_s-\bar{Y}^{i,N}_s \vert ^2] + C \sup_{ \underset{\Vert \varphi \Vert_{\infty} \le 1}{\varphi \in \shc_b(\shc^d)}} \E[ \vert \langle S^N(\mathbf{\bar{Y}})  - m^0 , \varphi \rangle \vert^2 ] \nonumber \\
& \leq & 2 C \sup_{i=1,\cdots,N} \E[\sup_{0 \leq s\leq t} \vert \xi^{i,N}_s-\bar{Y}^{i,N}_s \vert ^2] + C \sup_{ \underset{\Vert \varphi \Vert_{\infty} \le 1}{\varphi \in \shc_b(\shc^d)}} \E[ \vert \langle S^N(\mathbf{\bar{Y}})  - m^0 , \varphi \rangle \vert^2 ],
\end{eqnarray}
where the third inequality follows from Remark \ref{RADelta}. \\
Let us introduce the non-negative function $G$ defined on $[0,T]$ by 
$$
G(t):= \E[ \Vert u_t^{S^N(\mathbf{\xi})}-u^{m^0}_t \Vert_{\infty}^2 ] + \sup_{i=1,\cdots,N} \E[\sup_{0 \leq s\leq t} \vert \xi^{i,N}_s-\bar{Y}^{i,N}_s \vert ^2] \ .
$$
From  inequalities \eqref{um0um0} and \eqref{uSNxium0} that are valid for all $t \in [0,T]$, we obtain
\begin{eqnarray}
G(t) & \leq & 
(2C + 1) \sup_{i=1,\cdots,N} \E[\sup_{0 \leq s\leq t} \vert \xi^{i,N}_s-\bar{Y}^{i,N}_s \vert ^2] + C \sup_{ \underset{\Vert \varphi \Vert_{\infty} \le 1}{\varphi \in \shc_b(\shc^d)}} \E[ \vert \langle S^N(\mathbf{\bar{Y}})  - m^0 , \varphi \rangle \vert^2 ] \nonumber \\
& \leq & C \int_0^t \left( \E[ \Vert u_s^{S^N(\mathbf{\xi})}-u^{m^0}_s \Vert_{\infty}^2 ] + \sup_{i=1,\cdots,N} \E[ \sup_{0 \leq r \leq s} \vert \xi^{i,N}_r - \bar{Y}^{i,N}_r \vert^2 ] \right) ds \nonumber \\
& & + \; C \sup_{ \underset{\Vert \varphi \Vert_{\infty} \le 1}{\varphi \in \shc_b(\shc^d)}} \E[ \vert \langle S^N(\mathbf{\bar{Y}})  - m^0 , \varphi \rangle \vert^2 ] \nonumber \\
& \leq & C \int_0^t G(s) ds + C \sup_{ \underset{\Vert \varphi \Vert_{\infty} \le 1}{\varphi \in \shc_b(\shc^d)}} \E[ \vert \langle S^N(\mathbf{\bar{Y}})  - m^0 , \varphi \rangle \vert^2 ] \ .
\end{eqnarray}
By Gronwall's lemma, for all $t \in [0,T]$, we obtain
\begin{eqnarray}
\label{PCh1}
\E[ \Vert u_t^{S^N(\mathbf{\xi})}-u^{m^0}_t \Vert_{\infty}^2 ] + \sup_{i=1,\cdots,N} \E[\sup_{0 \leq s\leq t} \vert \xi^{i,N}_{s}-\bar{Y}^{i,N}_{s} \vert ^2] \leq Ce^{Ct} \sup_{ \underset{\Vert \varphi \Vert_{\infty} \le 1}{\varphi \in \shc_b(\shc^d)}} \E[ \vert \langle S^N(\mathbf{\bar{Y}})  - m^0 , \varphi \rangle \vert^2 ] \ .
\end{eqnarray}
\end{proof}
\begin{proof}[Proof of Theorem \ref{TPC}]
To prove inequalities \eqref{eq:xiYuFinalb1} and \eqref{eq:xiYuFinalb2}, we can deduce them from 
Proposition \ref{PMoreGeneral}. Indeed, we have to bound the quantity $\displaystyle{\sup_{ \underset{\Vert \varphi \Vert_{\infty} \le 1}{\varphi \in \shc_b(\shc^d)}} \E[ \vert \langle S^N(\mathbf{Y})  - m^0 , \varphi \rangle \vert^2 ] } $ . 
To this end, it is enough to apply Proposition \ref{PMoreGeneral},
in particular \eqref{eq:xiYuFinalGeneral},
 by setting for all $i \in \{1,\cdots,N\}$, $W^{i,N} := W^i$. Pathwise uniqueness of systems \eqref{eq:Yi} and \eqref{eq:Yi-general} implies $\bar{Y}^{i,N} = Y^i$ for all $i \in \{1,\cdots,N\}$.
Since $(Y^i)_{i=1,\cdots,N}$ are i.i.d. according to $m^0$, inequalities \eqref{eq:xiYuFinalb1} and \eqref{eq:xiYuFinalb2} follow from item 1. of Remark \ref{RMoreGeneral}. \\
It remains now to prove \eqref{eq:xiYuFinal3}. First, the  inequality  
\begin{equation}
\label{SNxim0L2}
\E[ \Vert u_t^{S^N(\xi)}-u^{m_0}_t\Vert_{2} ^2] \leq 2\E[ \Vert u_t^{S^N(\xi)}-u^{S^N(\mathbf{Y})}_t\Vert_{2} ^2] + 2\E[ \Vert u_t^{S^N(\mathbf{Y})}-u^{m_0}_t\Vert_{2} ^2],
\end{equation}
holds for all $t \in [0,T]$.
Using  inequality \eqref{uu'L2} of Lemma \ref{lem:uu'}, for all $t \in [0,T]$,
we get
\begin{eqnarray}
\label{SNxiSNYL2}
\E[ \Vert u_t^{S^N(\xi)}-u^{S^N(\mathbf{Y})}_t\Vert_{2} ^2] & \leq & C \E{[ W_t(S^N(\xi),S^N(\mathbf{Y}))^2]} \nonumber \\
& \leq & C \frac{1}{N} \sum_{j=1}^N \E{[ \sup_{0 \leq r \leq t} \vert \xi^{j,N}_r-Y^j_r \vert^2 ]} \nonumber \\
& \leq & \frac{C}{N} \ ,
\end{eqnarray}
where the latter inequality is obtained through \eqref{eq:xiYuFinalb2}. The second term of the r.h.s. in \eqref{SNxim0L2} needs more computations. Let us fix $i \in  \{1,\cdots,N\}$.
First, 
\begin{equation}
\label{major-uu'L2}
\E[ \Vert u_t^{S^N(\mathbf{Y})}-u^{m_0}_t\Vert_{2} ^2] \leq 2 \left( \E{[ \Vert A_t \Vert_2^2 ]} + \E{[ \Vert B_t \Vert_2^2 ]}  \right) \ ,
\end{equation}
where,
for all $t \in [0,T]$
\begin{equation}
\label{def-AB}
\left \{
\begin{array}{l}
A_t(x):= {\displaystyle \frac{1}{N}\sum_{j=1}^N K(x-Y^j_t)\Big [ V_t\big (Y^j,u^{S^N(\mathbf{Y})}(Y^j)\big ) - V_t\big (Y^j,u^{m^0}(Y^j)\big )\Big ]} \\
B_t(x) := \displaystyle{ \frac{1}{N} \sum_{j=1}^N  K(x-Y^j_t) V_t\big (Y^j,u^{m^0}(Y^j)\big )-\E\Big [K(x-Y^1_t) V_t\big (Y^1,u^{m^0}(Y^1)\big ) \Big]} \ ,
\end{array}
\right .
\end{equation}
where we remind that $m^0$ is the common law of all the processes $Y^i, 1 \le i \le N$. \\
To simplify notations, we set $P_j(t,x) := K(x-Y^j_t) V_t\big (Y^j,u^{m^0}(Y^j)\big )-\E\Big [K(x-Y^1_t) V_t\big (Y^1,u^{m^0}(Y^1)\big ) \Big]$ for all $j \in \{1,\cdots,N\}$, $x \in \R^d$ and $t \in [0,T]$. \\
We observe that for all $x \in \R^d, t \in [0,T]$, $(P_j(t,x))_{j=1,\cdots,N}$ are i.i.d. centered r.v. . Hence, 
$$
\E{[B_t(x)^2]} = \frac{1}{N} \E{[P_1^2(t,x)]} \leq \frac{4}{N} \E{[K(x-Y^1_t)^2 V_t\big (Y^1,u^{m^0}(Y^1)\big )^2]} \leq \frac{4M_K e^{2tM_{\Lambda}}}{N} \E{[ K(x-Y^1_t) ]}
$$ 
and by integrating each side of the  inequality above w.r.t. $x \in \R^d$, we obtain
\begin{equation}
\label{majorA}
 \E{\left[ \int_{\R^d} \vert B_t(x) \vert^2 dx\right]} = \int_{\R^d} \E{[\vert B_t(x) \vert^2]} dx \leq \frac{4M_K e^{2tM_{\Lambda}}}{N} \ ,
\end{equation}
where we have used that $\Vert K \Vert_1 = 1$. \\
Concerning $A_t(x)$, 
\begin{eqnarray}
\label{majorB}
\vert A_t(x) \vert^2 & \leq & \frac{1}{N}\sum_{j=1}^N K(x-Y^j_t)^2\Big [ V_t\big (Y^j,u^{S^N(\mathbf{Y})}(Y^j)\big )
- V_t\big (Y^j,u^{m^0}(Y^j)\big )\Big ]^2 \nonumber \\
& = & \frac{1}{N}\sum_{j=1}^N K(x-Y^j_t) K(x-Y^j_t)\Big [ V_t\big (Y^j,u^{S^N(\mathbf{Y})}(Y^j)\big )
- V_t\big (Y^j,u^{m^0}(Y^j)\big )\Big ]^2 \nonumber \\
& \leq & \frac{M_K T}{N} e^{2tM_{\Lambda}}L_{\Lambda}^2 \sum_{j=1}^N K(x-Y^j_t) \int_0^t \vert u_r^{S^N(\mathbf{Y})}(Y^j_r)-u_r^{m^0}(Y^j_r) \vert^2 dr \nonumber \\
& \le & \frac{M_K T}{N} e^{2tM_{\Lambda}}L_{\Lambda}^2 \sum_{j=1}^N K(x-Y^j_t) \int_0^t \Vert u_r^{S^N(\mathbf{Y})}-u_r^{m^0} \Vert_{\infty}^2 dr,  \nonumber \\
\end{eqnarray}
where the third inequality comes from \eqref{eq:Vmajor2}. Integrating w.r.t. $x \in \R^d$ and taking expectation on each side of the above inequality gives us,
 for all $t \in [0,T]$
\begin{eqnarray}
\label{majorBFinal}
\E{[ \int_{\R^d} \vert A_t(x) \vert^2 dx]} & \leq & M_K T e^{2tM_{\Lambda}}L_{\Lambda}^2 \int_0^t \E{[ \Vert u_r^{S^N(\mathbf{Y})}-u_r^{m^0} \Vert_{\infty}^2 ]} dr \nonumber \\
& \leq & M_K T^2 e^{2tM_{\Lambda}}L_{\Lambda}^2 C \sup_{ \underset{\Vert \varphi \Vert_{\infty} \le 1}{\varphi \in \shc_b(\shc^d)}} 
  \E{ [ \vert \langle S^N(\mathbf{Y})  - m^0 , \varphi \rangle \vert^2 ]} \nonumber \\
  & \leq & \frac{M_K T^2 e^{2tM_{\Lambda}}L_{\Lambda}^2 C}{N} \ ,
\end{eqnarray}
where we have used \eqref{eq:uu'Linf} of Lemma \ref{lem:uu'} for the second inequality above and \eqref{majorNormFaible} for the last one. To conclude, it is enough to replace \eqref{majorA}, \eqref{majorBFinal} in \eqref{major-uu'L2}, and inject \eqref{SNxiSNYL2}, \eqref{major-uu'L2} in \eqref{SNxim0L2}.
\end{proof}

\section{Particle algorithm }
\label{S8}

\setcounter{equation}{0}

\subsection{Time discretization of the particle system}
\label{time-discretization} 

In this Section Assumption \ref{ass:main}. will be in force again.
Let $(Y_0^i)_{i=1,\cdots,N}$ be i.i.d. r.v. distributed according to $\zeta_0$.
In this section we are interested in discretizing the interacting particle system
\eqref{eq:XIi} solved by the processes $\xi^{i,N}, 1 \le i \le N$.
Let us consider a regular time grid $0=t_0\leq \cdots\leq t_k=k\delta t\leq \cdots \leq t_n=T$, with $\delta t=T/n$.

We introduce the  continuous $\R^{dN}$-valued process $(\tilde \xi_t)_{t\in [0,T]}$ and the family of nonnegative functions $(\tilde v_t)_{t\in [0,T]}$ defined on $\R^d$ such that 
\begin{equation}
\label{eq:tildeYu}
\left \{\begin{array}{l}
\tilde \xi^{i,N}_{t}=\tilde \xi^{i,N}_{0}+\int_0^t \Phi(\tilde v_{r(s)}(\tilde \xi^{i,N}_{r(s)}))dW^i_{s}+\int_0^t g(\tilde v_{r(s)}(\tilde \xi^{i,N}_{r(s)}))ds \\ 
\tilde{\xi}^{i,N}_0 = Y^i_0 \\
\tilde v_{t}(y)=\frac{1}{N}\sum_{j=1}^N K(y-\tilde \xi^{j,N}_{t})\exp\big \{\int_0^t \Lambda(r(s),\tilde \xi^{j,N}_{r(s)},\tilde v_{r(s)}(\tilde \xi^{j,N}_{r(s)}))\,ds\big \} \ ,\ \textrm{for any}\ t\in [0,T],
\end{array}
\right .
\end{equation}
where $r:\,s\in [0,T]\,\mapsto r(s)\in \{t_0,\cdots t_n\}$ is the piecewise constant function such that $r(s)=t_k$ when $s\in [t_k,t_{k+1}[$. We can observe that $(\tilde{\xi}^{i,N})_{i=1,\cdots,N}$ is an adapted and continuous process. 
The interacting particle system $(\tilde \xi^{i,N})_{i=1,\cdots N}$ can be simulated perfectly at the discrete instants $(t_k)_{k=0,\cdots,n}$ from independent standard and centered Gaussian random variables. We will prove that this interacting particle system provides an approximation to $(\xi^{i,N})_{i=1,\cdots N}$, solution of the system~\eqref{eq:XIi} which converges at a rate of order $\sqrt{\delta t}$.
\begin{prop}
\label{prop:DiscretTime}
Suppose that  Assumption~\ref{ass:main} holds excepted 2. which is replaced by the following:
there exists a positive real $L_{\Lambda}$ such that for any $(t,t',y,y',z,z')\in [0,T]^2\times (\R^d)^2\times (\R^+)^2$, 
$$
\vert \Lambda (t,y,z)-\Lambda (t',y',z')\vert \leq L_{\Lambda}\,( \vert t-t'\vert+\vert y-y'\vert +\vert z-z'\vert )\ .
$$
Then, the time discretized particle system~\eqref{eq:tildeYu} converges to the original particle system~\eqref{eq:XIi}. More precisely, we have the  estimates 
\begin{equation}
\label{eq:tildeYMajor}
\E[\Vert \tilde v_t-u_t^{S^N(\xi)}\Vert_{\infty}^2]+\sup_{i=1,\cdots N} \E\left [\sup_{s\leq t}\vert \tilde  \xi^{i,N}_{s}-  \xi^{i,N}_{s}\vert^2\right ]\leq C\delta t\ ,
\end{equation}
where $C$ is a finite positive constant only depending on $M_K,M_{\Lambda},L_K,L_{\Lambda},T$.  \\
If we assume moreover that $K\in W^{1,2}(\R^d)$, then the following Mean Integrated Squared Error (MISE) estimate holds: 
\begin{equation}
\label{eq:tildeYMajorBis}
\E[\Vert \tilde v_t-u_t^{S^N(\xi)}\Vert_{2}^2]\leq C\delta t\ ,
\end{equation}
where $C$ is a finite positive constant only depending on $M_K,M_{\Lambda},L_K,L_{\Lambda},T$ and $\Vert \nabla K\Vert_2$. 
\end{prop}

\begin{rem} \label{RPIDE}
We keep in mind the probability measure $m_0$ defined at Section \ref{SChaos},
which is the law of processes $Y^i$, solutions of \eqref{eq:Yi}.
 We claim that $\tilde v$ can be used as a numerical approximation to the function $u^{m_0}$;
we remind that, by Theorem \ref{thm-PIDE}  $u^{m_0}$ is associated with a solution $\gamma^{m_0}$
of the PIDE~\eqref{PIDE} via the relation $u^m = K * \gamma^m$.

The committed expected squared error 
$\E[ \Vert u^{m_0}_t - \tilde v_t \Vert_{\infty} ^2]$ 
is lower than $C(T)(\delta t + 1/N)$, for a given finite constant $C(T)$. Indeed, it is bounded as follows:
$$
 \E[ \Vert u^{m_0}_t - \tilde v_t \Vert_{\infty} ^2] \leq 2 \E[ \Vert u^{m_0}_t - u_t^{S^N(\xi)}\Vert_{\infty} ^2] +
 2 \E[\Vert u_t^{S^N(\xi)} -  \tilde v_t\Vert_{\infty}^2].
$$
The first term  in the r.h.s. of the above inequality comes from the (strong) convergence of the particle system $(\xi^{i,N})_{i=1,\cdots,N}$ to $(Y^i)_{i=1,\cdots,N}$ 
whose  convergence is of order $\frac{1}{N}$, see Theorem \ref{TPC}, inequality \eqref{eq:xiYuFinalb1}. The second term comes from the time discretization whose 
expected  squared error 
is of order $ \delta t$, see Proposition \ref{prop:DiscretTime}, inequality \eqref{eq:tildeYMajor}.
\end{rem}

The proof of Proposition  \ref{prop:DiscretTime} is close to the one of  Theorem \ref{TPC}.
 The idea is first to estimate through Lemma \ref{lem:Discrete}
 the perturbation error due to the time discretization scheme of the SDE in 
  system~\eqref{eq:tildeYu}, and in  the integral
 appearing in the linking equation of~\eqref{eq:tildeYu}. 
Later the propagation of this error  through the dynamical system~\eqref{eq:XIi}
will be controlled via  Gronwall's lemma. 
Lemma \ref{lem:Discrete} below will be proved in the Appendix.

\begin{lem}
\label{lem:Discrete}
Under the same assumptions of Proposition~\ref{prop:DiscretTime}, there exists a finite constant $C>0$ only depending on $T,M_K,L_K,M_\Phi,L_\Phi,M_g,L_g$ and $M_{\Lambda},L_{\Lambda}$ such that for any $t\in [0,T]$, 
\begin{eqnarray}
\label{eq:vtilde}
\E[\vert \tilde \xi ^{i,N}_{r(t)}-\tilde \xi^{i,N}_{t}\vert ^2]\leq C\delta t  \\
\label{E82}
\E[\vert \tilde v_{r(t)}-\tilde v_t\Vert^2_{\infty}\leq C\delta t\\
\label{E83}
\E[\Vert \tilde v_{r(t)}-u^{S^N(\tilde \xi)}_t\Vert^2_{\infty}]\leq C\delta t\ .
\end{eqnarray}
\end{lem} 

\begin{proof}[\bf Proof of Proposition~\ref{prop:DiscretTime}.]
All along this proof, we denote by $C$ a positive constant that only  depends on \\
 $T ,M_K, L_K, M_\Phi,L_\Phi, M_g,L_g$ and $M_{\Lambda}$,$L_{\Lambda}$ and that can change from line to line. Let us fix $t \in [0,T]$.
\begin{itemize}
\item We begin by considering inequality~\eqref{eq:tildeYMajor}.
We first fix $1 \le i \le N$. 
By    \eqref{E82} and \eqref{E83} in Lemma~\ref{lem:Discrete} and Lemma \ref{lem:uu'}, we obtain
\begin{eqnarray}
\label{eq:tildevMajor}
\E[\Vert \tilde v_t-u^{S^N(\xi)}_t\Vert_\infty^2]
&\leq &
2\E[\Vert \tilde v_t-u_t^{S^N( \tilde \xi)}\Vert_{\infty}^2] + 2\E[\Vert u_t^{S^N(\tilde \xi)}-u_t^{S^N(\xi)}\Vert_\infty^2]\nonumber \\
&\leq & 4(\E[\Vert \tilde v_t-\tilde v_{r(t)} \Vert_{\infty}^2] + \E[\Vert \tilde v_{r(t)}-u_t^{S^N( \tilde \xi)}\Vert_{\infty}^2] ) + 2\E[\Vert u_t^{S^N(\tilde \xi)}-u_t^{S^N(\xi)}\Vert_\infty^2]\nonumber \\
& \leq & C\delta t +C\E [\vert W_t\big (S^N(\tilde \xi),S^N(\xi)\big )\vert ^2]\nonumber \\
&\leq & 
C\delta t +C\sup_{i=1,\cdots N} \E[\sup_{s\leq t} \vert \tilde \xi^{i,N}_s-\xi^{i,N}_s\vert ^2]\ ,
\end{eqnarray}
where  the function $u^{S^N(\tilde{\xi})}$ makes sense  since $ \tilde{\xi}$ has almost surely continuous trajectories
 and so $S^N(\tilde{\xi})$ is a random measure which is a.s. 
in $\shp(\shc^d)$. \\
Besides, by the second assertion of Lemma~~\ref{lem:yy'}, we get
\begin{eqnarray}
\label{E310}
\E[\sup_{s\leq t} \vert \tilde \xi^{i,N}_s-\xi^{i,N}_s\vert ^2]
&\leq & C \E\left [ \int_0^t \vert \tilde v_{r(s)}(\tilde \xi^{i,N}_{r(s)})-u_s^{S^N(\xi)}(\xi^{i,N}_s)\vert^2\,ds\right ] + C \int_0^t \E \left[\tilde{\xi}^{i,N}_{r(s)} - \tilde{\xi}^{i,N}_{s} \right] ds + C \delta t^2 \ . \nonumber \\ 
\end{eqnarray}
Concerning the first term in the r.h.s. of \eqref{E310}, we have for all $s \in [0,T]$
\begin{eqnarray}
\label{E311}
\vert \tilde v_{r(s)}(\tilde \xi^{i,N}_{r(s)})-u_s^{S^N(\xi)}(\xi^{i,N}_s)\vert^2 & \leq & 2 \vert \tilde v_{r(s)}(\tilde \xi^{i,N}_{r(s)})-u_s^{S^N(\xi)}(\tilde{\xi}^{i,N}_{r(s)})\vert^2 + 2 \vert u_s^{S^N(\xi)}(\tilde \xi^{i,N}_{r(s)})-u_s^{S^N(\xi)}(\xi^{i,N}_s)\vert^2 \nonumber \\
& \leq & 2 \Vert \tilde{v}_{r(s)} - u^{S^N(\xi)}_s \Vert_{\infty}^2 + 2 C \vert \tilde{\xi}^{i,N}_{r(s)} - \xi^{i,N}_s \vert^2 \ ,
\end{eqnarray}
where the second inequality above follows by Lemma \ref{lem:uu'}, see \eqref{eq:uu'} (Lipschitz property of the function $u^{S^N(\xi)}$).
 Consequently, by \eqref{E310}
\begin{eqnarray}
\E[\sup_{s\leq t} \vert \tilde \xi^{i,N}_s-\xi^{i,N}_s\vert ^2]
&\leq & 
C \left \{ \E\left [ \int_0^t \Vert \tilde v_{r(s)}-u_s^{S^N(\xi)}\Vert^2_{\infty} \,ds \right ]  + \int_0^t \E\left [ \vert \tilde{\xi}^{i,N}_{r(s)}-\xi_{s}^{i,N} \vert^2 \right] \,ds  + \delta t^2 \right \} \nonumber \\  
& \leq & C \left \{ \E\left [ \int_0^t \Vert \tilde v_{r(s)}- \tilde v_s \Vert^2_{\infty}\,ds\right ] + \E\left [ \int_0^t \Vert \tilde v_s-u_s^{S^N(\xi)}\Vert^2_{\infty}\,ds\right ]  \right . \nonumber \\ 
& & \left . + \E\left [\int_0^t\vert \tilde \xi^{i,N}_{r(s)}- \tilde \xi^{i,N}_s\vert ^2\,ds\right ] + \E\left [\int_0^t\vert \tilde \xi^{i,N}_{s}- \xi^{i,N}_s\vert ^2\,ds\right ] + \delta t^2 \right \} \ .
\end{eqnarray}
Using  inequalities \eqref{eq:vtilde} and \eqref{E82} in Lemma~\ref{lem:Discrete}, for all $t \in [0,T]$
we obtain
\begin{equation}
\label{eq:xitildeMajor}
\sup_{i=1,\cdots N}\E[\sup_{s\leq t} \vert \tilde \xi^{i,N}_s-\xi^{i,N}_s\vert ^2]
\leq C\delta t^2 +C \int_0^t \left [\E[\Vert \tilde v_{s}-u_s^{S^N(\xi)}\Vert^2_{\infty}]+ \sup_{i=1,\cdots N}\E[\sup_{\theta \leq s} \vert \tilde \xi^{i,N}_\theta-\xi^{i,N}_\theta\vert ^2]\right ]\,ds.
\end{equation}
Gathering the latter inequality together with~\eqref{eq:tildevMajor} yields 
\begin{eqnarray}
\E[\Vert \tilde v_t-u^{S^N(\xi)}_t\Vert_\infty^2] + \sup_{i=1,\cdots N}\E[\sup_{s\leq t} \vert \tilde \xi^{i,N}_s-\xi^{i,N}_s\vert ^2] & \leq & C\delta t + 2C\sup_{i=1,\cdots N} \E[\sup_{s\leq t} \vert \tilde \xi^{i,N}_s-\xi^{i,N}_s\vert ^2] \nonumber \\
& \leq & C\delta t \nonumber \\ 
& & + \; C \int_0^t \Big [\E[\Vert \tilde v_{s}-u_s^{S^N(\xi)}\Vert^2_{\infty}] \nonumber \\
&& + \sup_{i=1,\cdots N}\E[\sup_{\theta \leq s} \vert \tilde \xi^{i,N}_\theta-\xi^{i,N}_\theta\vert ^2]\Big ] ds \ .
\end{eqnarray}
Applying Gronwall's lemma to the function 
$$
t\mapsto \sup_{i=1,\cdots N}\E[\sup_{s\leq t} \vert \tilde \xi^{i,N}_s-\xi^{i,N}_s\vert ^2]+\E[\Vert \tilde v_{t}-u_t^{S^N(\xi)}\Vert^2_{\infty}]
$$ 
ends the proof  \eqref{eq:tildeYMajor}. 

\item We  focus now on~\eqref{eq:tildeYMajorBis}.  First we observe that 
\begin{equation}
\label{eq:vTildeL2}
\E [\Vert \tilde v_t-u^{S^N(\xi)}_t\Vert_2^2]\leq 2\E [\Vert \tilde v_t-u_t^{S^N(\tilde \xi)}\Vert_2^2]+2\E [\Vert u_t^{S^N(\tilde \xi)}-u_t^{S^N(\xi)}\Vert_2^2]\ .
\end{equation}
Using successively item 4. of Lemma~\ref{lem:uu'}, Remark \ref{RADelta} and 
inequality~\eqref{eq:tildeYMajor},
we can bound the second term on the r.h.s. of \eqref{eq:vTildeL2}  as follows:
\begin{eqnarray}
\label{E815}
\E [\Vert u_t^{S^N(\tilde \xi)}-u_t^{S^N(\xi)}\Vert_2^2]
&\leq &
C\E [\vert W_t\big (S^N(\tilde \xi),S^N(\xi)\big )\vert^2] \nonumber \\
&\leq &
C\sup_{i=1,\cdots N}\E[\sup_{s\leq t}\vert \tilde \xi^{i,N}_s-\xi^{i,N}_s\vert^2] \nonumber \\
&\leq &
C\delta t\ .
\end{eqnarray} 
 
To simplify the notations, we introduce the real valued random variables 
\begin{equation}
\label{eq:VVtildeDef}
V_t^i := e^{\int_0^t\Lambda\big (s,\tilde \xi^{i,N}_{s},u^{S^N(\tilde{\xi})}_{s}(\tilde \xi^{i,N}_{s})\big )ds}\quad \textrm{and}\quad \tilde V_t^i := e^{\int_0^t\Lambda\big (r(s),\tilde \xi^{i,N}_{r(s)},\tilde v_{r(s)}(\tilde \xi^{i,N}_{r(s)})\big )ds}\ ,
\end{equation}
defined for any $i=1,\cdots N$ and $t\in [0,T]$. \\
Concerning the first term on the r.h.s. of~\eqref{eq:vTildeL2}, inequality \eqref{E942} of Lemma \ref{lem:ViVi'} gives for all $y \in \R^d$
\begin{eqnarray}
\label{E816}
\vert \tilde v_{t}(y)-u_t^{S^N(\tilde \xi)}(y)\vert^2 \leq  \frac{M_K}{N}\sum_{i=1}^N K(y-\tilde \xi^{i,N}_t)  \vert \tilde V_t^i-V_t^i \vert^2 \ .
\end{eqnarray}
Integrating the inequality \eqref{E816} with respect to $y$,  yields
$$
\Vert \tilde{v}_t - u_t^{S^N(\tilde \xi)} \Vert_2^2 = \int_{y\in\R^d} \vert \tilde v_{t}(y)-u_t^{S^N(\tilde \xi)}(y)\vert^2\,dy
\leq 
\frac{M_K}{N}\sum_{i=1}^N   \vert \tilde V_t^i
-
V_t^i \vert^2\ ,
$$
which, in turn, implies
\begin{eqnarray}
\label{E817}
\E \left[ \Vert \tilde{v}_t - u^{S^N(\tilde{\xi})}_t \Vert_2^2 \right] \leq \frac{M_K}{N}\sum_{i=1}^N   \E \left[ \vert \tilde V_t^i - V_t^i \vert^2 \right] \ .
\end{eqnarray}
Using successively item $1.$ of Lemma \ref{lem:ViVi'} and inequality \eqref{eq:vtilde} of~Lemma~\ref{lem:Discrete} we obtain, for all $i \in \{1,\cdots,N\}$
\begin{eqnarray}
\label{E818}
 \E[\vert \tilde V_t^i-V_t^i \vert^2]
 & \leq &  C(\delta t)^2 + C\E \left[ \int_0^t \vert \tilde{\xi}^{i,N}_{r(s)} - \tilde{\xi}^{i,N}_{s}  \vert^2   \ ds \right] + C \E \left[ \int_0^t \vert \tilde v_{r(s)}(\tilde \xi^{i,N}_{r(s)})-u^{S^N(\tilde \xi)}_s(\tilde \xi^{i,N}_s)\vert^2 ds \right] \nonumber \\
&\leq & C\delta t + C \E \left[ \int_0^t \vert \tilde v_{r(s)}(\tilde \xi^{i,N}_{r(s)})-u^{S^N(\tilde \xi)}_s(\tilde \xi^{i,N}_s)\vert^2 ds \right] \nonumber\\
& \leq & C\delta t + C \E \left[ \int_0^t \vert \tilde v_{r(s)}(\tilde \xi^{i,N}_{r(s)})-u^{S^N(\tilde \xi)}_s(\tilde \xi^{i,N}_{r(s)})\vert^2 ds \right] \nonumber \\
&& + \; C \E \left[ \int_0^t \vert u_{s}^{S^N(\tilde{\xi})}(\tilde \xi^{i,N}_{r(s)})-u^{S^N(\tilde \xi)}_s(\tilde \xi^{i,N}_{s})\vert^2 ds \right] \nonumber \\
&\leq & C\delta t +C\int_0^t \left [ \E [\Vert \tilde v_{r(s)}-u^{S^N(\tilde \xi)}_s\Vert_\infty^2] + \E[\vert \tilde \xi^{i,N}_{r(s)}-\tilde \xi^{i,N}_s\vert^2] \right ]\, ds \nonumber\\
&\leq &
C\delta t +C\int_0^t \E [\Vert \tilde v_{r(s)}-u^{S^N(\tilde \xi)}_s\Vert_\infty^2] \, ds\ ,
\end{eqnarray}
where the fourth inequality above follows from  Lemma \ref{lem:uu'}, see \eqref{eq:uu'}. Consequently using~\eqref{E818} and inequality \eqref{E83} of Lemma~\ref{lem:Discrete},  \eqref{E817} becomes
\begin{eqnarray}
\label{EFinal_Discret}
\E [\Vert \tilde v_t-u_t^{S^N(\tilde \xi)}\Vert_2^2]
\leq \frac{C}{N}\sum_{i=1}^N \E [\vert \tilde V_t^i - V_t^i \vert^2] \underbrace{\leq}_{\eqref{E818}} C\delta t +C\int_0^t \E [\Vert \tilde v_{r(s)}-u^{S^N(\tilde \xi)}_s\Vert_\infty^2] \underbrace{\leq}_{\eqref{E83}}  C\delta t\ ,
\end{eqnarray}
Finally, injecting \eqref{EFinal_Discret} and \eqref{E815} in \eqref{eq:vTildeL2} yields
$$
\E [\Vert \tilde v_t-u^{S^N(\xi)}_t\Vert_2^2] \leq C \delta t \ ,
$$
\end{itemize}
which ends the proof of Proposition \ref{prop:DiscretTime}.
\end{proof}

\subsection{Numerical results}
\label{SNum}

\subsubsection{Preliminary considerations}

One motivating issue of the section is how the interacting particle system $\xi := \xi^{N,\varepsilon}$ defined in \eqref{eq:XIi} with
 $K = K^{\varepsilon}$, $K^{\varepsilon}(x) := \frac{1}{\varepsilon^d}\phi^d(\frac{x}{\varepsilon})$ for some mollifier $\phi^d$, can be used to approach 
the solution $v$ of the PDE \eqref{epde}. Two significant parameters, i.e.  $\varepsilon \rightarrow 0$, $N \rightarrow + \infty$ intervene.
We expect to approximate $v$ by $u^{\varepsilon,N}$, which is the solution of the linking equation \eqref{eq:u}, associated with the empirical measure $m = S^N(\xi)$. For this purpose, we want to control empirically the so-called Mean Integrated Squared Error (MISE) between the solution $v$ of \eqref{epde} and the particle approximation $u^{\varepsilon,N}$, i.e. for $t \in [0,T]$,
\begin{eqnarray} 
\label{eq:MISE}
\E[ \Vert u_t^{\varepsilon,N} -  v_t \Vert_2^2] \leq 2 \E[ \Vert u_t^{\varepsilon,N} -  u_t^{\varepsilon} \Vert_2^2] + 
2 \E[ \Vert u_t^{\varepsilon} -  v_t \Vert_2^2],
\end{eqnarray}
where $u^{\varepsilon} = u^{m_0}$ with $K = K^{\varepsilon}$, $m_0$ being the common law of processes $Y^i, \; 1 \leq i \leq N$ in \eqref{eq:Yi}. Even though the second expectation in the r.h.s. of \eqref{eq:MISE} does not explicitely involve the number of particles $N$, the first expectation crucially depends on both parameters $\varepsilon, N$. The behavior of the first expectation relies on the propagation of chaos. This phenomenon has been investigated under Assumption \ref{ass:main} for a fixed $\varepsilon > 0$, when $N \rightarrow + \infty$, see Theorem \ref{TPC}.
 According to Theorem \ref{TPC}, the first error term on the r.h.s. of the above inequality can be bounded by $\frac{C(\varepsilon)}{N}$. \\
Concerning the second error term, no result is available but we expect that it converges to zero when $\varepsilon \rightarrow 0$. To control the MISE, it remains to determine a relation $N \mapsto \varepsilon(N)$ such that
$$
\varepsilon(N) \xrightarrow[N \rightarrow + \infty]{} 0 \quad \textrm{ and } \quad \frac{C(\varepsilon(N))}{N} \xrightarrow[N \rightarrow + \infty]{} 0 \ .
$$
When the coefficients $\Phi$, $g$ and the initial condition are smooth with $\Phi$ non-degenerate and $\Lambda \equiv 0$ (i.e. in conservative case), Theorem 2.7 of \cite{JourMeleard} gives a description of such a relation. \\
In our empirical analysis, we have concentrated on a test case, for which we have an explicit solution.
\\
 We first illustrate the chaos propagation for fixed $\varepsilon > 0$, i.e. the result of Theorem \ref{TPC}. On the other hand, we give an empirical insight concerning the following:
\begin{itemize}
\item the asymptotic behavior of the second error term in inequality \eqref{eq:MISE} for $\varepsilon \rightarrow 0$;
\item the tradeoff $N \mapsto \varepsilon(N)$.
\end{itemize}
Moreover, the simulations reveal two behaviors regarding the chaos propagation intensity.

\subsubsection{The target PDE}

We describe now the test case.
For a given triple $(m,\mu , A)\in ]1,\infty[\times \R^d\times \R^{d \times d}$
we consider the following nonlinear PDE of the form~\eqref{epde}:
\begin{equation}
\label{eq:pdev}
\left \{
\begin{array}{lll}
\partial_t v&=&{\displaystyle \frac{1}{2}\sum_{i,j=1}^d \partial_{i,j}^2 \big (v(\Phi \Phi^t)_{i,j}(t,x,v)\big )}-div\big (vg(t,x,v)\big )+v\Lambda(t,x,v) \ ,\\
v(0,x)&=&B_m(2,x)f_{\mu,A}(x)\quad\textrm{for all}\ x\in \R^d \ ,
\end{array}
\right .
\end{equation}
where the functions $\Phi\,,\,g\,,\, \Lambda$ defined on $[0,T]\times \R^d\times \R$ are such that 
\begin{equation}
\label{eq:Phi}
\Phi(t,x,z)=f_{\mu,A}^{\frac{1-m}{2}}(x)z^{\frac{m-1}{2}}I_d\ ,
\end{equation} 
$I_d$ denoting the identity matrix in $\R ^{d \times d}$, 
\begin{equation}
\label{eq:gL}
g(t,x,z)=f_{\mu,A}^{1-m}(x)z^{m-1}\frac{A+A^t}{2}(x-\mu )\ ,
\quad
\textrm{and} 
\quad
\Lambda(t,x,z)=f_{\mu,A}^{1-m}(x)z^{m-1}Tr\left(\frac{A+A^t}{2}\right).
\end{equation}
Here $f_{\mu,A}:\R^d \rightarrow \R $ is given by 
\begin{equation}
\label{eq:f}
f_{\mu ,A}(x)= C e^{-\frac{1}{2}\langle x-\mu,A(x-\mu)\rangle}\ ,
\quad\textrm{normalized by}\quad  
{\displaystyle C=\left [ \int_{x\in \R^d} B_m(2,x)e^{-\frac{1}{2} (x-\mu) \cdot A(x-\mu)}\right ]^{-1}}
\end{equation}
 and    $B_m$ is the $d$-dimensional Barenblatt-Pattle density associated to $m>1$, i.e. 
\begin{equation}
\label{eq:Barenblatt} 
B_m(t,x)= \frac{1}{2}(D -\kappa  t^{-2\beta }\vert x\vert)_+^{\frac{1}{m-1}}t^{-\alpha},
\end{equation}
with
$\alpha =\frac{d}{(m-1)d+2}\ ,$   $\beta =\frac{\alpha }{d} \ ,$  $\kappa =\frac{m-1}{m}\beta$  and 
 $D =[2\kappa ^{-\frac{d}{2}}\frac{\pi^{\frac{d}{2}}\Gamma (\frac{m}{m-1})}{\Gamma (\frac{d}{2}+\frac{m}{m-1})} ]^{\frac{2(1-m)}{2+d(m-1)}}\ .
$

In the specific case where $A$ is the zero matrix of $\R^{d \times d}$, then $f_{\mu,A}\equiv 1$; $g\equiv 0$ and $\Lambda \equiv 0$. Hence, we recover the conservative porous media equation, whose explicit solution is
$$
v(t,x)=B_m(t+2,x)\ ,\quad \textrm{for all} \ (t,x)\in [0,T]\times \R^d,
$$
see~\cite{Barenb}.
For general values of $A\in \R^{d \times d}$, extended calculations produce
 the following explicit solution
\begin{equation}
\label{eq:sol}
v(t,x)=B_m(t+2,x)f_{\mu,A}(x)\ ,\quad \textrm{for all} \ (t,x)\in [0,T]\times \R^d\ ,
\end{equation}
 of~\eqref{eq:pdev}, which is non conservative.

\subsubsection{Details of  the implementation}
Once fixed the number $N$ of particles, we have run $M=100$ i.i.d. particle systems producing $M$ i.i.d. estimates $(u^{\varepsilon,N,i})_{i=1,\cdots M}$. The MISE is then approximated by the  Monte Carlo approximation 
\begin{equation}
\label{eq:MISEApprox}
\E[ \Vert u_t^{\varepsilon,N}-v_t\Vert_2^2 ] \approx \frac{1}{MQ}\sum_{i=1}^M \sum_{j=1}^Q \vert u_t^{\varepsilon,N,i}(X^j)-v_t(X^j)\vert ^2 v^{-1}(0,X^j)\ ,\quad\textrm{for all}\ t\in [0,T]\ ,
\end{equation}
where $(X^j)_{j=1,\cdots, Q=1000}$ are i.i.d $\R^d$-valued random variables with common density $v(0,\cdot)$. 
In our simulation, we have chosen $T=1$, $m=3/2$, $\mu=0$ and $A=\frac{2}{3} I_d$.
 $K^{\varepsilon}=\frac{1}{\varepsilon^d}\phi^d(\frac{\cdot}{\varepsilon})$ with $\phi^d$ being the standard and centered Gaussian density. 
We have run a discretized version of the interacting particle system with Euler scheme mesh  $kT/n$ with $n=10$. Notice that this discretization error is neglected in the present analysis. 
The initial condition $v(0,\cdot)$ is perfectly simulated using a rejection algorithm with a Gaussian instrumental distribution.

\subsubsection{Simulations analysis}
\label{Simul}

Our simulations show that the approximation error presents two types of behavior depending on the number $N$ of particles  
with respect  to the regularization parameter $\varepsilon$.
\begin{enumerate}
\item For large values of $N$, we visualize  a \textit{chaos propagation behavior} for which the  error estimates 
are similar to the ones
provided by  the  density estimation theory~\cite{SilvBook}  corresponding to the classical framework of independent samples.
\item For small values of $N$ appears  a \textit{transient behavior} for which the bias and variance errors
cannot be  easily described.
\end{enumerate} 

Observe that the  Mean Integrated Squared Error ${\rm MISE}_t(\varepsilon,N):=\E[ \Vert u_t^{\varepsilon,N}-v_t\Vert_2^2 ]$ can be decomposed as the sum of the variance $ V_t(\varepsilon,N)$ and squared bias $B_t^2(\varepsilon,N)$ as follows:
\begin{eqnarray}
\label{eq:VarBias}
{\rm MISE}_t(\varepsilon,N)&=&V_t(\varepsilon,N)+B_t^2(\varepsilon,N) \nonumber \\
&=& \E\left[ \Vert u_t^{\varepsilon,N}-\E [u_t^{\varepsilon,N}]\Vert_2^2 \right ]+\E\left [ \Vert\E[ u_t^{\varepsilon,N}]-v_t\Vert_2^2 \right ] \ .
\end{eqnarray}
For $N$ large enough, according to Remark \ref{RMoreGeneral}, one expects that the propagation of chaos holds.
 Then the particle system $(\tilde \xi^{i,N})_{i=1,\cdots, N}$ 
(solution of \eqref{eq:tildeYu}) is close to an i.i.d. system with common law $m^0$.
 We observe that, in the specific case where the weighting function $\Lambda$ does not depend on the density $u$, 
 for $t \in [0,T]$, we have
 \begin{eqnarray}
\label{eq:BiasApprox}
\E[u_t^{\varepsilon , N}] & = & \frac{1}{N} \E \left[ \sum_{j=1}^N K^{\varepsilon}( \cdot-\tilde \xi^{j,N}_{t})\exp\big \{\int_0^t \Lambda(r(s),\tilde \xi^{j,N}_{r(s)})\,ds\big \} \right] \ ,  \nonumber\\
& = & \E \left[ K^{\varepsilon}( \cdot-Y^{1}_t) V_t\big (Y^{1}\big ) \right] \nonumber \\
& = & u_t^{\varepsilon} \ . 
\end{eqnarray} 
 We remind that the relation $u^{\varepsilon} = K^{\varepsilon} \ast v^{\varepsilon}$ comes from Theorem \ref{thm-PIDE}.
 Therefore, under the  chaos propagation behavior,
 the approximations below  hold for
 the variance and the squared bias: 
\begin{equation}
\label{eq:aproxVarBias}
V_t(\varepsilon,N)
\approx 
\E\left[ \Vert u_t^{\varepsilon,N}-u_t^{\varepsilon}\Vert_2^2 \right ]
\quad\textrm{and}\quad 
B_t^2(\varepsilon,N) 
\approx \E \left[ \Vert u_t^{\varepsilon }-v_t\Vert_2^2 \right ]\ .
\end{equation}
On Figure~\ref{fig:Variance5d}, we have reported the estimated variance error $V_t(\varepsilon,N)$  as a function of 
the particle number $N$, (on the left graph) and as a function of the regularization parameter $\varepsilon$, (on the right graph),
 for $t = T=1$ and $d=5$.\\
That figure shows that, when the number of particles is \textit{large enough}, the variance error behaves precisely as in
 the classical case of density estimation encountered in~\cite{SilvBook}, i.e., vanishing at a rate $\frac{1}{N\varepsilon ^d}$, see relation (4.10), Chapter 4., Section 4.3.1.
 This is in particular illustrated by the log-log graphs, showing almost linear curve, when $N$ is sufficiently large. In particular
we observe the following. 
\begin{itemize}
\item On the left graph,  $\log(V_t(\varepsilon,N)) \approx a-\alpha\log N$ with slope $\alpha=1$;
\item On the right graph, $\log V_t(\varepsilon,N) \approx b-\beta \log \varepsilon $  with slope $\beta =5=d$. 
\end{itemize}
It seems that the threshold $N$ after which appears the linear behavior (compatible with the  propagation of chaos situation
corresponding to  asymptotic-i.i.d. particles) decreases
 when $\varepsilon$ grows. In other words, when $\varepsilon$ is large, less particles $N$ are needed
 to give evidence to the chaotic behavior. 
  This phenomenon could be explained by analyzing the particle system dynamics. 
  Indeed, at each time step,
  the interaction between the particles is due to  the empirical estimation of $K^\varepsilon \ast v^\varepsilon $ based on the particle system.
 Intuitively, the more accurate the estimation is, the less strong the interaction between particles will be. Now observe that at time step $0$, the particle system $(\tilde{\xi}^{i,N}_0)$ is i.i.d.  according to $v(0,\cdot)$, so that the estimation of $(K^\varepsilon \ast v^\varepsilon)(0,\cdot)$ provided by \eqref{eq:tildeYu} reduces to the classical density estimation approach. 
  In that classical framework, it is well-known that for larger values of $\varepsilon$ the number of particles, needed to achieve a given density estimation accuracy, is smaller. Hence, one can imagine that for larger values $\varepsilon$ less particles will be needed to obtain a quasi-i.i.d particle system at time step $1$, $(\tilde{\xi}^{i,N}_1)$. Then one can think that this initial error propagates along the time steps.\\
On Figure~\ref{fig:Bias5d}, we have reported the estimated squared bias error, $B^2_t(\varepsilon,N)$,  as a function of the  regularization parameter, $\varepsilon$, for different values of the particle number $N$, for $t = T=1$ and $d=5$.\\
One can observe that, similarly to the classical i.i.d. case,  (see relation (4.9) in Chapter 4., Section 4.3.1 in \cite{SilvBook}),
 for $N$ large enough, the bias error  does not depend on $N$ and 
 can be  approximated by $a \varepsilon^4$, for some constant $a>0$.
 This is in fact coherent with the bias approximation~\eqref{eq:aproxVarBias}, developed in the specific case where
 the weighting function $\Lambda$ does not depend on the density. 
Assuming the validity of approximation~\eqref{eq:aproxVarBias} and  of the previous empirical observation implies that one can bound the error between the solution, $v^\varepsilon$,
 of the regularized PDE of the form \eqref{PIDE} (with $K = K^{\varepsilon}$) associated to \eqref{eq:pdev}, and the solution, $v$, of the limit
(non regularized) PDE~\eqref{eq:pdev} as follows 
\begin{eqnarray}
\label{eq:vepsilonv}
\E \Big[ \Vert v_t^\varepsilon-v_t\Vert_2^2 \Big]& \leq & 2 \E \Big[ \Vert v_t^\varepsilon-u_t^{\varepsilon}\Vert_2^2 \Big] + 2 \E \Big[ \Vert u_t^\varepsilon-v_t\Vert_2^2 \Big] \nonumber \\
  &\leq & 2 \E \Big[ \Vert v_t^\varepsilon-K^\varepsilon\ast v_t^\varepsilon\Vert_2^2 \Big] + 2 \E \Big[ \Vert u_t^\varepsilon-v_t\Vert_2^2 \Big] \nonumber \\
&\leq & 2(a'+a)\varepsilon^4.
\end{eqnarray}
Indeed, at least, the first  term in  the second line can be easily bounded, supposing that  $v_t^\varepsilon$ has
a bounded second derivative.
This constitutes an empirical proof of the fact  that $v^\varepsilon$ converges to $v$.\\
As observed in the variance error graphs, the threshold $N$, above which the propagation of chaos behavior is observed decreases with $\varepsilon$. 
Indeed, for $\varepsilon>0.6$ we observe a chaotic behavior of the bias error, starting from $N\geq 500$, whereas for
 $\varepsilon\in [0.4,0.6]$, this chaotic behavior appears only for $N\geq 5000$. Finally, 
 for small values of $\varepsilon\leq 0.6$, the bias highly depends on $N$ for any $N\leq 10^4$; moreover that dependence
becomes less relevant when  $N$ increases. 


Taking into account both  the bias and the variance error in the MISE \eqref{eq:VarBias}, the choice of $\varepsilon$ has to
 be carefully optimized w.r.t. the number of particles:  $\varepsilon$ going to zero together with $N$ going to infinity at a 
judicious relative rate seem to ensure the convergence of the estimated MISE to zero. 
This kind of tradeoff is standard in density estimation theory and was already investigated theoretically in the context of
 forward interacting particle systems related to conservative regularized nonlinear PDE in
 \cite{JourMeleard}. Extending this type of theoretical analysis to our non conservative framework is beyond the scope
 of the present paper.

\begin{figure}[!h]
\begin{center}
\subfigure[Variance as a function of $N$]
{\includegraphics[width=0.49\linewidth]{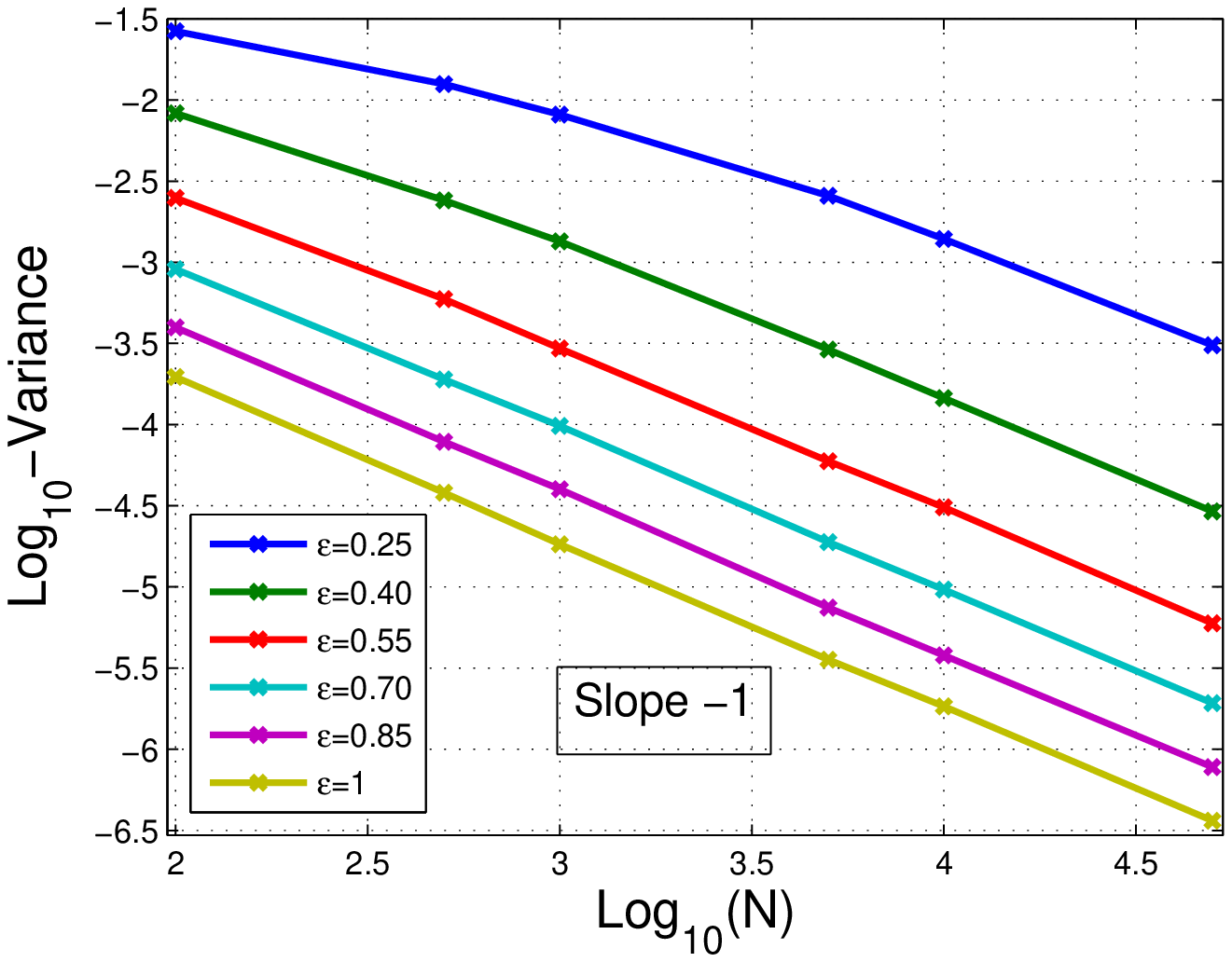}}
\subfigure[Variance as a function of $\varepsilon$]{\includegraphics[width=0.49\linewidth]{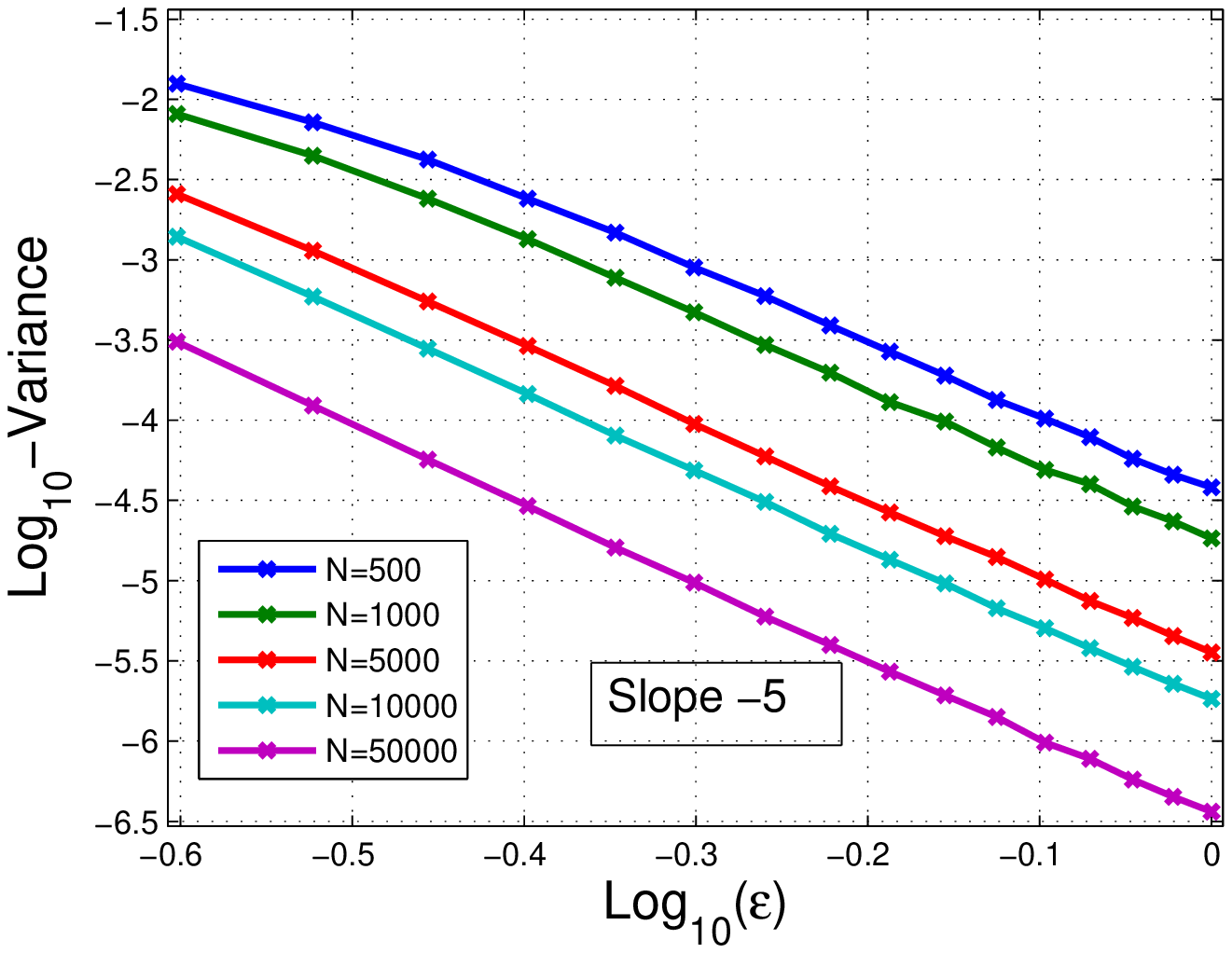}}
\end{center}
\caption{{\small  Variance error as a function of the number of particles, $N$, and the mollifier window width, $\varepsilon$, for dimension $d=5$ at the final time step $T=1$. } }
\label{fig:Variance5d}
\end{figure}

\begin{figure}[!h]
\begin{center}
{\includegraphics[width=0.8\linewidth,height=7cm]{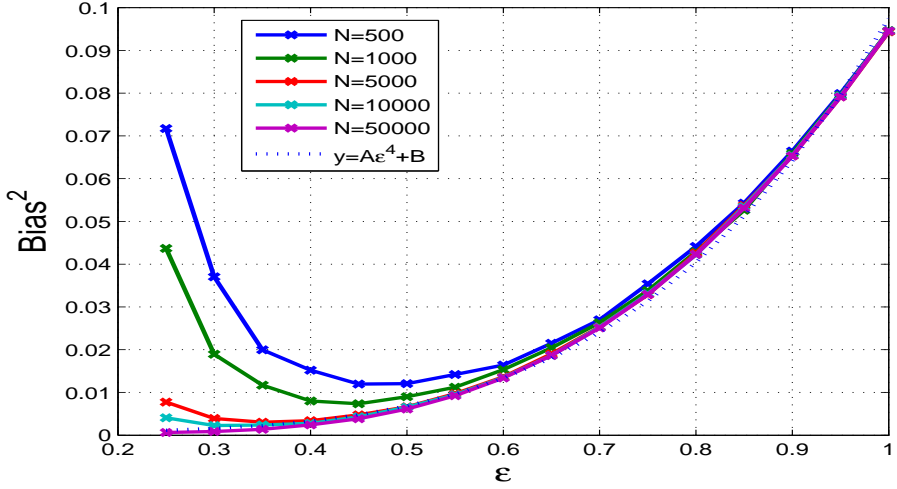}}
\end{center}
\caption{{\small  Bias error as a function of the mollifier window width, $\varepsilon$, for dimension $d=5$ at the final time step $T=1$. } }
\label{fig:Bias5d}
\end{figure}

\newpage

\section{Appendix}

\setcounter{equation}{0}
\label{SAppendix}

In this appendix, we present the proof of some technical results.

\begin{rem}  \label{counter-example}
We start with an observation which concerns a possible relaxation of the hypotheses of Lemma \ref{karatzas};
the  uniform convergence assumption for the integrands is crucial and it cannot be replaced
by a pointwise convergence. \\

 Let define $\Omega = [0,1]$ equipped with the Borel $\sigma$-field, $(Z_n)_{n \geq 0}$ a sequence of continuous, real-valued functions s.th. 
\begin{equation}
 \left \{
 \begin{array}{lll}
0 & , & x \geq \frac{2}{n} \\  
nx & , & x \in [0,\frac{1}{n}] \\
-nx + 2 & , & x \in [\frac{1}{n},\frac{2}{n}].
 \end{array}
 \right .
\end{equation}
We consider a sequence of probability measures $(m_n)_{n \geq 0}$ s.th. $m_n(dx) = \delta_{\frac{1}{n}}(dx)$ and $m_0(dx) = \delta_{0}(dx)$.

On the one hand, we can observe the following. 
\begin{itemize}
 \item $Z_n \xrightarrow[\text{$n \rightarrow +\infty$}]{\text{}} 0$, pointwise.
 \item for all $n \geq 0$, $|Z_n| \leq 1$, surely.
 \item $m_n \xrightarrow[\text{$n \rightarrow +\infty$}]{\text{}} m$,  weakly.
\end{itemize}

On the other hand,$ \int_{0}^{1} Z_n dm_n = Z_n(\frac{1}{n}) = 1 \nrightarrow 0$.
\end{rem}

Before stating a tightness criterion for our family of approximating 
sequences we need to express  the classical Theorem of Kolmogorov-Centsov,
 stated in Theorem 4.10, Chapter 2 in \cite{karatshreve},
taking into account Remark 4.13.

\begin{prop}
\label{PKC}
Let $r \in \N^{\star}$. A sequence $(\P_n)_{n \geq 0}$ of Borel probability measures on $\shc^r$ is tight if and only if
\begin{itemize}
\item
\begin{equation}
\label{E46bis}
\lim_{\lambda \longrightarrow + \infty} \; \sup_{n \in \N} \; \P_n \left( \left\{ \omega \in \shc^r \; \vert \; \vert \omega_0 \vert > \lambda  \right\} \right) = 0 \ ,
\end{equation}
\item $\forall(\varepsilon,s,t) \in \R_+^{\star} \times [0,T] \times [0,T]$,
\begin{equation}
\label{E47bis}
 \lim_{\delta \downarrow 0} \; \sup_{n \in \N} \; \P_n ( \{ \omega \in \shc^r \; \vert \; \max_{ \underset{\vert t-s \vert \le \delta}{(s,t) \in [0,T]^2}} \vert \omega_t - \omega_s \vert > \varepsilon  \} ) = 0 \ .
\end{equation}
\end{itemize}
\end{prop}
\medskip
\begin{lem} \label{tightness} 
Let $K: \R^d \rightarrow \R$
be bounded and Lipschitz.
 For each $n \in \N$, we consider Borel functions   
$\Phi_{n} : [0,T] \times \R^d \times \R \longrightarrow \R^{d \times p}$,
 $g_{n} : [0,T] \times \R^d \times \R \longrightarrow \R^d$, and 
$\Lambda_{n}: [0,T] \times \R^d \times \R \longrightarrow \R$ 
 uniformly bounded in $n$.
We also consider a tight sequence $(\zeta_0^n)$ of
probability measures on $\R^d$.
Let $(Y^{n},u_n)$ be solutions of
   \begin{equation}
   \left \{
   \begin{array}{lll}
    dY^{n}_{t} = \Phi_{n}(t,Y^{n}_t,u_{n}(t,Y_{t}^{n}))dW_t + g_{n}(t,Y^{n}_t,u_{n}(t,Y_{t}^{n}))dt \\
 u_{n}(t,x) := \int_{\mathcal{C}^d} K(x-X_t(\omega)) \, 
\exp \left \{\int_{0}^{t}\Lambda_{n} \big 
(r,X_r(\omega),u_{n}(r,X_r(\omega))\big )dr\right \} dm^{n}(\omega) \, \\
m_n  = \shl(Y_n),
   \end{array}
   \right .
  \end{equation}
where for all $n \in \N$, $Y^n_0$ 
is a r.v.  distributed according to $\zeta_0^n$.\\
Then, the family $\big(\nu^{n} = \mathcal{L}(Y^{n}_{\cdot},u_{n}(\cdot,Y^{n}_{\cdot})), \; n \geq 0 \big)$ is tight.
\end{lem}
\begin{proof}
  If we denote by $\P_n$ the law of $(Y_n,u^n(\cdot,Y^n))$ we bound the l.h.s of \eqref{E46bis} as follows:
\begin{eqnarray}
\label{E910}
\P_n(\{ \omega \in \shc^{d+1} \vert \; \vert \omega_0 \vert > \lambda \}) & = & \P(\{ \vert (Y_0^n,u^n(0,Y_0^n)) \vert > \lambda \}) \nonumber \\
& \leq & \P(\{ \vert Y_0^n \vert + \vert u^n(0,Y_0^n) \vert > \lambda \}) \nonumber \\
& \leq & \P(\{ \vert Y_0^n \vert > \frac{\lambda}{2} \}) + 
 \P( \{ \vert u^n(0,Y_0^n) \vert > \frac{\lambda}{2} \}) \nonumber \\
 & \leq & \zeta^n_0(\{ x \in \R^d \vert \; \vert x \vert > \frac{\lambda}{2} \}) + \P( \{ \vert u^n(0,Y_0^n) \vert > \frac{\lambda}{2} \}).
\end{eqnarray}
Let us fix $\varepsilon > 0$.
On the one hand, $(\zeta^n_0)$ being tight there exists a compact set $\mathfrak{K}_{\varepsilon}$ of $\R^d$ such that $\displaystyle{ \sup_{n \in \N} \zeta^n_0(\mathfrak{K}_{\varepsilon}^c) \leq \varepsilon }$. Then, there exists $\lambda_{\varepsilon} > 0$ 
such that $  \{ x \in \R^d \vert \; \vert x \vert > \frac{\lambda_{\varepsilon}}{2} \} \subset \mathfrak{K}_{\varepsilon}^c$ which implies 
$$\displaystyle{ \sup_{n \in \N} \zeta^n_0(\{ x \in \R^d \vert \; \vert x \vert > \frac{\lambda_\varepsilon}{2} \}) \leq \sup_{n \in \N} \zeta^n_0(\mathfrak{K}_{\varepsilon}^c) \leq \varepsilon }.$$
 On the other hand, since $u^n$ is uniformly bounded,
for all $\lambda > 0$,
 Chebyshev's inequality implies 
\begin{eqnarray} 
\label{ELambda}
 \P( \{ \vert u^n(0,Y_0^n) \vert > \frac{\lambda}{2} \}) \leq 4 \frac{\E[ \vert u^n(0,Y^n_0) \vert^2]}{\lambda^2} \leq 4 \frac{(M_K e^{T M_{\Lambda}})^2}{\lambda^2} \ . 
\end{eqnarray}
Consequently for $\lambda \ge \lambda_\varepsilon$, 
we get
\begin{eqnarray}
\label{E912}
 \sup_{n \in \N} \P_n(\{ \omega \in \shc^{d+1} \vert \; \vert \omega_0 \vert > \lambda \}) \leq 
 4 \frac{(M_K e^{T M_{\Lambda}})^2}{\lambda^2} + \varepsilon \ .
 \end{eqnarray}
Taking the limit when $\lambda$ goes to infinity,
we finally get inequality \eqref{E46bis} since $\varepsilon > 0$ is arbitrary.

It remains to prove \eqref{E47bis}. \\
We will make use of Garsia-Rodemich-Rumsey Theorem, see e.g. Theorem 2.1.3, Chapter 2 in \cite{stroock} or \cite{BarYo}.
We will show that, for all $0 \le s < t \le T$, there exists a positive real constant $C \ge 0$ 
\begin{equation}
\label{EBY}
\E{[|Y^{n}_{t} - Y^{n}_{s}|^{4} + |u_{n}(t,Y^{n}_{t})-u_{n}(s,Y^{n}_{s})|^{4}]} \le C \vert t-s \vert^2 \ ,
\end{equation}
where $C$ does not depend on $n$.
Suppose for a moment that \eqref{EBY} holds true. \\
Let $\varepsilon > 0$ fixed. Let $\delta > 0$. If $\P_n$ denotes again the law of $(Y^n,u^n(\cdot,Y^n))$,  the quantity
\begin{eqnarray} \label{E945}
\P_n ( \{ \omega \in \shc^{d +1}\; \vert \; \sup_{ \underset{\vert t-s \vert \le \delta}{(s,t) \in [0,T]^2}} \vert \omega_t - \omega_s \vert > \varepsilon  \} )
\end{eqnarray}
intervening in \eqref{E47bis} is bounded, up to a constant, by
\begin{eqnarray}
\label{BB1}
\P ( \max_{ \underset{\vert t-s \vert \le \delta}{(s,t) \in [0,T]^2}} \left\{ \vert Y^n_t - Y^n_s \vert + \vert u^n(t,Y^n_t) - u^n(s,Y^n_s) \vert \right\} > \varepsilon  ) \ .
\end{eqnarray}  
Let us fix $\gamma \in  ]0,\frac{1}{4}[$.
By Garsia-Rodemich-Rumsey theorem, there is a sequence of non-negative r.v. $\Gamma^n$ such that, a.s.
$$
\sup_{n \in \N} \E{[ (\Gamma^n)^4 ]} < \infty 
$$
\begin{eqnarray}
\label{GRR}
\forall (s,t) \in [0,T]^2, \; \vert Y^n_t - Y^n_s\vert + \vert u^n(t,Y^n_t) - u^n(s,Y^n_s) \vert \leq \Gamma^n \vert t-s|^{\gamma}.
\end{eqnarray}
 If $\vert t-s \vert \le \delta$ \eqref{GRR} gives 
\begin{eqnarray}
\label{BB3}
\displaystyle{ \max_{ \underset{\vert t-s \vert \le \delta}{(s,t) \in [0,T]^2}} \left\{ \vert Y^n_t - Y^n_s \vert + \vert u^n(t,Y^n_t) - u^n(s,Y^n_s) \vert \right\} \le \Gamma^n \delta^{\gamma} } \ .
\end{eqnarray}
By    \eqref{BB3} and   Chebyshev's inequality, for any $n \in \N$,
 the quantity \eqref{E945} is bounded by
\begin{eqnarray}
\P( \Gamma^n \delta^{\gamma} > \varepsilon) & = & \P( \Gamma^n > \varepsilon \delta^{-\gamma}) \nonumber \\
& \le & \frac{\delta^{4 \gamma}}{\varepsilon^4}, \nonumber
\end{eqnarray}
for any $n \in \N$. Since $\delta >0$ is arbitrary, 
 \eqref{E47bis} follows. To conclude the proof of the lemma, it remains to show \eqref{EBY}. \\
We recall that $M_{\Phi}$, $M_{g}$, $M_{\Lambda}$, $M_{K}$ denote the uniform upper bound of the
 sequences $(\vert \Phi_{n} \vert)$, $(\vert g_{n} \vert)$, $(\vert \Lambda_{n} \vert)$ and of the function $ K$.
Let $0 \leq s < t \leq T$.
To show \eqref{EBY},
 we have to evaluate 
\begin{eqnarray}
\label{tigth0}
 \E{[|Y^{n}_{t} - Y^{n}_{s}|^{4}]} + \E{[|u_{n}(t,Y^{n}_{t})-u_{n}(s,Y^{n}_{s})|^{4}]} \ .
\end{eqnarray}

By classical computations (e.g. It\^o's isometry, Cauchy-Schwarz inequality), 
we easily obtain 
\begin{equation}
 \label{eq:tight1}
 \forall k \in \N^{\star}, \; \forall T>0, \; \exists C' := C'_{(k,T,M_{\Phi},M_{g},M_{\Lambda})} > 0, \;  \E{[|Y^{n}_{t} - Y^{n}_{s}|^{2k}]} \leq C' |t-s|^{k},
\end{equation}
where the constant $C'$ does not depend on $n$ because $\Phi_n, g_n$ are uniformly bounded, in particular w.r.t. $n$.

 Regarding the second expectation in \eqref{tigth0}, we get 
 \begin{eqnarray}
  \label{eq:tight2}
  \E{[|u_{n}(t,Y^{n}_{t})-u_{n}(s,Y^{n}_{s})|^{4}]} & = & \int_{\mathcal{C}^d} \big( u_{n}(t,X_{t}(\omega))-u_{n}(s,X_{s}(\omega)
) \big)^{4} dm^{n}(\omega) \nonumber \\
&& \\
  & \leq & 8(I_{1} + I_{2}) \ , \nonumber
 \end{eqnarray}
 where 
 \begin{equation}
 \begin{array}{lll}
 \displaystyle{
 I_{1} := \int_{\mathcal{C}^d} \big( u_{n}(t,X_{t}(\omega))-u_{n}(s,X_{t}(\omega)) \big)^{4} dm^{n}(\omega) } \\
 \displaystyle{I_{2} := \int_{\mathcal{C}^d} \big( u_{n}(s,X_{t}(\omega))-u_{n}(s,X_{s}(\omega)) \big)^{4} dm^{n}(\omega) \ . }
 \end{array}
 \end{equation}
 On the one hand, for all $x \in \R^d$, 
 \begin{eqnarray*}
   |u_{n}(t,x) - u_{n}(s,x)| & = & \Big| \E{ \Big[K(x-Y^{n}_{t}) e^{\int_{0}^{t} \Lambda_{n}(r,Y^{n}_r,u_{n}(r,Y^{n}_r))dr} \Big]} -
 \E{ \Big[K(x-Y^{n}_{s}) e^{ \int_{0}^{s} \Lambda_{n}(r,Y^{n}_r,u_{n}(r,Y^{n}_r))dr } \Big]} \Big| \nonumber \\    
   & \leq &  \int_{\mathcal{C}^d} |K(x-X_{t}(\omega)) - K(x-X_{s}(\omega))| \exp\left( \int_{0}^{t} \Lambda_{n}(r,X_r,u_{n}(r,X_r))dr \right) dm^{n}(\omega) \nonumber  \\
   & & \left . + \; \int_{\mathcal{C}^d} K(x-X_{s}(\omega))\; \Big| \exp\left(\int_{0}^{t} \Lambda_{n}(r,X_r(\omega),u_{n}(r,X_r(\omega))\right)dr
   \right . \nonumber  \\
   & & \left . - \; \exp\left( \int_{0}^{s} \Lambda_{n}(r,X_r(\omega),u_{n}(r,X_r(\omega)))dr  \right) \Big| \; dm^{n}(\omega) \right . \nonumber  
\end{eqnarray*}
By \eqref{EMajor1} and \eqref{eq:tight1} (with $k=1$) together with Cauchy-Schwarz inequality, this is lower than 
\begin{eqnarray*}
   &  & L_K \exp(M_{\Lambda}T) 
\int_{\mathcal{C}^d} |X_{t}(\omega) - X_{s}(\omega)|dm^{n}(\omega) \\
   &+ &  M_{K}\exp(M_{\Lambda}) \int_{\mathcal{C}^d} \Big| \int_{s}^{t} \Lambda_{n}(r,X_r(\omega),u_{n}(r,X_r(\omega)))dr \Big|
 dm^n(\omega) \\ 
  & \leq & (L_K \exp(M_{\Lambda}T)\sqrt{C'} + M_{K}\exp(M_{\Lambda})M_{\Lambda} \sqrt{T}) \sqrt{|t-s|},
 \end{eqnarray*}
 which implies
  \begin{equation}
  \label{eq:majorI2}
  \begin{array}{lll}
   I_1 = \int_{\mathcal{C}^d}|u_{n}(t,X_t(\omega)) - u_{n}(s,X_t(\omega))|^{4}dm^{n}(\omega) & \leq & (L_K \exp(M_{\Lambda}T) \sqrt{C'} 
 + M_K \exp(M_{\Lambda})M_{\Lambda} \sqrt{T})^{4}|t-s|^{2} \ .
  \end{array}
 \end{equation}
 On the other hand, for all $(x,y) \in \R^d \times \R^d$
  \begin{equation}
  \label{eq:majorI3}
  \begin{array}{lll}
   |u_{n}(s,x) - u_{n}(s,y)| & \leq &  \E{[\big| K(x-Y^n_{s})-K(y-Y^n_{s})\big| \exp\big( \int_{0}^{s} \Lambda_{n}(r,Y^n_r,u_{n}(r,Y^n_r))dr \big)]}  \\  
   & \leq & L_K \exp(M_{\Lambda}T)|x-y|, \ ,
  \end{array}
 \end{equation}
 which implies
   \begin{equation}
  \label{eq:majorI4}
  \begin{array}{lll}
   I_2 = \int_{\mathcal{C}^d} |u_{n}(s,X_t(\omega)) - u_{n}(s,X_s(\omega))|^{4}dm^{n}(\omega) & \leq & L_K \exp(M_{\Lambda}T) \int_{\mathcal{C}^d} |X_t(\omega)-X_s(\omega)|^{4}dm^{n}(\omega)  \\  
   & \leq & L_K \exp(M_{\Lambda}T)C'|t-s|^{2} \ ,
  \end{array}
 \end{equation}
 where the second inequality comes from \eqref{eq:tight1} with $k = 2$. \\
 Coming back to \eqref{eq:tight2}, we have $|I_{1} + I_{2}| \leq C'' |t-s|^{2}$ with $C''$ a constant value depending only on 
$T,M_{\Phi}, M_g,M_{\Lambda}, M_K,L_K,T $.
 This enable us to conclude the proof of \eqref{EBY} and finally the one of
Lemma \ref{tightness}. \\
\end{proof}
We proceed now with the proof  of  Lemma \ref{lem:Discrete}, that will make use of the following intermediary result.
\begin{lem}
\label{lem:ViVi'}
Let $N \in \N^{\star}$. Let $(\xi^{i,N})_{i=1,\cdots,N}$ be a solution of the interacting particle system \eqref{eq:XIi};
 let $(\tilde{\xi}^{i,N})_{i=1,\cdots,N}$ and $\tilde v$  as defined as in  the discretized interacting particle system \eqref{eq:tildeYu}. \\
Under the same assumptions as in Proposition \ref{prop:DiscretTime}, the random variables $V_t^i := e^{\int_0^t\Lambda\big (s,\tilde \xi^{i,N}_{s},u^{S^N(\tilde{\xi})}_{s}(\tilde \xi^{i,N}_{s})\big )ds}$ and $ \tilde V_t^i := e^{\int_0^t\Lambda\big (r(s),\tilde \xi^{i,N}_{r(s)},\tilde v_{r(s)}(\tilde \xi^{i,N}_{r(s)})\big )ds}$,
 for all  $t \in [0,T]$, $i \in \{ 1,\cdots,N\}$ fulfill the following.
\begin{enumerate}
\item For all $t \in [0,T]$, $i \in \{1,\cdots,N\}$
\begin{eqnarray}
\label{E941}
 \E[\vert \tilde V_t^i-V_t^i \vert^2]
\leq C(\delta t)^2 + C\E \left[ \int_0^t \vert \tilde{\xi}^{i,N}_{r(s)} - \tilde{\xi}^{i,N}_{s}  \vert^2   \ ds \right] + C \E \left[ \int_0^t \vert \tilde v_{r(s)}(\tilde \xi^{i,N}_{r(s)})-u^{S^N(\tilde \xi)}_s(\tilde \xi^{i,N}_s)\vert^2 ds \right],
\end{eqnarray}
where $C$ is a real positive constant depending only on $M_{\Lambda}$, $L_{\Lambda}$ and $T$.
\item For all $(t,y) \in [0,T] \times \R^d$, $i \in \{1,\cdots,N\}$
\begin{eqnarray}
\label{E942}
\vert \tilde v_{t}(y)-u_t^{S^N(\tilde \xi)}(y)\vert^2 \leq  \frac{M_K}{N}\sum_{i=1}^N K(y-\tilde \xi^{i,N}_t) \;  \vert \tilde V_t^i-V_t^i \vert^2 \ .
\end{eqnarray}
\end{enumerate}
\end{lem} 
\begin{proof}[Proof of Lemma \ref{lem:ViVi'}]
Let us fix $t \in [0,T]$, $i \in \{1,\cdots,N\}$. To prove \eqref{E941}, it is enough to recall that $\Lambda$ being uniformly Lipschitz w.r.t. the time and space variables, the inequality~\eqref{EMajor1} yields
\begin{equation}
\label{eq:VtildeV2}
 \vert \tilde V_t^i
-
V_t^i \vert^2
\leq 
3e^{2tM_{\Lambda}}L^2_{\Lambda} \int_0^t \left [\vert r(s)-s\vert^2 +\vert \tilde \xi^{i,N}_{r(s)}-\tilde \xi^{i,N}_s\vert^2+\vert \tilde v_{r(s)}(\tilde \xi^{i,N}_{r(s)})-u_s^{S^N(\tilde \xi)}(\tilde \xi^{i,N}_s)\vert^2 \right ]\,ds\ ,
\end{equation}
and taking the expectation in both sides of \eqref{eq:VtildeV2} implies \eqref{E941} with $C := 3e^{2TM_{\Lambda}}L^2_{\Lambda}$. \\
Let us fix $y \in \R^d$. Concerning \eqref{E942}, by recalling the third line equation of \eqref{eq:tildeYu} and the linking equation \eqref{eq:u} (with $m = S^N(\tilde{\xi})$), we have 
\begin{eqnarray}
\vert \tilde v_{t}(y)-u_t^{S^N(\tilde \xi)}(y)\vert^2 & = & \left \vert \frac{1}{N} \sum_{i=1}^N K(y- \tilde{\xi}^{i,N}) \tilde{V}^i_t - \frac{1}{N} \sum_{i=1}^N K(y- \tilde{\xi}^{i,N}) V^i_t   \right \vert^2 \nonumber \\
& = & \left \vert \frac{1}{N} \sum_{i=1}^N K(y- \tilde{\xi}^{i,N}) \left( \tilde{V}^i_t - V^i_t \right)  \right \vert^2 \nonumber \\
&\leq &
\label{E944}
\frac{1}{N}\sum_{i=1}^N K^2(y-\tilde \xi^{i,N}_t)  \vert \tilde V_t^i
-
V_t^i \vert^2 \nonumber\\
& \leq & \frac{M_K}{N}\sum_{i=1}^N  K(y-\tilde \xi^{i,N}_t) \;  \vert \tilde V_t^i-V_t^i \vert^2 \ ,
\end{eqnarray}
which concludes the proof of \eqref{E942} and therefore of  Lemma \ref{lem:ViVi'}.
\end{proof}

\begin{proof}[\bf Proof of Lemma \ref{lem:Discrete}.]
All along this proof, $C$ will denote a positive constant that only  depends \\
 $T,M_K,L_K,M_\Phi,L_\Phi,M_g,L_g$ and $M_{\Lambda},L_{\Lambda}$ and that can change from line to line. Let us fix $t \in [0,T]$.

\begin{itemize}
	\item Inequality \eqref{eq:vtilde} of Lemma~\ref{lem:Discrete} is simply a consequence of the fact that the coefficients $\Phi$ and $g$ are uniformly bounded. Indeed,
\begin{eqnarray*}
\E[\vert \tilde \xi ^{i,N}_{r(t)}-\tilde \xi^{i,N}_{t}\vert^2]&=&\E\left [\Big \vert \int_{r(t)}^{t}\Phi(\tilde{v}_{r(s)}(\tilde \xi ^{i,N}_{r(s)}))\,dW_s+\int_{r(t)}^{t}g(\tilde{v}_{r(s)}(\tilde \xi ^{i,N}_{r(s)}))\,ds\Big \vert^2 \right ]\\
&\leq &
2\E\left [ \int_{r(t)}^{t}\vert \Phi(\tilde{v}_{r(s)}(\tilde \xi ^{i,N}_{r(s)}))\vert ^2\,ds\right ]+2(t-r(t))\E \left [\int_{r(t)}^{t}\vert g(\tilde{v}_{r(s)}(\tilde \xi ^{i,N}_{r(s)}))\vert ^2\,ds \right ]\\
&\leq &
2M^2_{\Phi}\delta t +2M^2_g(\delta t)^2 \\
&\leq &
C\delta t \ ,\quad \textrm{as soon as }\quad \delta t \in\, ]0, 1[ \ .
\end{eqnarray*}

\item Now, let us focus on the second inequality \eqref{E82} of Lemma~\ref{lem:Discrete}. Note that for any $y\in \R^d$, the following inequality holds:
\begin{eqnarray}
\label{eq:r(t)t}
\vert \tilde v_{r(t)}(y)-\tilde v_t(y)\vert & \leq & \frac{1}{N} \sum_{i=1}^N \left[ \left \vert K(y-\tilde \xi^{i,N}_{r(t)}) -K(y-\tilde \xi^{i,N}_t)\right \vert \, e^{\int_0^{r(t)}\Lambda\big (r(s),\tilde \xi^{i,N}_{r(s)},\tilde v_{r(s)}(\tilde \xi^{i,N}_{r(s)})\big )ds} \right . \nonumber \\
& &  +
 \left .  K(y-\tilde \xi^{i,N}_t) \,\left \vert e^{\int_0^{r(t)}\Lambda\big (r(s),\tilde \xi^{i,N}_{r(s)},\tilde v_{r(s)}(\tilde \xi^{i,N}_{r(s)})\big )ds}- e^{\int_0^t\Lambda\big (r(s),\tilde \xi^{i,N}_{r(s)},\tilde v_{r(s)}(\tilde \xi^{i,N}_{r(s)})\big )ds} \right \vert \right ] \ . \nonumber \\ 
\end{eqnarray}
Using the Lipschitz property of $\Lambda$ and the fact that  $K$ and $\Lambda$ are bounded, 
one can apply~\eqref{EMajor1} to bound the second term of the sum on the r.h.s. of the above inequality as follows:
\begin{eqnarray}
\label{eq:K_exp}
K(y-\tilde \xi^{i,N}_t) \,\left \vert e^{\int_0^{r(t)}\Lambda\big (r(s),\tilde \xi^{i,N}_{r(s)},\tilde v_{r(s)}(\tilde \xi^{i,N}_{r(s)})\big )ds}- e^{\int_0^t\Lambda\big (r(s),\tilde \xi^{i,N}_{r(s)},\tilde v_{r(s)}(\tilde \xi^{i,N}_{r(s)})\big )ds}\right \vert 
& \leq & M_Ke^{(t-r(t))M_{\Lambda}}(t-r(t))M_{\Lambda} \nonumber \\
& \leq & C\delta t \ .
\end{eqnarray}
The first term of the sum on the r.h.s. of~\eqref{eq:r(t)t} is bounded using the Lipschitz property of $K$ and the fact that $\Lambda$ is bounded. 
\begin{eqnarray}
\label{eq:KKtilde}
\left \vert K(y-\tilde \xi^{i,N}_{r(t)}) -K(y-\tilde \xi^{i,N}_t)\right \vert\, e^{\int_0^{r(t)}\Lambda\big (r(s),\tilde \xi^{i,N}_{r(s)},\tilde v_{r(s)}(\tilde \xi^{i,N}_{r(s)})\big )ds} \leq  
L_Ke^{tM_{\Lambda}} \vert \tilde \xi^{i,N}_{r(t)}-\tilde \xi^{i,N}_t\vert \ .
\end{eqnarray}
Injecting \eqref{eq:K_exp}  and \eqref{eq:KKtilde} in \eqref{eq:r(t)t} we obtain for all $y \in \R^d$
$$
\vert \tilde{v}_{r(t)}(y) - \tilde{v}_{t}(y) \vert \leq C \delta t + \frac{L_Ke^{tM_{\Lambda}}}{N} \sum_{i=1}^N \vert \tilde{\xi}^{i,N}_{r(t)}-\tilde{\xi}^{i,N}_{t} \vert,
$$
which finally implies that 
$$
\Vert \tilde v_{r(t)}-\tilde v_t \Vert_{\infty}^2\leq C\delta t ^2+\frac{C}{N}\sum_{i=1}^N \vert \tilde \xi^{i,N}_{r(t)}-\tilde \xi^{i,N}_t\vert^2 \ .
$$
We conclude by using  inequality \eqref{eq:vtilde} of Lemma~\ref{lem:Discrete} after taking the expectation of the r.h.s. of the above inequality. 

\item Finally, we  deal with  inequality \eqref{E83} of Lemma~\ref{lem:Discrete}. Observe that the 
error on the left-hand side can be decomposed as  
\begin{eqnarray}
\label{eq:tildev1}
\E[\Vert \tilde v_{r(t)}-u_t^{S^N(\tilde \xi)}\Vert ^2_{\infty}]
&\leq &2\E[\Vert \tilde v_{r(t)}-\tilde v_t\Vert^2_{\infty}]+2\E[\Vert \tilde v_{t}-u_t^{S^N(\tilde \xi)}\Vert^2_{\infty}] \nonumber \\
&\leq & C\delta t +2\E[\Vert \tilde v_{t}-u_t^{S^N(\tilde \xi)}\Vert^2_{\infty}]\ ,
\end{eqnarray}
where we have used  inequality \eqref{E82} of Lemma \ref{lem:Discrete}. \\
Let us consider the second term on the r.h.s. of the above inequality. 
To simplify the notations, we introduce the real valued random variables 
\begin{equation}
\label{eq:VVtildeDef2}
V_t^i := e^{\int_0^t\Lambda\big (s,\tilde \xi^{i,N}_{s},u^{S^N(\tilde{\xi})}_{s}(\tilde \xi^{i,N}_{s})\big )ds}\quad \textrm{and}\quad \tilde V_t^i := e^{\int_0^t\Lambda\big (r(s),\tilde \xi^{i,N}_{r(s)},\tilde v_{r(s)}(\tilde \xi^{i,N}_{r(s)})\big )ds}\ ,
\end{equation}
defined for any $i=1,\cdots N$ and $t\in [0,T]$. \\ 
Using successively inequalities \eqref{E941} of Lemma \ref{lem:ViVi'}, \eqref{eq:vtilde} of~Lemma~\ref{lem:Discrete} and \eqref{eq:uu'} of Lemma \ref{lem:uu'}, we have for all $i \in \{1,\cdots,N\}$,   
\begin{eqnarray}
\label{eq:Vnew}
 \E[\vert \tilde V_t^i-V_t^i \vert^2]
&\leq & C\delta t + C \E \left[ \int_0^t \vert \tilde v_{r(s)}(\tilde \xi^{i,N}_{r(s)})-u^{S^N(\tilde \xi)}_s(\tilde \xi^{i,N}_s)\vert^2 ds \right] \nonumber\\
& \leq & C\delta t + C \E \left[ \int_0^t \vert \tilde v_{r(s)}(\tilde \xi^{i,N}_{r(s)})-u^{S^N(\tilde \xi)}_s(\tilde \xi^{i,N}_{r(s)})\vert^2 ds \right] \nonumber \\
&& + \; C \E \left[ \int_0^t \vert u_{s}^{S^N(\tilde{\xi})}(\tilde \xi^{i,N}_{r(s)})-u^{S^N(\tilde \xi)}_s(\tilde \xi^{i,N}_{s})\vert^2 ds \right] \nonumber \\
&\leq & C\delta t +C\int_0^t \left [ \E [\Vert \tilde v_{r(s)}-u^{S^N(\tilde \xi)}_s\Vert_\infty^2] + \E[\vert \tilde \xi^{i,N}_{r(s)}-\tilde \xi^{i,N}_s\vert^2] \right ]\, ds \nonumber\\
&\leq &
C\delta t +C\int_0^t \E [\Vert \tilde v_{r(s)}-u^{S^N(\tilde \xi)}_s\Vert_\infty^2] \, ds \ .
\end{eqnarray}
On the other hand, inequality \eqref{E942} of Lemma \ref{lem:ViVi'} implies
\begin{eqnarray}
\label{major_E86}
\Vert  \tilde v_{t}-u_t^{S^N(\tilde \xi)} \Vert_{\infty}^2 \leq \frac{M_K^2}{N}\sum_{i=1}^N   \vert \tilde V_t^i-V_t^i \vert^2 \ .
\end{eqnarray}
Taking the expectation in both sides of \eqref{major_E86} and using \eqref{eq:Vnew} give
\begin{equation}
\label{eq:tildevxi}
\E[\Vert \tilde v_{t}-u_t^{S^N(\tilde \xi)}\Vert^2_{\infty}]
\leq \frac{M_K^2}{N}\sum_{i=1}^N   \E \left[ \vert \tilde V_t^i-V_t^i \vert^2 \right] \leq  C\delta t +C\int_0^t \E [\Vert \tilde v_{r(s)}-u^{S^N(\tilde \xi)}_s\Vert_\infty^2] \, ds\ .
\end{equation}
We end the proof by injecting this last inequality in~\eqref{eq:tildev1} and by applying Gronwall's lemma.
\end{itemize}
\end{proof}

\noindent\textbf{ACKNOWLEDGEMENTS.}
The third named   author has benefited partially from the
support of the ``FMJH Program Gaspard Monge in optimization and operation
research'' (Project 2014-1607H).

\newpage
\bibliographystyle{plain}
\bibliography{NonConservativePDE}

\def\cprime{$'$}
\begin{thebibliography}{10}

\bibitem{BRR2}
V.~Barbu, M.~R\"ockner, and F.~Russo.
\newblock Probabilistic representation for solutions of an irregular porous
  media type equation: the irregular degenerate case.
\newblock {\em Probab. Theory Related Fields}, 151(1-2):1--43, 2011.

\bibitem{BRR3}
V.~Barbu, M.~R{\"o}ckner, and F.~Russo.
\newblock A stochastic {F}okker-{P}lanck equation and double probabilistic
  representation for the stochastic porous media type equation.
\newblock {\em arXiv preprint arXiv:1404.5120}, 2014.

\bibitem{Barenb}
G.~I. Barenblatt.
\newblock On some unsteady motions of a liquid and gas in a porous medium.
\newblock {\em Akad. Nauk SSSR. Prikl. Mat. Meh.}, 16:67--78, 1952.

\bibitem{BarYo}
M.~T. Barlow and M.~Yor.
\newblock Semimartingale inequalities via the {G}arsia-{R}odemich-{R}umsey
  lemma, and applications to local times.
\newblock {\em J. Funct. Anal.}, 49(2):198--229, 1982.

\bibitem{BCR1}
N.~Belaribi, F.~Cuvelier, and F.~Russo.
\newblock A probabilistic algorithm approximating solutions of a singular {PDE}
  of porous media type.
\newblock {\em Monte Carlo Methods and Applications}, 17(4):317--369, 2011.

\bibitem{BCR3}
N.~Belaribi, F.~Cuvelier, and F.~Russo.
\newblock Probabilistic and deterministic algorithms for space multidimensional
  irregular porous media equation.
\newblock {\em SPDEs: Analysis and Computations}, 1(1):3--62, 2013.

\bibitem{BCR2}
N.~Belaribi and F.~Russo.
\newblock Uniqueness for {F}okker-{P}lanck equations with measurable
  coefficients and applications to the fast diffusion equation.
\newblock {\em Electron. J. Probab.}, 17:no. 84, 28, 2012.

\bibitem{Ben_Vallois}
S.~Benachour, P.~Chassaing, B.~Roynette, and P.~Vallois.
\newblock Processus associ\'es \`a\ l'\'equation des milieux poreux.
\newblock {\em Ann. Scuola Norm. Sup. Pisa Cl. Sci. (4)}, 23(4):793--832, 1996.

\bibitem{BertShre}
D.~P. Bertsekas and S.~E. Shreve.
\newblock {\em Stochastic optimal control}, volume 139 of {\em Mathematics in
  Science and Engineering}.
\newblock Academic Press, Inc. [Harcourt Brace Jovanovich, Publishers], New
  York-London, 1978.
\newblock The discrete time case.

\bibitem{billingsley}
P.~Billingsley.
\newblock {\em Convergence of probability measures}.
\newblock Wiley Series in Probability and Statistics: Probability and
  Statistics. John Wiley \& Sons, Inc., New York, second edition, 1999.
\newblock A Wiley-Interscience Publication.

\bibitem{BRR1}
P.~Blanchard, M.~R{\"o}ckner, and F.~Russo.
\newblock Probabilistic representation for solutions of an irregular porous
  media type equation.
\newblock {\em Ann. Probab.}, 38(5):1870--1900, 2010.

\bibitem{BouchardTouzi}
B.~Bouchard and N.~Touzi.
\newblock Discrete-time approximation and {M}onte {C}arlo simulation of
  backward stochastic differential equations.
\newblock {\em Stochastic Process. Appl.}, 111:175--206, 2004.

\bibitem{cheridito}
P.~Cheridito, H.~M. Soner, N.~Touzi, and N.~Victoir.
\newblock Second-order backward stochastic differential equations and fully
  nonlinear parabolic {PDE}s.
\newblock {\em Comm. Pure Appl. Math.}, 60(7):1081--1110, 2007.

\bibitem{crauel}
H.~Crauel.
\newblock {\em Random probability measures on {P}olish spaces}, volume~11 of
  {\em Stochastics Monographs}.
\newblock Taylor \& Francis, London, 2002.

\bibitem{Dynkin}
E.~B. Dynkin.
\newblock Superprocesses and partial differential equations.
\newblock {\em Ann. Probab.}, 21:1185--1262, 1993.

\bibitem{FlemingSoner}
W.~H. Fleming and H.~M. Soner.
\newblock {\em Controlled {M}arkov processes and viscosity solutions},
  volume~25 of {\em Stochastic Modelling and Applied Probability}.
\newblock Springer, New York, second edition, 2006.

\bibitem{GobetWarin}
E.~Gobet, J-P. Lemor, and X.~Warin.
\newblock A regression-based {M}onte {C}arlo method to solve backward
  stochastic differential equations.
\newblock {\em Ann. Appl. Probab.}, 15(3):2172--2202, 2005.

\bibitem{labordere}
P.~Henry-Labord\`ere.
\newblock Counterparty risk valuation: {A} marked branching diffusion approach.
\newblock {\em Available at SSRN: http://ssrn.com/abstract=1995503 or
  http://dx.doi.org/10.2139/ssrn.1995503}, 2012.

\bibitem{LabordereTouziTan}
P.~Henry-Labord{\`e}re, X.~Tan, and N.~Touzi.
\newblock A numerical algorithm for a class of {BSDE}s via the branching
  process.
\newblock {\em Stochastic Process. Appl.}, 124(2):1112--1140, 2014.

\bibitem{JourMeleard}
B.~Jourdain and S.~M{\'e}l{\'e}ard.
\newblock Propagation of chaos and fluctuations for a moderate model with
  smooth initial data.
\newblock {\em Ann. Inst. H. Poincar\'e Probab. Statist.}, 34(6):727--766,
  1998.

\bibitem{karatshreve}
I.~Karatzas and S.~E. Shreve.
\newblock {\em Brownian motion and stochastic calculus}, volume 113 of {\em
  Graduate Texts in Mathematics}.
\newblock Springer-Verlag, New York, second edition, 1991.

\bibitem{McKean}
H.~P.~Jr. McKean.
\newblock Propagation of chaos for a class of non-linear parabolic equations.
\newblock In {\em Stochastic {D}ifferential {E}quations ({L}ecture {S}eries in
  {D}ifferential {E}quations, {S}ession 7, {C}atholic {U}niv., 1967)}, pages
  41--57. Air Force Office Sci. Res., Arlington, Va., 1967.

\bibitem{pardouxgeilo}
E.~Pardoux.
\newblock Backward stochastic differential equations and viscosity solutions of
  systems of semilinear parabolic and elliptic {PDE}s of second order.
\newblock In {\em Stochastic analysis and related topics, {VI} ({G}eilo,
  1996)}, volume~42 of {\em Progr. Probab.}, pages 79--127. Birkh\"auser
  Boston, Boston, MA, 1998.

\bibitem{pardoux}
{\'E}.~Pardoux and S.~G. Peng.
\newblock Adapted solution of a backward stochastic differential equation.
\newblock {\em Systems Control Lett.}, 14(1):55--61, 1990.

\bibitem{rascanu}
E.~Pardoux and A.~Ra{\c{s}}canu.
\newblock {\em Stochastic differential equations, Backward SDEs, Partial
  differential equations}, volume~69.
\newblock Springer, 2014.

\bibitem{Zhang}
M.~R{\"o}ckner and X.~Zhang.
\newblock Weak uniqueness of {F}okker-{P}lanck equations with degenerate and
  bounded coefficients.
\newblock {\em C. R. Math. Acad. Sci. Paris}, 348(7-8):435--438, 2010.

\bibitem{rogers_v1}
L.~C.~G. Rogers and D.~Williams.
\newblock {\em Diffusions, {M}arkov processes, and martingales. {V}ol. 1}.
\newblock Cambridge Mathematical Library. Cambridge University Press,
  Cambridge, 2000.
\newblock Foundations, Reprint of the second (1994) edition.

\bibitem{rogers_v2}
L.~C.~G. Rogers and D.~Williams.
\newblock {\em Diffusions, {M}arkov processes, and martingales. {V}ol. 2}.
\newblock Cambridge Mathematical Library. Cambridge University Press,
  Cambridge, 2000.
\newblock It{\^o} calculus, Reprint of the second (1994) edition.

\bibitem{SilvBook}
B.~W. Silverman.
\newblock {\em Density estimation for statistics and data analysis}.
\newblock Monographs on Statistics and Applied Probability. Chapman \& Hall,
  London, 1986.

\bibitem{stroock}
D.~W. Stroock and S.~R.~S. Varadhan.
\newblock {\em Multidimensional diffusion processes}.
\newblock Classics in Mathematics. Springer-Verlag, Berlin, 2006.
\newblock Reprint of the 1997 edition.

\bibitem{sznitman}
A-S. Sznitman.
\newblock Topics in propagation of chaos.
\newblock In {\em \'{E}cole d'\'{E}t\'e de {P}robabilit\'es de {S}aint-{F}lour
  {XIX}---1989}, volume 1464 of {\em Lecture Notes in Math.}, pages 165--251.
  Springer, Berlin, 1991.

\bibitem{villani}
C.~Villani.
\newblock {\em Optimal transport}, volume 338 of {\em Grundlehren der
  Mathematischen Wissenschaften [Fundamental Principles of Mathematical
  Sciences]}.
\newblock Springer-Verlag, Berlin, 2009.
\newblock Old and new.

\end{thebibliography}

\end{document}